\documentclass[11pt,reqno]{amsart}
\usepackage{amsfonts}
\usepackage{amsthm, amsfonts, amssymb, color}
 \usepackage{mathrsfs}
 \usepackage{bbm}
 \usepackage{thmtools}
\usepackage{booktabs}
\usepackage{diagbox}
\usepackage{threeparttable}
\usepackage[numbers,sort&compress]{natbib}
\usepackage{caption}
\usepackage{enumerate}
\usepackage{cases}
\usepackage{amsmath}
\usepackage{mathbbol}
\usepackage{footnote}
 \usepackage{amstext, amsxtra}
  \usepackage{txfonts}
   \usepackage{latexsym}
  \usepackage{tikz}
 \usepackage[colorlinks, linkcolor=black, citecolor=blue,   hypertexnames=false]{hyperref}
\usepackage{graphicx}
\usepackage{subfigure}
\usepackage[symbol]{footmisc}
\declaretheorem[numbered=no,
name= Theorem \ref{thmmain} (low energy part)]{euclid}
\declaretheorem[numbered=no,
name= Theorem \ref{thmmain} (high energy part)]{euclid1}

 \allowdisplaybreaks
\def\C {{\mathcal C}}

\def\R {\mathbb{R}}

\def\d{{\rm d}}

\def\i{{\rm i}}

\def \and {{\qquad\text{and}\qquad}}

\DeclareMathOperator{\sgn}{\text{sgn}}

\newtheorem{theorem}{Theorem}[section]
\newtheorem{proposition}[theorem]{Proposition}

\newtheorem{lemma}[theorem]{Lemma}
\newtheorem{definition}[theorem]{Definition}

\newtheorem{remark}[theorem]{Remark}

\numberwithin{equation}{section}
\theoremstyle{definition}

\title[]
{Pointwise estimates for the fundamental solutions of higher order Schr\"{o}dinger equations in odd dimensions \uppercase\expandafter{\romannumeral2}: high  dimensional case}

\author{Han Cheng,\, Shanlin Huang,\,Tianxiao Huang,\,  Quan Zheng}

\address{Han Cheng, Institute of Applied Physics and Computational Mathematics, Beijing 100088, China.}
\email{chmathh@163.com}

\address{Shanlin Huang, School of Mathematics and Statistics, Hubei Key Laboratory of Engineering Modeling and Scientific Computing, Huazhong University of Science and Technology, Wuhan 430074, Hubei, PR China}
\email{shanlin\_huang@hust.edu.cn}

\address{Tianxiao Huang, School of Mathematics (Zhuhai), Sun Yat-sen University, Zhuhai 519082, Guangdong, China }
\email{htx5@mail.sysu.edu.cn}

\address{Quan Zheng, School of Mathematics and Statistics, Huazhong University of Science and Technology, Wuhan 430074, Hubei, China }
\email{qzheng@hust.edu.cn}


\begin{document}

\begin{abstract}
In this paper, for any odd $n$ and any integer $m\geq1$ with $n>4m$, we  study  the fundamental solution of the  higher order Schr\"{o}dinger equation
\begin{equation*}
\i\partial_tu(x, t)=((-\Delta)^m+V(x))u(x, t), \quad t\in \mathbb{R},\,\,x\in \mathbb{R}^n,
\end{equation*}
where $V$ is a real-valued $C^{\frac{n+1}{2}-2m}$ potential with certain  decay.
Let $P_{ac}(H)$ denote the projection onto the absolutely continuous spectrum space of $H=(-\Delta)^m+V$, and assume that $H$ has no positive embedded eigenvalue.
Our main result says that
$e^{-\i tH}P_{ac}(H)$ has integral kernel $K(t,x,y)$ satisfying
\begin{equation*}
|K(t, x,y)|\le C(1+|t|)^{-(\frac{n}{2m}-\sigma)}(1+|t|^{-\frac{n}{2 m}})\left(1+|t|^{-\frac{1}{2 m}}|x-y|\right)^{-\frac{n(m-1)}{2 m-1}},\quad t\neq0,\,x,y\in\mathbb{R}^n,
\end{equation*}
where $\sigma=2$ if $0$ is an eigenvalue of $H$, and $\sigma=0$ otherwise. A similar result for smoothing operators $H^\frac{\alpha}{2m}e^{-\i tH}P_{ac}(H)$ is also given. The regularity condition $V\in C^{\frac{n+1}{2}-2m}$ is optimal in the second order case, and it also seems optimal when $m>1$.
\end{abstract}

\maketitle
\begingroup
\tableofcontents

\section{Introduction}\label{section1}

\subsection{Motivation and main result}\

We continue the study in \cite{CHHZ} for the pointwise estimate for the fundamental solution of higher order Schr\"{o}dinger equation
\begin{equation}\label{equ1.1.1}
	\i\partial_tu(x, t)=((-\Delta)^m+V(x))u(x, t), \quad t\in \mathbb{R},\,\,x\in \mathbb{R}^n,
\end{equation}
where $m$ is any positive integer, and $V$ is a real-valued decaying potential in $\mathbb{R}^n$.

In the first paper \cite{CHHZ} of this series, it has been proved when $n<4m$ and $n$ is odd that, if $V$ has sufficient decay at the infinity, denoted by $P_{ac}(H)$ the projection onto the absolutely continuous spectrum space of $H=(-\Delta)^m+V$, then $e^{-\i tH}P_{ac}(H)$ has integral kernel $K(t,x,y)$ satisfying
\begin{equation}\label{lowdimresult}
	|K(t, x,y)|\le C(1+|t|)^{-h}(1+|t|^{-\frac{n}{2 m}})\left(1+|t|^{-\frac{1}{2 m}}|x-y|\right)^{-\frac{n(m-1)}{2 m-1}},\quad t\neq0,\,x,y\in\mathbb{R}^n,
\end{equation}
where $h>0$ can be specified by $m,n$ and the zero energy resonances of $H$. The motivation of considering such problem was closely related to the  dispersive estimates of Schr\"{o}dinger equations, for instance, we refer to the  classical  works \cite{JSS,Ya,RS} and the survey papers \cite{Sch07,Sch21}. Moreover, it was also inspired by many recent studies on  higher order Schr\"{o}dinger equations with potentials (see e.g. in \cite{FSY,FSWY,EG22,EGG231,EG23,EGT,EGL,EGG23,GY,GT19,GG21,MWY}). The overall strategy to show \eqref{lowdimresult} was to first decompose $e^{-\i tH}P_{ac}(H)$ into low and high energy parts by the spectral theorem
\begin{equation}\label{split}
	\begin{split}
		e^{-\i tH}P_{ac}(H)&=\int_0^{+\infty}e^{-\i t\lambda}\chi(\lambda)\d E_\lambda+\int_0^{+\infty}e^{-\i t\lambda}\tilde{\chi}(\lambda)\d E_\lambda\\
		&:=e^{-\i tH}\chi(H)P_{ac}(H)+e^{-\i tH}\tilde{\chi}(H)P_{ac}(H),
	\end{split}
\end{equation}
where $\d E_\lambda$ is the spectral measure of $H$, $\chi(\lambda)+\tilde{\chi}(\lambda)\equiv1$, and $\chi\in C_0^\infty((-\lambda_0^\frac{1}{2m},\lambda_0^\frac{1}{2m}))$ for sufficiently small $\lambda_0>0$. The difficulty of proving \eqref{lowdimresult} mostly lies in the kernel estimate of the low energy part $e^{-\i tH}\chi(H)P_{ac}(H)$, where the asymptotic expansion of $\d E_\lambda$ as $\lambda\rightarrow0$ was obtained in a detailed form. It turns out that in low odd dimensions $n<4m$, the asymptotic expansion of $\d E_\lambda$ depends on the resonance type of $H$ at zero energy, which finally leads to the dependence $h$ in \eqref{lowdimresult}. We refer to \cite{CHHZ} for the definition of zero energy resonance.

The current paper deals with such problem in high odd dimensions $n>4m$. In contrary to the low dimensional case, we shall see that when $n>4m$, the low energy part $e^{-\i tH}\chi(H)P_{ac}(H)$ is somehow easier to deal with for there is no nontrivial zero energy resonance (i.e. $0$ is either regular or an eigenvalue of $H$). However, the high energy part becomes quite complicated where the regularity of $V$ will be involved. Our main result is the following.

\begin{theorem}\label{thmmain}
	Let $n$ be odd, $m\geq1$ be an integer with $n>4m$, $V\in C^{\frac{n+1}{2}-2m}(\mathbb{R}^n)$ be real-valued satisfying $|V(x)|\lesssim\langle x\rangle^{-(n+2)-}$ and
	\begin{equation}\label{derest}
		|\partial^\alpha V(x)|\lesssim\langle x\rangle^{-(\frac{3n+1}{2}-2m)-},\quad0\leq|\alpha|\leq\mbox{$\frac{n+1}{2}-2m$}.
	\end{equation}
	We also assume that $H=(-\Delta)^m+V$ has no positive embedded eigenvalue. Then $e^{-\i tH}P_{ac}(H)$ has integral kernel $K(t,x,y)$ with
	
	\noindent(\romannumeral1) If the zero energy of $H$ is regular, i.e. $0$ is not an eigenvalue of $H$, then
	\begin{equation}\label{regular}
		|K(t, x,y)|\le C|t|^{-\frac{n}{2 m}}\left(1+|t|^{-\frac{1}{2 m}}|x-y|\right)^{-\frac{n(m-1)}{2 m-1}},\quad t\neq0,\,x,y\in\mathbb{R}^n.
	\end{equation}
\noindent(\romannumeral2) If $0$ is an eigenvalue of $H$, then
\begin{equation}\label{eigen}
	|K(t, x,y)|\le C(1+|t|)^{-(\frac{n}{2m}-2)}(1+|t|^{-\frac{n}{2 m}})\left(1+|t|^{-\frac{1}{2 m}}|x-y|\right)^{-\frac{n(m-1)}{2 m-1}},\quad t\neq0,\,x,y\in\mathbb{R}^n.
\end{equation}
\end{theorem}

We remark that the distinction by the zero energy of $H$ is completely due to the estimate of the low energy part $e^{-\i tH}\chi(H)P_{ac}(H)$ where \eqref{derest} will not be used actually. When the zero energy of $H$ is regular, the estimate \eqref{regular} coincides with the free case $V\equiv0$ (see \cite{Mi}), and it immediately implies the dispersive bound
\begin{equation}\label{highdispersive}
	\|e^{-\i tH}P_{ac}(H)\|_{L^1-L^\infty}\lesssim|t|^{-\frac{n}{2m}},
\end{equation}
but the converse is not true if $m>1$. We also note that when $0$ is an eigenvalue of $H$, the estimate \eqref{eigen} implies the dispersive bound
\begin{equation*}
	\|e^{-\i tH}P_{ac}(H)\|_{L^1-L^\infty}\lesssim(1+|t|)^{-(\frac{n}{2m}-2)}(1+|t|^{-\frac{n}{2 m}}),
\end{equation*}
and this coincides with the result in Goldberg-Green \cite{GG15} for the second order case $m=1$ where better long time decay was also discussed if the zero eigenspace has more vanishing structure.

The main focus of this paper is to show how the regularity of $V$ plays a role in the estimate of the high energy part $e^{-\i tH}\tilde{\chi}(H)P_{ac}(H)$, which leads to the assumptions $V\in C^{\frac{n+1}{2}-2m}(\mathbb{R}^n)$ and \eqref{derest} in Theorem \ref{thmmain}. The fact that in dimensions $n>4m$, the regularity of $V$ may be needed was first found in the second order case $m=1$, i.e. the dispersive estimate of Schr\"{o}dinger equations:
\begin{equation}\label{seconddispersiveest}
	\|e^{-\i tH}P_{ac}(H)\|_{L^1-L^\infty}\lesssim|t|^{-\frac n2},\quad H=-\Delta+V.
\end{equation}
In Journ\'{e}-Soffer-Sogge  \cite{JSS},  the authors  established for the first time the dispersive estimate  \eqref{seconddispersiveest} for $n\geq3$ under the regularity requirement $\hat{V}\in L^1$. Here $\hat{V}$ represents the Fourier transform of $V$. In fact, the regularity of $V$ is necessary when $n>3$, because Goldberg-Visan \cite{GV} has shown that there exists compactly supported real-valued $V\in C^{\frac{n-3}{2}-}(\mathbb{R}^n)$ such that estimate \eqref{seconddispersiveest} fails, while Erdo\v{g}an-Green \cite{EG10} later proved that \eqref{seconddispersiveest} holds with $C^{\frac{n-3}{2}}$ decaying real-valued potentials when $n=5,7$ (assuming zero energy of $H$ is regular). It seems that when $n\geq9$ is odd, the dispersive estimate \eqref{seconddispersiveest} with $C^{\frac{n-3}{2}}$ potential is still not known, but Erdo\v{g}an-Green \cite{EG10} gives a brief technical discussion for the possibility and difficulty of proving such result.

In the higher order case $m>1$, when $n>4m$ and $n$ is odd, to the authors' best knowledge, the only existing works Erdo\v{g}an-Green \cite{EG22,EG23} that are able to obtain the dispersive bound \eqref{highdispersive} have to use the $L^1$ and $L^\infty$ boundedness of the wave operators, where the regularity condition on $V$ were put in the form
\begin{equation}
	\|\mathscr{F}(\langle\cdot\rangle^\sigma V(\cdot))\|_{L^\frac{n-1-\delta}{n-2m-\delta}}\lesssim1,\quad \sigma>\mbox{$\frac{2n-4m}{n-1-\delta}+\delta$},\,0<\delta<<1,
\end{equation}
which roughly means that $V$ is required to have more than $\frac{n}{n-1}(\frac{n+1}{2}-2m)$ reasonable derivatives. We mention that the study of $L^p$ boundedness of the wave operators is technically quite similar to the study of dispersive estimate, see \cite{EGL,GG21,MWY,GY} for example. 

It seems not clear now  what regularity condition should be the sharp one to imply dispersive estimate \eqref{highdispersive}, however when $n>4m$, Erdo\v{g}an-Goldberg-Green \cite{EGG23} has constructed compactly supported $V\in C^\alpha(\mathbb{R}^n)$ for every $\alpha\in[0,\frac{n+1}{2}-2m)$ such that the wave operators for $H=(-\Delta)^m+V$ are not bounded on $L^\infty(\mathbb{R}^n)$, and this may be seen as a circumstantial evidence that indicates the sharpness of our assumption $V\in C^{\frac{n+1}{2}-2m}(\mathbb{R}^n)$ in the even stronger result Theorem \ref{thmmain} (in high odd dimensions) than the dispersive estimate. 

As discussed in \cite{CHHZ}, the immediate applications of Theorem \ref{thmmain} and its proof are the $L^p-L^q$ type estimate and smoothing estimate for $e^{-\i tH}P_{ac}(H)$, and we only phrase them without repeating the proofs.

\begin{proposition}
Under the assumption of Theorem \ref{thmmain}, we have the following.

\noindent\emph{(\romannumeral1)} It follows that
\begin{equation*}
	\|e^{-\i tH}P_{ac}(H)\|_{L^p_*-L^q_*}\lesssim\begin{cases}
		|t|^{-\frac{n}{2m}(\frac{1}{p}-\frac{1}{q})},\quad&\text{if $0$ is not an eigenvalue of $H$},\\
		(1+|t|)^{2(\frac 2p-1)}|t|^{-\frac{n}{2m}(\frac{1}{p}-\frac{1}{q})},\quad&\text{if $0$ is an eigenvalue of $H$},
	\end{cases}
\end{equation*}
where
\begin{equation*}
	L^p_*-L^q_*=
	\begin{cases}
		L^1-L^{\tau_m,\infty}\, \text{or}\, H^1-L^{\tau_m},&\text{if}\,\, (p,q)=(1, \tau_m),\\[4pt]
		L^{\tau'_m,1}-L^\infty\,  \text{or}\, L^{\tau_m'}-\text{BMO},&\text{if}\,\,(p,q)=(\tau'_m,\infty),\\[4pt]
		L^p-L^q,&\text{otherwise},
	\end{cases}
\end{equation*}
$\tau_m=\frac{2m-1}{m-1}$, $\tau_m'=\frac{2m-1}{m}$, $H^1$ is the Hardy space on $\mathbb{R}^n$, BMO is the space of functions with bounded mean oscillation on $\mathbb{R}^n$, and $(\frac 1q,\frac 1q)$ lies in the closed quadrilateral ABCD explained in Figure \ref{fig1}.

\noindent\emph{(\romannumeral2)} If $\alpha\in[0,n(m-1)]$, denoted by $K_{\alpha}(t,x,y)$ the kernel of $H^\frac{\alpha}{2m} e^{-\i tH}P_{ac}(H)$, then
\begin{equation*}
	\begin{split}
		&|K_\alpha(t,x,y)|\\
		\lesssim&\begin{cases}
			|t|^{-\frac{n+\alpha}{2m}}\left(1+|t|^{-\frac{1}{2m}}|x-y|\right)^{-\frac{n(m-1)-\alpha}{2m-1}},\quad&\text{if $0$ is not an eigenvalue of $H$,}\\
			(1+|t|)^{-\frac{1+\alpha}{2m}}(1+|t|^{-\frac{n+\alpha}{2m}})\left(1+|t|^{-\frac{1}{2m}}|x-y|\right)^{-\frac{n(m-1)-\alpha}{2m-1}},\quad&\text{if $0$ is an eigenvalue of $H$.}
		\end{cases}
	\end{split}
\end{equation*}

\begin{figure}[h]
	\centering
	\begin{tikzpicture}[scale=5.0]
		\draw [<->,] (0,1.14) node (yaxis) [right] {$\frac{1}{q}$}
		|- (1.24,0) node (xaxis) [right] {$\frac{1}{p}$};


		\draw (0.80,0) coordinate (a_1)node[below]{\color{blue}{\small{$D=(\frac{1}{\tau_m'},0)$}}};
		\draw[dashed](1,0) coordinate (a_3)node[below=3.5mm] [right=0.1mm] {\small{\color{blue}{C=(1,0)}}} -- (1,1) coordinate (a_4);
		\draw[dashed](1,0) coordinate (a_3) -- (0,1) coordinate (a_5)node[left=0.9mm] {$1$};
		\draw[dashed](1,1) coordinate (a_3) -- (0,1) coordinate (a_5);
		\draw (1,0.2) coordinate (a_5)node[right]{\color{blue}{$B=(1,\frac{1}{\tau_m})$}} -- (0.56,0.444) coordinate (a_2);
		
		\fill[blue]  (0.5,0.5)coordinate (a_6) circle (0.2pt) node [above=3.3mm] [right=0.1mm] {$(\frac12, \frac12)$} node[above=0.3mm]{{A}} ;
		\fill[blue]  (0.80,0) circle (0.2pt);
		\fill[blue] (1,0) circle (0.2pt);
		\fill[blue]  (0.56,0.444) circle (0.2pt);
		\fill[blue]  (1,0.2) circle (0.2pt);
		
		\draw[fill=gray!20]  (0.5,0.5)--(0.80,0) -- (1,0) -- (1,0.2) -- cycle;
		\draw [densely dotted] (0.5,0.5) -- (1,0.2);
		\draw [densely dotted] (0.5,0.5) -- (0.8,0);
		\end{tikzpicture}
		\caption{The  $L^p-L^q$ estimates}\label{fig1}
	\end{figure}
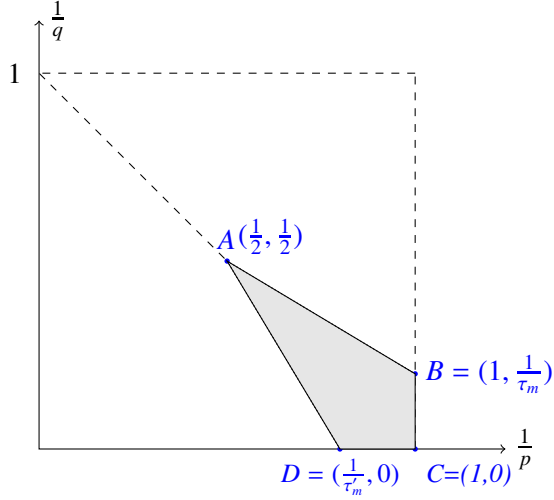

\end{proposition}

\subsection{Technical difficulties in  dimensions $n>4m$}\label{section1.2}\

In Erdo\v{g}an-Green \cite{EG10} for the dispersive estimate with $C^{\frac{n-3}{2}}$ potentials in the case $m=1$ and $n=5,7$, the authors pioneered a delicate integration by parts scheme, where sophisticated singularities are involved. For a comparison, we shortly discuss the counterpart in this paper which develops the ideas in \cite{EG10} for $m\geq1$ and all odd $n>4m$.


We will encounter an integration by parts scheme for some specific oscillatory integrals, where a special case loosely speaking appears in the following form
\begin{equation}\label{toycase}
		\lambda^{k(\frac{n+1}{2}-2m)}\int_{\mathbb{R}^{kn}}e^{\mathrm{i}\lambda (|x_0-x_1|+\cdots+|x_k-x_{k+1}|)}\mbox{$\prod\limits_{i=0}^k|x_i-x_{i+1}|^{-\frac{n-1}{2}}$}\mbox{$\prod\limits_{i=1}^kV(x_i)$}\d x_1\cdots \d x_k,
\end{equation}
where $k$ is any fixed positive integer, $x_0,x_{k+1}\in\mathbb{R}^n$ are fixed parameters, $V$ has certain regularity and decay, and $\lambda^{k(\frac{n+1}{2}-2m)}$ comes from the resolvent expansion where $k$ many free resolvents $R^\pm(\lambda^{2m})=\lim_{\epsilon\downarrow0}((-\Delta)^m-\lambda^{2m}\mp \i0)$ are involved. In the kernel of each resolvent (see \eqref{eq2.8}), there is a term with the highest power in $\lambda$ presented as a finite sum of
\begin{equation*}
	\lambda^{\frac{n+1}{2}-2m}\frac{e^{\i e^{\i\frac{\pi k}{m}}\lambda|x-y|}}{|x-y|^{\frac{n-1}{2}}}
\end{equation*}
which provides the growth $\lambda^{\frac{n+1}{2}-2m}$. For our final purpose of fundamental solution estimate, we need to eliminate the total growth $\lambda^{k(\frac{n+1}{2}-2m)}$, and a natural choice is to integrate by parts $k(\frac{n+1}{2}-2m)$ times in \eqref{toycase}. Averagely, we have to perform integration by parts $\frac{n+1}{2}-2m$ times in each $x_i$ variable, so $V\in C^{\frac{n+1}{2}-2m}$ seems to be a proper condition. 

However, integration by parts may introduce complicated singularities. Denoted by $E_{x_{i-1}x_{i}x_ix_{i+1}}=\frac{x_{i-1}-x_i}{|x_{i-1}-x_i|}-\frac{x_i-x_{i+1}}{|x_i-x_{i+1}|}$, note that
\begin{equation}
	\nabla_{x_i}e^{\mathrm{i}\lambda (|x_0-x_1|+\cdots+|x_k-x_{k+1}|)}=-\i\lambda e^{\mathrm{i}\lambda (|x_0-x_1|+\cdots+|x_k-x_{k+1}|)}E_{x_{i-1}x_{i}x_ix_{i+1}},
\end{equation}
then integration by parts in $x_i$ naturally introduces a type of singularity at the zeros of $E_{x_{i-1}x_{i}x_ix_{i+1}}$, that is when $x_i$ lies in the line segment $[x_{i-1},x_{i+1}]$, and we will call $|E_{x_{i-1}x_{i}x_ix_{i+1}}|^{-1}$ a line singularity. It is not hard to show 
\begin{equation}
	\begin{split}
		\nabla_{x_i}|E_{x_{i-1}x_{i}x_ix_{i+1}}|^{-1}=|E_{x_{i-1}x_{i}x_ix_{i+1}}|^{-2}O\left(|x_{i-1}-x_i|^{-1}+|x_i-x_{i+1}|^{-1}\right),\\
		\nabla_{x_i}|x_{i-1}-x_i|^{-1}=O\left(|x_{i-1}-x_i|^{-2}\right),\quad \nabla_{x_i}|x_{i}-x_{i+1}|^{-1}=O\left(|x_{i}-x_{i+1}|^{-2}\right),
	\end{split}
\end{equation}
then we see integration by parts in $x_i$ each time introduces terms with at most two more line singularities $|E_{x_{i-1}x_{i}x_ix_{i+1}}|^{-2}$, and one more point singularity $|x_{i-1}-x_i|^{-1}$ or $|x_{i}-x_{i+1}|^{-1}$. The more complicated issue is that, if we further integrate by parts in $x_{i+1}$, not only the new line singularity $|E_{x_{i}x_{i+1}x_{i+1}x_{i+2}}|^{-1}$ will be introduced, but differentiation in $x_{i+1}$ will also drop on and increase the previous line singularity $|E_{x_{i-1}x_{i}x_ix_{i+1}}|^{-1}$ and point singularity $|x_{i}-x_{i+1}|^{-1}$. 

Therefore, the way we perform integration by parts leads to estimates of integrals with some mixtures of line and point singularities, and finally, it is important to consider what types of such mixtures are integrable. In Erdo\v{g}an-Green \cite{EG10}, for the case of $m=1$, $n=5$, the problem was reduced to consider integrals with two different types of line singularities in the form
\begin{equation*}
	\int_{\mathbb{R}^n}\frac{\langle z\rangle^{-3-}\d z}{|x-z|^k|x-y|^l|E_{xzzw}|^{n-3}|E_{zwwy}|^{n-3}},
\end{equation*}
and in the case of $m=1$, $n=7$, they have to finally consider integrals with three different types of line singularities in the form
\begin{equation*}
	\int_{\mathbb{R}^n}\int_{\mathbb{R}^n}\frac{\langle z\rangle^{-3-}\langle w\rangle^{-3-}\d z\d w}{|x-z|^k|z-w|^l|w-y|^p|E_{xzzw}|^{n-3}|E_{zwwy}|^{n-3}|E_{zwyu}|^{n-3}}.
\end{equation*}
It was also expected in \cite{EG10} that in the general case of $m=1$, odd $n\geq9$, the similar consideration will lead to estimates for integrals which involve many different types of line singularities, but the analysis of such integrals should be quite different and more complicated.

In this paper for the much more general case $m\geq1$ and odd $n>4m$, what meets the expectation in \cite{EG10} is that, we do have to consider integrals with many different types of line singularities in the form of
\begin{equation}\label{manyline}
	\int_{\mathbb{R}^{kn}}\frac{\langle x_1\rangle^{-\beta_1-}\cdots\langle x_k\rangle^{-\beta_k-}}{|x_0-x_1|^{a_0}\cdots|x_k-x_{k+1}|^{a_k}}|E_{\eta_1,1}|^{-q_1}\cdots|E_{\eta_k,k}|^{-q_k}\d x_1\cdots \d x_k,
\end{equation}
where $E_{\eta_i,i}=E_{x_{\eta_i}x_{\eta_i-1}x_ix_{i+1}}$. Somewhat unexpected in \cite{EG10} however, we will reduce the estimate of \eqref{manyline} to estimating integrals with only two different types of line singularities in the form
\begin{equation}\label{two}
		\int_{\mathbb{R}^n}\frac{\langle|x-y|^{-k_1}\rangle\langle|y-z|^{-l_1}\rangle\langle y-y_0\rangle^{-\beta-}}{\langle x-y\rangle^{k_2}\langle y-z\rangle^{l_2}|E_{xyyz}|^p|E_{ww'xy}|^q}\d y,
\end{equation}
and this is due to the crucial fact that our strategy of performing integration by parts will give a specific structure of the indices $\{(\eta_j,j);\,j\in\mathbb{K}\}$ in \eqref{manyline}, which we will call "admissible" in Definition \ref{defadmissible}, and the admissibility of $\{(\eta_j,j);\,j\in\mathbb{K}\}$ allows us to always find a specific variable $x_\tau$, such that a successful estimate of the integral in $x_\tau$ by \eqref{two} will preserve the form of \eqref{manyline}.

\subsection{Plan of the paper}\

We  outline the strategy for proving Theorem \ref{thmmain}.
Our starting point is the Stone's formula
\begin{equation}\label{equ0.2}
	\langle e^{-\i tH}P_{ac}(H)f,\, g\rangle=\frac{1}{2\pi \i}\int_0^{\infty}e^{-\i t\lambda}\langle(R^+(\lambda)-R^-(\lambda))f,g\rangle \d\lambda,\quad f, g\in C_0^{\infty}(\mathbb{R}^n),
\end{equation}
where $R^\pm(\lambda):=(H-\lambda\mp \i0)^{-1}$. The assumption that $H$ has no positive eigenvalue allows us to split
\begin{align*}
 e^{-\i tH}P_{ac}(H)&=\frac{1}{2\pi i} \int_{0}^{+\infty} e^{-\i  t \lambda}\left(R^{+}(\lambda)-R^{-}(\lambda)\right) (\chi(\lambda)+\tilde{\chi}(\lambda))\, \d \lambda\\
  & :=e^{-\i tH}\chi(H)P_{ac}(H)+e^{-\i tH}\tilde{\chi}(H)P_{ac}(H),
\end{align*}
into low and high energy part, where $\chi(\lambda)+\tilde{\chi}(\lambda)\equiv1$, $\chi\in C_0^\infty((-\lambda_0^\frac{1}{2m},\lambda_0^\frac{1}{2m}))$ some $\lambda_0>0$, and $\chi(\lambda)=1$ when $\lambda\in((-(\frac{\lambda_0}{2})^\frac{1}{2m},(\frac{\lambda_0}{2})^\frac{1}{2m}))$. Theorem \ref{thmmain} follows immediately from the two theorems below.
\begin{euclid}
	Under the assumption of Theorem \ref{thmmain}, if $\lambda_{0}\in (0, 1)$ is small enough, then  $e^{-\i tH}\chi(H)P_{ac}(H)$ has integral kernel  $K_{L}(t,x,y)$ satisfying
\begin{equation*}\label{equ4.1.2.0}
		\left|K_{L}(t, x, y)\right| \lesssim (1+|t|)^{-(\frac{n}{2m}-\sigma)}(1+|t|^{-\frac{n}{2 m}})\left(1+|t|^{-\frac{1}{2 m}}|x-y|\right)^{-\frac{n(m-1)}{2 m-1}},\quad t\neq0,\,\,x,y\in\mathbb{R}^n,
	\end{equation*}
where $\sigma=2$ if $0$ is an eigenvalue of $H$, and $\sigma=0$ otherwise.
\end{euclid}

\begin{euclid1}\label{theorem4.1}
	Under the assumption of Theorem \ref{thmmain}, $e^{-\i tH}\tilde{\chi}(H)P_{ac}(H)$ (for any fixed $\lambda_0>0$)  has integral kernel $K_{H}(t,x,y)$ satisfying
\begin{equation*}\label{equ4.1.2.11}
		\left|K_{H}(t, x, y)\right| \lesssim |t|^{-\frac{n}{2 m}}\left(1+|t|^{-\frac{1}{2 m}}|x-y|\right)^{-\frac{n(m-1)}{2 m-1}},\quad t\neq0,\,\,x,y\in\mathbb{R}^n.
	\end{equation*}
\end{euclid1}

In Section \ref{section2}, we give some technical preliminaries, including the kernels of the boundary values of the free resolvents $R_0^\pm(\lambda^{2m})=((-\Delta)^m-\lambda^{2m}\mp\i 0)^{-1}$, the expansions of the perturbed resolvents $R^\pm(\lambda^{2m})=(H-\lambda^{2m}\mp\i 0)^{-1}$, estimates for integrals with point and line singularities like \eqref{two}, and two  oscillatory integral estimates in one dimension.

In Section \ref{section3}, we prove the low energy part of Theorem \ref{thmmain}. We first decompose 
$$
e^{-\i tH}\chi( H)P_{ac}(H)=\sum_{k=0}^{2N}\Omega^{low}_{k}-\left(\Omega^{+, low}_{ r}-\Omega^{-, low}_{r}\right)
$$
via spectral theorem and an iteration of the resolvent identity. The estimate for $\Omega^{low}_{k}(t,x,y)$ is given in Section \ref{sec4.2.3}. The estimate for the remainder term
$(\Omega^{+, low}_{ r}-\Omega^{-, low}_{r})(t,x,y)$ is given in Section \ref{sec4.2.4}. We note that, in contrast to the situation when  $n<4m$, the
asymptotic expansions of $(U+vR_0^\pm(\lambda^{2m})v)^{-1}$ ($v(x)=|V(x)|^{\frac12}, \,\,U=\sgn V$) for small $\lambda>0$ are much simpler in the case $n>4m$, which has been treated in \cite{FSWY} (see Theorem \ref{thm3.4}).

In Section \ref{section4}, we prove the high energy part of Theorem \ref{thmmain} except for a particularly complicated part. Similar to the low energy part, we first decompose the high energy part $e^{-\i tH}\tilde{\chi}(H)P_{ac}(H)$, by a slightly different resolvent identity of $R^\pm(\lambda)$, into
\begin{equation*}
	e^{-\mathrm{i}tH}P_{ac}(H)\tilde{\chi}(H)=\sum_{k=0}^{2K-1}\Omega_k^{high}+\Omega_{K,r}^{+,high}-\Omega_{K,r}^{-,high}.
\end{equation*}
The estimates for $\Omega_{K,r}^{\pm,high}(t,x,y)$ can be proved if $K$ is chosen sufficiently large as already shown in \cite{CHHZ}, so we omit the proof. $\Omega_k^{high}(t,x,y)$ can be written as the oscillatory integral
\begin{equation*} \int_0^{+\infty}e^{\mathrm{i}t\lambda}\tilde{\chi}(\lambda)\left(\int_{\mathbb{R}^{kn}}\left(\mbox{$\prod\limits_{i=0}^kR_0^+(\lambda)(x_i-x_{i+1})-\prod\limits_{i=0}^kR_0^-(\lambda)(x_i-x_{i+1})$}\right)\mbox{$\prod\limits_{i=1}^kV(x_i)$}\d x_1\cdots\d x_k\right)\d\lambda,
\end{equation*}
where $R_0^\pm(\lambda)(x_i-x_{i+1})$ are the kernels of $R_0^\pm$($\lambda$), and we will further decompose $\Omega_k^{high}(t,x,y)$ into $\Omega_k^{high,1}(t,x,y)$ and $\Omega_k^{high,2}(t,x,y)$ by the space-time regions $\{X\leq\delta T\}$ and $\{X\geq\delta T\}$ respectively for some small $\delta>0$, where $X=|x-x_1|+|x_1-x_2|+\cdots+|x_k-y|$, $T=|t|^\frac{1}{2m}+|t|$. The estimate of $\Omega_k^{high,1}(t,x,y)$ is given in Section \ref{section4.2}, but the estimates of $\Omega_k^{high,2}(t,x,y)$ is only sketched in Section \ref{section4.3}, leaving the technical details for a particularly complicated part in Section \ref{section5}.

In Section \ref{section5}, we complete the estimate for $\Omega_k^{high,2}(t,x,y)$, where the techniques for solving the issues indicated in Section \ref{section1.2} are fully demonstrated. We first heuristically discuss in Section \ref{section5.1} the overall plan of treating a particular type of terms $I^{\vec{l}}(t;x,y)$ in the expansion of $\Omega_k^{high,2}(t,x,y)$ where the regularity of $V$ is needed. In Section \ref{section5.2}, we establish an integration by parts scheme for integrals in a more general form than \eqref{toycase}, and the main conclusion is Proposition \ref{stepmu}, where after integration by parts, a bunch of clusters of line singularities $\prod_{i=1}^{s}\|F_i\|^{-{p_i}}$ shows up in the integrand. Here $F_i$ is a particular set of line singularities $\{E_{x_{i_1}x_{i_1+1}x_{j_1}x_{j_1+1}},\cdots,E_{x_{i_r}x_{i_r+1}x_{j_r}x_{j_r+1}}\}$, and
\begin{equation*}
	\|F_i\|=\left(\sum_{l=1}^r|E_{x_{i_l}x_{i_l+1}x_{j_l}x_{j_l+1}}|^2\right)^\frac12.
\end{equation*}
In Section \ref{section5.3}, with a specific structure of these $F_i$ that comes along with our integration by parts scheme, we turn to reduce the family of $\|F_i\|$, so that what have to be estimated will be reduced to integrals in the form of \eqref{manyline} in Section \ref{section5.4}, where we can take the advantage of the admissibility of $\{(\eta_j,j);\,j\in\mathbb{K}\}$ mentioned in Section \ref{section1.2} to successfully complete all relevant estimates by only applying estimate for integrals like \eqref{two} with two different types of line singularities, and the main conclusion is Proposition \ref{intmainest}. With all the technical preparation above, we finally complete the estimate for $\Omega_k^{high,2}(t,x,y)$ in Section \ref{section5.5}.

\subsection{Notations}\

We first collect some common notations and conventions.
Throughout the paper, $\mathbb{N}_+=\{1,2,\ldots\}$, $\mathbb{N}_0=\{0,1,2,\ldots\}$, and $\mathbb{Z}=\{0,\pm1, \pm2,\ldots\}$, $L^2=L^2(\R^n;\mathbb{C})$. $[l]$ denotes the greatest integer at most $l$. $A\lesssim B$ means $A \leq  CB$, where $C>0$ is an absolute constant whose dependence will be specified whenever necessary, and the value of C may vary from line to line. $a-$ (resp. $a+$) means $a-\epsilon$ (resp. $a+\epsilon$) for some $\epsilon>0$.

Next, we define some (vector-valued) function classes that will be used in most of the oscillatory integral estimates in this paper.
\begin{definition}
For $b\in\mathbb{R}$,  $K\in\mathbb{N}_0$, and an open set $\Omega\subset\mathbb{R}$.

\noindent\emph{(\romannumeral1)} We say $f\in S_K^b(\Omega)$ if $f(\lambda)\in C^K(\Omega)$ and
\begin{equation}\label{eq1}
|\partial_\lambda^j f(\lambda)|\lesssim|\lambda|^{b-j},\quad\lambda\in\Omega,\,\,0\leq j\leq K.
\end{equation}
We also denote
 \begin{equation}\label{eq-Skb-infit}
 S^b(\Omega)=\bigcap_{K=0}^\infty S_K^b(\Omega).
 \end{equation}

\noindent\emph{(\romannumeral2)}
For functions with parameters in the following form, we say
$f\left(\lambda, x, y, s_{1}, s_{2}\right)\in  S_{K}^{b}(\Omega)$, if $\lambda\mapsto f\left(\lambda, x, y, s_{1}, s_{2}\right)$ is in $C^K(\Omega)$ for fixed $x, y, s_{1}, s_{2}$, and
\begin{equation}\label{eq-Skb-2}
  |\partial_{\lambda}^{\gamma} f\left(\lambda, x, y, s_{1}, s_{2}\right)|\lesssim\lambda^{b-\gamma},\quad\lambda\in\Omega,\,\,\,0\leq j\leq K,
\end{equation}
holds uniformly with respect to the parameters  $x, y, s_{1}, s_{2}$.

\noindent\emph{(\romannumeral3)}
We say $g(\lambda, s,\cdot, x)\in S_K^b(\Omega,\,\, \|\cdot \|_{L^{2}})$, if
$\lambda\mapsto g(\lambda, s, y, x)$ is in $C^{K}(\Omega)$ for any fixed $s, x, y$, and
   \begin{equation}\label{eq1.111}
\|\partial_{\lambda}^j g(\lambda, s,\cdot, x)\|_{L^{2}}\leq C_j|\lambda|^{b-j},\quad\lambda\in\Omega,\,\,\,0\leq j\leq K
   \end{equation}
and $C_j>0$  does not depend on the parameters  $x, s$.

\noindent\emph{(\romannumeral4)}
We say $T(\lambda)\in \mathfrak{S}_K^b(\Omega)$ if $\{T(\lambda)\}_{\lambda\in\Omega}$ is a family of bounded operators in $L^2$ such that
\begin{equation}\label{eq2}
|\partial_\lambda^j \langle T(\lambda)f, g\rangle|\leq C_j\|f\|_{L^2}\|g\|_{L^2}|\lambda|^{b-j},\quad\lambda\in\Omega,\,\,\,0\leq j\leq K
\end{equation}
holds for all $f, g\in L^2$, and the constant $C_j$ is independent of $f, g, \lambda$.
\end{definition}

We mention in particular that in Section \ref{section5}, there are some frequently used notations which are not standard, and we make an index here for readers who go into those details.
\begin{itemize}
	\item  $E_{xyyz}$, $E_{w'wxy}$ (See \eqref{linesing}),  $E_{i,j}$ (see \eqref{Eij}), $\|F\|$ (see \eqref{defFnorm});\vspace{10pt}
	
	\item  $L_i$ (see \eqref{defL_i});\vspace{10pt}
	
	\item $N(A,i),\,L(A,i)$ (see \eqref{NL}), $D_iA$ (see \eqref{DiA}), $D_IA$ (see \eqref{DIA});\vspace{10pt}
	
	\item $\iota_j$ (see \eqref{defh_j}).

\end{itemize}

\section{Preliminaries}\label{section2}
In this section, we gather some technical preliminaries that shall be used through the paper.

\subsection{The free resolvents}\label{section2.1}\

For $z\in \mathbb{C}\setminus[0,\infty)$, we set $R_0(z)=((-\Delta)^m-z)^{-1}$. For $\lambda>0$, the well known limiting absorption principle (see \cite{AH}) implies that the weak* limits $R_0^{\pm}(\lambda^{2m}):= R_0(\lambda^{2m}\pm i0)=w*-\lim_{\epsilon\downarrow0}R_0(\lambda^{2m}\pm\i\epsilon)$ exist as bounded operators between certain weighted $L^2$ spaces. When $n\geq 4m$ is odd, it was shown in \cite{CHHZ} that if $\lambda>0$, the convolution kernels of $R_0^\pm(\lambda^{2m})$ satisfy

\begin{equation}\label{eq2.8}
R_0^{\pm}(\lambda^{2m})(x)=\frac{1}{(4\pi)^{\frac{n-1}{2}}m\lambda^{2m}}\sum_{k\in I^{\pm}}\frac{\lambda_k^2e^{ \i\lambda_k|x|}}{|x|^{n-2}}\sum^{\frac{n-3}{2}}_{j=0}{c_j(\i\lambda_k|x|)^j},\quad x\in\mathbb{R}^n,\,\mbox{$c_j=\frac{(-2)^j(n-3-j)!}{j!(\frac{n-3}{2}-j)!}$},
\end{equation}
where $\lambda_k=\lambda e^{\i\frac{\pi k}{m}}$, $I^+=\{0,1,\cdots,m-1\}$ and $I^-=\{1,\cdots,m\}$. We will also need the following facts which are special cases of \cite[Proposition 2.1]{CHHZ}:
%
%
%
\begin{equation}\label{eq2.10}
	R_{0}^{+}(\lambda^{2m})(x)-R_{0}^{-}(\lambda^{2m})(x)=\sum_{k\in\{0,m\}}e^{\frac{\i k\pi}{m}}\left( \sum_{j=0}^{\frac{n-3}{2}}C_{j,0}\lambda_k^{n-2m}\int_{0}^{1}e^{\i\lambda_{k}s|x|}(1-s)^{n-3-j}\d s\right),
\end{equation}
and
\begin{equation}\label{equ4.17}
	\begin{split}
		R_{0}^{\pm}(\lambda^{2 m})(x)=&\sum_{k\in I^{\pm}}|x|^{2m-n} \left(\sum_{l=0}^{2m-3} C_{l} \int_{0}^{1} e^{\i s \lambda_{k}|x|}(1-s)^{2m-l-3}\d s\right) \\
		&+\sum_{k\in I^{\pm}}\sum_{l=2m-2}^{\frac{n-3}{2}}D_{l}\lambda_k^{l+2-2m}|x|^{l+2-n}e^{\i\lambda_{k}|x|},
	\end{split}
\end{equation}
hold for some constants $C_{j,0},C_l,D_l$.

\subsection{The perturbed resolvent expansions around zero energy}\label{section2.2}\

In order to study the spectral measure of $H$ near zero energy, we set $R^\pm(\lambda^{2m}):=(H-\lambda^{2m}\mp \i0)^{-1}$, which is well defined under our assumptions on $V$ (see \cite{AH}). Denote
\begin{equation*}\label{equ3.1}
	M^{\pm}(\lambda)=U+vR_0^\pm(\lambda^{2m})v,\quad v(x)=|V(x)|^{\frac12}, \,\,U=\sgn V,
\end{equation*}
where $\sgn x=1$ when $x\ge 0$ and $\sgn x=-1$ when $x< 0$. If $(M^\pm(\lambda))^{-1}$ exist in $L^2$, one checks the following symmetric resolvent identities
\begin{equation}\label{equ3.2}
	R_V^{\pm}(\lambda^{2m})=R_0^\pm(\lambda^{2m})-R_0^\pm(\lambda^{2m})v(M^\pm(\lambda))^{-1}vR_0^\pm(\lambda^{2m}).
\end{equation}
Set
\begin{equation*}\label{equ3.5}
	T_0=U+b_0vG_{2m-n}v,\quad b_0=(4\pi)^{-\frac{n-1}{2}}\sum_{j=0}^{2m-2}\frac{c_j}{(2m-2-j)!},
\end{equation*}
where $c_j$ is defined in \eqref{eq2.8} and
\begin{equation*}\label{equ3.3}
	(G_{2m-n}f)(x)=\int_{\R^n}|x-y|^{2m-n}f(y)\,\d y.
\end{equation*}
Denoted by $S_{m-\frac{n}{2}}$ the orthogonal projection from $L^2$ to $\mbox{ker}\,T_0$, the following expansions have already been proved in \cite{FSWY}.

\begin{theorem}\label{thm3.4}
	Under the assumption of Theorem \ref{thmmain}, there exists some $\lambda_0\in(0,1)$ such that $M^\pm(\lambda)$ have bounded inverses in $L^2$ for all $0<\lambda<\lambda_0$, and the following statements hold.
	
	\noindent(\romannumeral1) If $0$ is not an eigenvalue of $H$, we have
	\begin{equation*}\label{equ3.48}
		(M^\pm(\lambda))^{-1}=
		D_0+\sum\limits_{j=1 }^{[\frac{n}{2m}]-1}\lambda^{2mj}B_j+\Gamma_{0}^\pm(\lambda),\quad0<\lambda<\lambda_0,
	\end{equation*}
	where $D_0$, $B_j$ and $\Gamma_{0}^\pm(\lambda)$ are bounded operators in $L^2$, and
	\begin{equation*}\label{equ3.47.1-000}
		\Gamma_{0}^\pm(\lambda)\in \mathfrak{S}^{n-2m}_{\frac{n+1}{2}}\big((0,\lambda_0)\big).
	\end{equation*}
	\noindent(\romannumeral2)  If $0$ is an eigenvalue of $H$, we have
	\begin{equation*}\label{equ3.48-222}
		(M^\pm(\lambda))^{-1}=
		\lambda^{-2m}S_{m-\frac{n}{2}}D_1S_{m-\frac{n}{2}}+Q_{m-\frac{n}{2}}D_0Q_{m-\frac{n}{2}}+\sum\limits_{j=1 }^{[\frac{n}{2m}]-1}\lambda^{2m(j-1)}B_j+\Gamma_{1}^\pm(\lambda),\quad0<\lambda<\lambda_0,
	\end{equation*}
	where $D_1$, $B_j$ and $\Gamma_{1}^\pm(\lambda)$ are bounded operators in $L^2$,
	\begin{equation*}\label{equ3.47.1}
		\Gamma_{ 2m-\frac{n}{2},  2m-\frac{n}{2}}^\pm(\lambda)\in \mathfrak{S}^{n-6m}_{\frac{n+1}{2}}\big((0,\lambda_0)\big),
	\end{equation*}
	and $Q_{m-\frac{n}{2}}=I-S_{m-\frac{n}{2}}$.
\end{theorem}

\subsection{Integrals with point and line singularities}\

We now introduce some estimates for integrals with point and at most two different types of line singularities, which will be frequently used especially in the high energy part estimate. We note that these results are valid for all dimensions $n$ but not only the odd ones.

The following lemma can be seen in \cite[Lemma 3.8]{GV} or \cite[Lemma 6.3]{EG10}.
\begin{lemma}\label{lem3.10}
	Let $n\ge 1$. Then there is some absolute constant $C>0$ such that
	\begin{equation*}\label{eq2.20}
		\int_{\mathbb{R}^n}|x-y|^{-k}\langle y\rangle^{-l}\,\d y\leq C\langle x\rangle^{-\min\{k,\, k+l-n\}},
	\end{equation*}
	provided  $l\ge 0$, $0\le k<n$ and $k+l>n$.
\end{lemma}

The following lemma was already introduced in \cite{CHHZ}, whose proof is almost the same to that of \cite[Lemma 6.3]{EG10}.

\begin{lemma}\label{lmEd}
	Suppose $n\geq1$, $k_1,l_1\in[0,n)$, $k_2,l_2\in[0,+\infty)$, $\beta\in(0,+\infty)$, and $k_2+l_2+\beta\geq n$. It follows uniformly in $y_0\in\mathbb{R}^n$ that
	\begin{equation*}
		\begin{split}
			&\int_{\mathbb{R}^n}\frac{\langle|x-y|^{-k_1}\rangle\langle|y-z|^{-l_1}\rangle\langle y-y_0\rangle^{-\beta-}}{\langle x-y\rangle^{k_2}\langle y-z\rangle^{l_2}}\d y\\
			\lesssim&\begin{cases}
				\langle|x-z|^{-\max\{0,k_1+l_1-n\}}\rangle\langle x-z\rangle^{-\min\{k_2,l_2,k_2+l_2+\beta-n\}},\quad&k_1+l_1\neq n,\\
				\langle|x-z|^{0-}\rangle\langle x-z\rangle^{-\min\{k_2,l_2,k_2+l_2+\beta-n\}},&k_1+l_1=n.
			\end{cases}
		\end{split}
	\end{equation*}
\end{lemma}

We now turn to integrals with both point singularities line singularities when $n\geq2$. Given separated $w,w',x,z\in\mathbb{R}^n$, we define quantities
\begin{equation}\label{linesing}
	E_{xyyz}=\mbox{$\frac{x-y}{|x-y|}-\frac{y-z}{|y-z|}$},\quad E_{ww'xy}=\mbox{$\frac{w-w'}{|w-w'|}-\frac{x-y}{|x-y|}$}.
\end{equation}

\begin{proposition}\label{twoline}
	Suppose $n\geq2$, $k_1,l_1\in[0,n)$, $k_2,l_2\in[0,+\infty)$, $\beta\in(0,+\infty)$, $k_2+l_2+\beta\geq n$, and $p,q\in[0,n-1)$. It follows uniformly in $y_0\in\mathbb{R}^n$ that
	\begin{equation*}\label{l15}
		\begin{split}
			&\int_{\mathbb{R}^n}\frac{\langle|x-y|^{-k_1}\rangle\langle|y-z|^{-l_1}\rangle\langle y-y_0\rangle^{-\beta-}}{\langle x-y\rangle^{k_2}\langle y-z\rangle^{l_2}|E_{xyyz}|^p|E_{ww'xy}|^q}\d y\\
			\lesssim&|E_{ww'xz}|^{-q}\begin{cases}
				\langle|x-z|^{-\max\{0,k_1+l_1-n\}}\rangle\langle x-z\rangle^{-\min\{k_2,l_2,k_2+l_2+\beta-n,k_2+l_2-\max\{p,q\}\}},\,&k_1+l_1\neq n,\\
				\langle|x-z|^{0-}\rangle\langle x-z\rangle^{-\min\{k_2,l_2,k_2+l_2+\beta-n,k_2+l_2-\max\{p,q\}\}},&k_1+l_1=n.
			\end{cases}
		\end{split}
	\end{equation*}
\end{proposition}

The tedious proof of Proposition \ref{twoline} will be given in Appendix \ref{app-001}, which is based on first proving the special case $q=0$ that will also be frequently used later in Section \ref{section5.4}.

\subsection{Two oscillatory integral estimates}\

Given $\lambda_0>0$, we will need estimates of one dimensional oscillatory integral in the form of
\begin{equation*}
	I(t,x)=\int_0^{+\infty}e^{-\i t\lambda^{2m}+\i x\lambda}\psi(\lambda)\d\lambda,\quad t\neq0,\,x\in\mathbb{R}.
\end{equation*}
where $\psi$ is supported in $[0,\lambda_0]$ or $[\frac{\lambda_0}{2},+\infty)$. Denoted by
\begin{equation}\label{mu}
	\mbox{$\mu_b=\frac{m-1-b}{2m-1}$},
\end{equation}
the following two estimates on one dimensional oscillatory integrals can be found in \cite{HHZ}.

\begin{lemma}\label{lemmaA.2}
	Suppose $N\in\mathbb{N}_+$, $b\in\mathbb{R}$ and $\psi\in S_N^b((0,\lambda_0))$ supported in $[0,\lambda_0]$.
	
	\noindent (i) If $b\in[-\frac{1}{2},\, 2Nm-1)$, then
	\begin{equation}\label{eqA.4}
		|I(t,x)|\lesssim
		|t|^{-\frac{1+b}{2m}}(|t|^{-\frac{1}{2m}}|x|)^{-\mu_b},\quad |t|^{-\frac{1}{2m}}|x|\gtrsim1.
	\end{equation}
	\noindent (ii) If $b\in(-1,\, 2Nm-1)$, then
	\begin{equation}\label{eqA.5}
		|I(t,x)|\lesssim(1+|t|^{\frac{1}{2m}})^{-(1+b)},\quad 0<|t|^{-\frac{1}{2m}}|x|\lesssim1.
	\end{equation}
\end{lemma}

\begin{lemma}\label{lemmaA.222}
	Suppose $N\in\mathbb{N}_+$, $b\in\mathbb{R}$ and $\psi\in S_N^b((\frac{\lambda_{0}}{2},+\infty))$ supported in $[\frac{\lambda_0}{2},+\infty)$.	
	
	\noindent (i) If $b\in[m-1-N(2m-1),2Nm-1)$, then
	\begin{equation}\label{longime}
		|I(t,x)|\lesssim\begin{cases}
			|t|^{-\frac 12+\mu_{b}}|x|^{-\mu_{b}},\quad&|t|\gtrsim1,\,|t|^{-1}|x|\gtrsim1,\\
			|t|^{-N},&|t|\gtrsim1,\,|t|^{-1}|x|<<1.
		\end{cases}
	\end{equation}
	
	\noindent (ii) If $b\in[-\frac 12,2Nm-1)$, then
	\begin{equation}\label{shorttime}
		|I(t,x)|\lesssim|t|^{-\frac{1+b}{2m}}\left(1+|t|^{-\frac{1}{2m}}|x|\right)^{-\mu_{b}},\quad0<|t|\lesssim1,\,x\in\mathbb{R}.
	\end{equation}
	
	The constants in estimates \eqref{eqA.4}, \eqref{eqA.5}, \eqref{longime} and \eqref{shorttime} stay bounded when the seminorms of $\psi$ in $S_N^b$ stay bounded.
\end{lemma}

\section{The proof of Theorem \ref{thmmain} (low energy part)}\label{section3}
By an iteration  of the resolvent identity and \eqref{equ3.2}, we have
\begin{align}\label{eq4.54}\small
	&e^{-\i tH}\chi( H)P_{ac}(H)\nonumber\\
	=&
	\sum_{k=0}^{2N}\frac{(-1)^{k}m}{\pi \i}\, \int_{0}^{+\infty}e^{-\i t\lambda^{2m}}\left(R_{0}^{+}(\lambda^{2m})(VR_{0}^{+}(\lambda^{2m}))^{k}-R_{0}^{-}(\lambda^{2m})(VR_{0}^{-}(\lambda^{2m}))^{k}\right)\lambda^{2m-1}\chi(\lambda^{2m})\d\lambda\nonumber\\
	&-\frac{m}{\pi \i}\, \int_{0}^{\infty}e^{-\i t\lambda^{2m}}(R_{0}^{+}(\lambda^{2m})V)^{N}R_{0}^{+}vM^{+}(\lambda)^{-1}vR_{0}^{+}(\lambda^{2m})(VR_{0}^{+}(\lambda^{2m}))^{N}\lambda^{2m-1}\chi(\lambda^{2m})\d\lambda\nonumber\\	
	&+\frac{m}{\pi \i}\, \int_{0}^{\infty}e^{-\i t\lambda^{2m}}(R_{0}^{-}(\lambda^{2m})V)^{N}R_{0}^{-}vM^{-}(\lambda)^{-1}vR_{0}^{-}(\lambda^{2m})(VR_{0}^{-}(\lambda^{2m}))^{N}\lambda^{2m-1}\chi(\lambda^{2m})\d\lambda\nonumber\\
	:=& \sum_{k=0}^{2N}\Omega^{low}_{k}-\left(\Omega^{+, low}_{ r}-\Omega^{-, low}_{r}\right),
\end{align}
where $N\in \mathbb{N}_0$ is fixed and will be chosen later.
Since
$\Omega^{low}_{0}=e^{-\i tH_0}\chi(H_0)$,
by \cite{Mi,HHZ}, we have
\begin{equation} \label{eq4.1.1}
	\left|\Omega^{ low}_{0}(t, x, y)\right|\lesssim |t|^{-\frac{n}{2 m}} \left(1+|t|^{-\frac{1}{2 m}}|x-y|\right)^{-\frac{n(m-1)}{2 m-1}},\quad t\neq0,\,\,x,y\in\mathbb{R}^n.
\end{equation}


We next estimate the kernels of the initial terms $\Omega^{low}_{k}$ for $1\le k\le 2N$, and then prove estimate for the kernel of the remainder term $\Omega^{+, low}_{ r}-\Omega^{-, low}_{r}$. Theorem \ref{thmmain} (low energy part) shall follow immediately from \eqref{eq4.1.1}, and \eqref{eq4.57}, \eqref{equ-ker-remai} below.

\subsection{Estimate for $\Omega^{low}_{k}(t,x,y)$}\label{sec4.2.3}\

%
For any $k\in \mathbb{N}_+$, we shall prove in this part that if  $|V(x)|\lesssim\langle x\rangle^{-(n+2)-}$, then the kernel of $\Omega^{low}_{k}$ satisfies
\begin{equation}\label{eq4.57}
	\left|\Omega^{low}_{k}(t,x,y)\right| \lesssim |t|^{-\frac{n}{2m}}\left(1+|t|^{-\frac{1}{2m}}|x-y|\right)^{-\frac{n(m-1)}{2m-1}}.
\end{equation}
Note that the kernel $\Omega^{low}_{k}(t,x,y)$ can be expressed by
 \begin{align*}\label{eq4.58}
 	&\frac{(-1)^{k}m}{\pi \i}\int_{0}^{+\infty}\int_{\R^{nk}}e^{-\i t\lambda^{2m}}\left(R_{0}^{+}(\lambda^{2m})(z_{0}-z_{1})\prod_{j=1}^{k}(V(z_{j})R_{0}^{+}(\lambda^{2m})(z_{j}-z_{j+1}))\right.\nonumber\\	&\left.-R_{0}^{-}(\lambda^{2m})(z_{0}-z_{1})\prod_{j=1}^{k}(V(z_{j})R_{0}^{-}(\lambda^{2m})(z_{j}-z_{j+1}))\right)\lambda^{2m-1}\chi(\lambda^{2m})\d z_{1}\cdots \d z_{k}\d\lambda,
 \end{align*}
 where $z_{0}=x,\ z_{k+1}=y$, we further decompose this integral by restricting the integration in $\mathbb{R}^{nk}$ to
$$
U_1=\left\{(z_1,\ldots, z_k)\in \R^{nk}: \,|t|^{-\frac{1}{2m}}(|z_{0}-z_{1}|+\cdots+|z_{k}-z_{k+1}|)\leq1\right\},\,\quad U_2=\R^{nk}\setminus  U_2,
$$
and denote the part in $U_i$ by $\Omega^{low,i}_{k}(t,x,y)$.

By \eqref{eq2.10}, \eqref{equ4.17} and the algebraic identity
$$
\prod_{j=0}^{k+1}A_{j}^{+}-\prod_{j=0}^{k+1}A_{j}^{-}=\sum_{j=0}^{k+1}A_{0}^{-}\cdots(A_{j}^{+}-A_{j}^{-})A_{j+1}^{+}\cdots A_{k}^{+},
$$
$\Omega^{low,1}_{k}(t,x,y)$ can be written as a linear combination of
\begin{equation*}\label{eq4.59}
\begin{split}
	\int_{[0,1]^{k+1}}&\int_{U_1}\int_{0}^{+\infty}e^{-\i t\lambda^{2m}+\i\lambda_{k_{j_{0}}}s_{j_{0}}|z_{j_{0}-1}-z_{j_{0}}|+\sum_{j=1,j\neq j_{0}}^{k+1}\i\lambda_{k_{j}}s_{j}^{p_{j}}|z_{j-1}-z_{j}|}\prod_{j=1,j\neq j_{0}}^{k+1}|z_{j-1}-z_{j}|^{-t_{j}}\\
	&\times\prod_{j=1}^{k}V(z_{j})\lambda^{n-1+\sigma}\chi(\lambda^{2m})\prod_{j=1}^{k+1}(1-s_{j})^{q_{j}}	\d \lambda \d z_1\cdots \d z_k \d s_1\cdots \d s_{k+1},
\end{split}
\end{equation*}
where $\sigma\geqslant0$, $\lambda_{k_{j}}=\lambda \exp{(\frac{\i k_{j}z}{m})}$, $\lambda_{k_{j_{0}}}=\lambda\,\mbox{or}\,-\lambda$, $1\leq j_{0}\leq k+1$, 
\begin{equation*}
\begin{cases}
	k_{j}\in\{1,\cdots,m\} \qquad\,\,\,\,\,  \,\,\,\,\,\,\mbox{ if }\   j<j_{0},\\
	k_{j}\in\{0,\cdots,m-1\} \ \,\,\,\,\, \,\,\,\,\,\,\,\mbox{ if }\   j_{0}<j\leq k+1,\\
	t_{j}\in\{\frac{n-1}{2},\cdots,n-2m\}\ \,\,\,\,\, \mbox{ for }\  j=1,\cdots,k+1,\  j\neq j_{0},\\
	p_{j}\in\{0,1\}\,\,\,\,\,\,\,\,\,\,\,\qquad\qquad  \mbox{ for }\  j\neq j_{0},
\end{cases}
\end{equation*}
$	q_{j}\in\{0,\cdots,2m-3\}$, and $	q_{j_{0}}\in\{\frac{n-3}{2},\frac{n-1}{2},\cdots,n-3\}$.
Since in $U_1$, it follows that
$$|t|^{-\frac{1}{2m}}\left| \sum_{j=1,j\neq j_{0}}^{k+1}(s_{j}^{p_{j}}|z_{j-1}-z_{j}|)\pm s_{j_{0}}|z_{j_{0}-1}-z_{j_{0}}|\right| \leq1,$$
and $|t|^{-\frac{1}{2m}}|x-y|\leq1$ by the triangle inequality, we can apply Lemma \ref{lemmaA.2} with $b=n-1$ to have
\begin{equation*}
\begin{aligned}
	|\Omega^{low,1}_{k}(t,x,y)|&\lesssim |t|^{-\frac{n}{2m}}\int_{\R^{nk}}\prod_{j=1,j\neq j_{0}}^{k+1}|z_{j-1}-z_{j}|^{-t_{j}}\prod_{j=1}^{k}\langle z_{j}\rangle^{-\beta}dz_{1}\cdots dz_{k}\\
	&\lesssim |t|^{-\frac{n}{2m}}\left(1+|t|^{-\frac{1}{2m}}|x-y|\right)^{-\frac{n(m-1)}{2m-1}},
\end{aligned}
\end{equation*}
where the last inequality follows from Lemma  \ref{lem3.10}.

$\Omega^{low,2}_{k}(t,x,y)$ is, apart from a constant, the integral
$$
\int_{0}^{+\infty}\int_{U_2}e^{-\i t\lambda^{2m}}R_{0}^{\pm}(\lambda^{2m})(z_{0}-z_{1})\prod_{j=1}^{k}\left(V(z_{j})R_{0}^{\pm}(\lambda^{2m})(z_{j}-z_{j+1})\right)\lambda^{2m-1}\chi(\lambda^{2m})\d\lambda \d z_{1}\cdots \d z_{k},
$$
and we will only consider the sign $"+"$ here since the other case is similar. We further split $U_2=\bigcup_{i=1}^k U_{2, i}$ where
$$U_{2, i}=\left\{(z_1,\ldots, z_k)\in U_2: \,|z_{i-1}-z_{i}|\geq|z_{j-1}-z_{j}|\,\mbox{for all}\, j=1,\cdots,k+1\right\}.$$
Without loss of generality, we only consider the integral in $ U_{2, 1}$ which, by \eqref{equ4.17}, can be written as a linear combination of
\begin{equation}\label{eq4.60}
\begin{split}
	&\int_{[0,1]^k}\int_{U_{2, 1}}\int_{0}^{+\infty}e^{-\i t\lambda^{2m}+\i\lambda_{k_{1}}|z_{0}-z_{1}|+\sum_{j=2}^{k+1}\i\lambda_{k_{j}}s_{j}^{p_{j}}|z_{j-1}-z_{j}|}\prod_{j=1}^{k+1}|z_{j-1}-z_{j}|^{-t_{j}}\\
	&\times\prod_{j=1}^{k}V(z_{j})\prod_{j=2}^{k+1}(1-s_{j})^{q_{j}}
	\lambda^{n-1-t_{1}+\sigma}\chi(\lambda^{2m})\d\lambda \d z_{1}\cdots \d z_{k}\d s_{2}\cdots\d s_{k+1},
\end{split}
\end{equation}
where $\sigma\geq 0$,  $\lambda_{k_{j}}=\lambda \exp(\frac{\i k_{j}\pi}{m})$ and $k_{j}\in\{0,\cdots,m-1\}$ for $j=1,\cdots,k+1$, $p_{j}\in\{0,1\}$, $q_{j}\in\{0,\cdots,2m-3\}$ and
\begin{equation*}
\begin{cases}
	t_{1}\in\{\frac{n-1}{2},\cdots,n-2\},\\
	t_{j}\in\{\frac{n-1}{2},\frac{n+1}{2},\cdots,n-2m\}\,\,\, \mbox{ for}\,\,  j\neq1.
\end{cases}
\end{equation*}
In order to apply Lemma \ref{lemmaA.2} to estimate \eqref{eq4.60}, we write
\begin{equation}\label{equ4.low}
\begin{split}	e^{\i\lambda_{k_{1}}|z_{0}-z_{1}|+\sum_{j=2}^{k+1}\i\lambda_{k_{j}}s_{j}^{p_{j}}|z_{j-1}-z_{j}|}=&e^{\i\lambda|z_{0}-z_{1}|+\sum_{j=2}^{k+1}\i\lambda s_{j}^{p_{j}}|z_{j-1}-z_{j}|}\\
	&\times e^{\i\lambda_{k}|z_{0}-z_{1}|-\i\lambda|z_{0}-z_{1}|+\sum_{j=2}^{k+1}
		(\i\lambda_{k_{j}}s_{j}^{p_{j}}|z_{j-1}-z_{j}|-\i\lambda s_{j}^{p_{j}}|z_{j-1}-z_{j}|)}.
\end{split}
\end{equation}
One easily checks with respect to $\lambda$ that
$$
e^{\i\lambda_{k}|z_{0}-z_{1}|-\i\lambda|z_{0}-z_{1}|+\sum_{j=2}^{k+1}
(\i\lambda_{k_{j}}s_{j}^{p_{j}}|z_{j-1}-z_{j}|-\i\lambda s_{j}^{p_{j}}|z_{j-1}-z_{j}|)}\in S^{0}((0,1)),\quad k=0,\cdots,m-1,
$$	
where the seminorms are uniformly bounded in parameters, so we can plug the identity \eqref{equ4.low} into \eqref{eq4.60}, and apply Lemma \ref{lemmaA.2} with $b=n-1-t_1$ to obtain
\begin{equation*}\label{eq4.61}
\begin{split} |\eqref{eq4.60}|\lesssim&\int_{U_{2, 1}}\int_{[0,1]^k}|t|^{-\frac{n-t_1}{2m}}\left(|t|^{-\frac{1}{2m}}\left(|z_{0}-z_{1}|+\sum_{j=2}^{k+1}s_{j}^{p_{j}}|z_{j-1}-z_{j}|\right)\right)^{-\frac{m-1-(n-1-t_{1})}{2m-1}}\\
	&\times\prod_{j=1}^{k+1}|z_{0}-z_{1}|^{-t_{j}}\prod_{j=1}^{k}\langle z_{j}\rangle^{-\beta}\d s_{2}\cdots \d s_{k+1}\d z_{1}\cdots \d z_{k}\\
	\lesssim& \int_{U_{2, 1}} |t|^{-\frac{n}{2m}}(1+|t|^{-\frac{1}{2m}}|z_{0}-z_{1}|)^{-t_1-\frac{m-1-(n-1-t_{1})}{2m-1}}\prod_{j=1}^{k+1}|z_{0}-z_{1}|^{-t_{j}}\prod_{j=1}^{k}\langle z_{j}\rangle^{-\beta}\d z_{1}\cdots \d z_{k}\\
	\lesssim& |t|^{-\frac{n}{2m}}\left(1+|t|^{-\frac{1}{2m}}|x-y|\right)^{-\frac{n(m-1)}{2m-1}},
\end{split}
\end{equation*}
where in the second inequality, we have used Lemma \ref{lem3.10} and the fact in $U_{2, 1}$ that
$$|z_{0}-z_{1}|\sim|z_{0}-z_{1}|+\sum_{j=z}^{k+1}s_{j}^{p_{j}}|z_{j-1}-z_{j+1}|,\quad|t|^{-\frac{1}{2m}}|z_{0}-z_{1}|\gtrsim1;$$
in the last inequality, we have used the facts that
$|z_{0}-z_{1}|\gtrsim \sum_{j=1}^{k+1}{|z_{j-1}-z_j|}\ge |x-y|$ holds in $U_{2, 1}$, and that $t_1+\frac{m-1-(n-1-t_{1})}{2m-1}\ge \frac{n(m-1)}{2m-1}$ holds when $t_1\ge \frac{n-1}{2}$. Therefore we complete the proof of \eqref{eq4.57}.

\subsection{Estimate for the remainder term $(\Omega_{r}^{+, low}-\Omega_{r}^{-, low})(t,x,y)$}\label{sec4.2.4}\

Let $N>[\frac{n}{2m}]+2$ in \eqref{eq4.54}. The aim of this part is to prove that the kernel of $\Omega_{r}^{+, low}-\Omega_{r}^{-, low}$ satisfies
\begin{equation}\label{equ-ker-remai}
	\left|(\Omega_{r}^{+, low}-\Omega_{r}^{-, low})(t,x,y)\right| \lesssim (1+|t|)^{-(\frac{n}{2m}-\sigma)}(1+|t|^{-\frac{n}{2 m}})\left(1+|t|^{-\frac{1}{2 m}}|x-y|\right)^{-\frac{n(m-1)}{2 m-1}},
\end{equation}
where $\sigma=0$ if zero is not an  eigenvalue of $H$, and $\sigma=2$ otherwise.

We first introduce the following two lemmas, whose proofs are given in Appendix \ref{app-002}.
\begin{lemma}\label{lemma4.6}
	Let $n>4m$ be odd, $l\in\{[\frac{n}{2m}]+3,\cdots\}$, $s_{i}\geq 0$ be integers with $s_0+\cdots+s_{l}\leq \frac{n+1}{2}$, and $R_{0}^{\pm,(s_{i})}(\lambda^{2m})$ be the operators with integral kernels $\partial_{\lambda}^{s_{i}}\big[R_{0}^{\pm}(\lambda^{2m})(x-y)\big]$.
	If $|V(x)|\lesssim\langle x\rangle^{-(n+2)-}$, then for each $k\in I^\pm$ and $\tau\in\{0,\cdots,n-1\}$, it follows that
	{\small
		\begin{equation}\label{eq4.53.1}
			\left\|v(y) \left(\prod_{j=0}^{l-1}(R_{0}^{\pm,(s_{j})}(\lambda^{2m})V)\partial_\lambda^{s_l}\left(|x-\cdot|^{-\tau}e^{\mp \i\lambda s|x|}e^{\i\lambda_{k}s|x-y|}  \right)\right)(y)\right\|_{L^2_y} \lesssim \lambda^{-s_l}\langle x\rangle^{-\tau},
		\end{equation}
	}where the estimate holds uniformly in $\lambda,s\in(0,1)$, and $x\in\mathbb{R}^n$. Further,
	we also have for each $k\in I^+$, $j\in\{0,\cdots,l-1\}$ and $\tau\in\{\frac{n-1}{2},\cdots,n-2m\}$ that
	\begin{equation}\label{eq4.53.2}
		{\small
			\begin{aligned}
				\big\|  v(y)\left(R_{0}^{-,(s_{0})}(\lambda^{2m})
				\right. V& \cdots  \left(R_{0}^{+,(s_{j})}(\lambda^{2m})-R_{0}^{-,(s_{j})}(\lambda^{2m})\right)V
				\cdots R_{0}^{+,(s_{l-2})}(\lambda^{2m}) \\
				&\left.
				\times VR_{0}^{+,(s_{l-1})}(\lambda^{2m})V\partial_\lambda^{s_l}\left(|x-\cdot|^{-\tau}e^{-\i\lambda s|x|}e^{\i\lambda_{k}s|x-y|}  \right)\right)(y)\big\|_{L^2_y}\lesssim{\lambda^{n-2m-s_j}},
			\end{aligned}
		}
	\end{equation}
	which holds uniformly in $\lambda\in(0,1)$ and $x\in\mathbb{R}^n$.
\end{lemma}

\begin{lemma}\label{lemma4.9}
	Assume that $n>4m$ is odd, $|V(x)|\lesssim\langle x\rangle^{-(n+2)-}$ and   $l>[\frac{n}{2m}]+2$, then
	\begin{equation}\label{eq4.62}
		v(R_{0}^{\pm}(\lambda^{2m})V)^{l}R_{0}^{\pm}(\lambda^{2m})(x-\cdot)
		=\int_{0}^{1}e^{\pm \i\lambda|x|}\omega_{1,0}^{\pm}(\lambda,s,\cdot,x)\d s+\int_{0}^{1}e^{\pm \i\lambda s|x|}\omega_{1,1}^{\pm}(\lambda,s,\cdot,x)\d s,
	\end{equation}
	where
	\begin{equation}\label{eq4.62.1}
		\langle x\rangle^{\frac{n-1}{2}}\omega_{1,q}^{\pm}(\lambda,s,\cdot,x)\in S_{\frac{n+1}{2}}^{0}((0,1),\,\, \|\cdot \|_{L^{2}} ),\quad q=0,1.
	\end{equation}
	We also have
	\begin{equation}\label{eq4.63}
		\begin{split}
			&v\left((R_{0}^{+}(\lambda^{2m})V)^{l}R_{0}^{+}(\lambda^{2m})(x-\cdot)-(R_{0}^{-}(\lambda^{2m})V)^{l}R_{0}^{-}(\lambda^{2m})(x-\cdot)\right) \\
			=&\sum_{i=0}^1\left(\int_{0}^{1}e^{\i\lambda s^q|x|}\omega_{2,i}^{+}(\lambda,s,\cdot,x)\d s-\int_{0}^{1}e^{-\i\lambda s^q|x|}\omega_{2,i}^{-}(\lambda,s,\cdot,x)\d s\right),
		\end{split}
	\end{equation}
	and for $i=0,1$ that
	\begin{equation}\label{eq4.63.2}
		\omega_{2,i}^{\pm}(\lambda,s,\cdot,x)\in S_{\frac{n+1}{2}}^{n-2m}((0,1),\,\, \|\cdot \|_{L^{2}} ),
	\end{equation}
	\begin{equation}\label{eq4.63.1}
		\langle x\rangle^{\frac{n-1}{2}}\omega_{2,i}^{\pm}(\lambda,s,\cdot,x)\in S_{\frac{n+1}{2}}^{\frac{n+1}{2}-2m}((0,1),\,\, \|\cdot \|_{L^{2}}).
	\end{equation}
\end{lemma}

Now we are ready to prove the results for $\Omega_{r}^{+, low}-\Omega_{r}^{-, low}$ in \eqref{eq4.54}. If $\lambda_0$ is given by Theorem \ref{thm3.4}, note that $\mathrm{supp}\,\chi(\lambda^{2m})\subset[-\lambda_0,\lambda_0]$, then it suffices to estimate
\begin{equation}\label{eq4.64}\small
\begin{aligned}
	\int_{0}^{+\infty}&e^{-\i t\lambda^{2m}}\chi(\lambda^{2m})\left( \left\langle  (M^+(\lambda))^{-1}v(R_{0}^{+}(\lambda^{2m})V)^{N}R_{0}^{+}(\lambda^{2m})(\cdot-y),
	v(R_{0}^{-}(\lambda^{2m})V)^{N}R_{0}^{-}(\lambda^{2m})(\cdot-x)\right\rangle\right.  \\
	&-\left. \left\langle (M^-(\lambda))^{-1}v(R_{0}^{-}(\lambda^{2m})V)^{N}R_{0}^{-}(\lambda^{2m})(\cdot-y),
	v(R_{0}^{+}(\lambda^{2m})V)^{N}R_{0}^{+}(\lambda^{2m})(\cdot-x)\right\rangle\right)\lambda^{2m-1}
	\d\lambda.
\end{aligned}
\end{equation}
If we set
\begin{equation*}\small
\begin{aligned}
	\Upsilon_{1}(\lambda,x,y)=&
	\left\langle(M^+(\lambda))^{-1}
	v\left( (R_{0}^{+}(\lambda^{2m})V)^{N}R_{0}^{+}(\lambda^{2m})(\cdot-y)
	-(R_{0}^{-}(\lambda^{2m})V)^{N}R_{0}^{-}(\lambda^{2m})(\cdot-y)\right) ,\right.\\
	&\quad\left. v(R_{0}^{-}(\lambda^{2m})V)^{N}R_{0}^{-}(\lambda^{2m})(\cdot-x)\right\rangle,
\end{aligned}
\end{equation*}		
\begin{equation*}\small
\begin{aligned}
	&\Upsilon_2(\lambda,x,y)\\
	=&\left\langle\left( (M^+(\lambda))^{-1}-(M^-(\lambda))^{-1}\right) v(R_{0}^{+}(\lambda^{2m})V)^{N}R_{0}^{+}(\lambda^{2m})(\cdot-y),
	\, v(R_{0}^{+}(\lambda^{2m})V)^{N}R_{0}^{+}(\lambda^{2m})(\cdot-x)\right\rangle,
\end{aligned}
\end{equation*}
and	
\begin{equation*}\small
\begin{aligned}
	\Upsilon_{3}(\lambda,x,y)=&\left\langle (M^-(\lambda))^{-1}
	v(R_{0}^{-}(\lambda^{2m})V)^{N}R_{0}^{-}(\lambda^{2m})(\cdot-y),\right. \\
	&\quad\left.\,v\left( (R_{0}^{-}(\lambda^{2m})V)^{N}R_{0}^{-}(\lambda^{2m})(\cdot-x)-(R_{0}^{+}(\lambda^{2m})V)^{N}R_{0}^{+}(\lambda^{2m})(\cdot-x)\right) \right\rangle,
\end{aligned}
\end{equation*}
then we can rewrite \eqref{eq4.64} as
\begin{equation*}
\int_{0}^{+\infty}e^{-\i t\lambda^{2m}}\chi(\lambda^{2m})(\Upsilon_{1}(\lambda,x,y)+\Upsilon_{2}(\lambda,x,y)+\Upsilon_{3}(\lambda,x,y))\lambda^{2m-1}\d\lambda.
\end{equation*}
It follows from \eqref{eq4.62}--\eqref{eq4.63} that $\Upsilon_{1}(\lambda,x,y)$ has the following expression
\begin{equation*}\label{eq4.65}\small
\begin{aligned}
	\sum_{\pm}\sum_{p,q=0,1}\int_{0}^{1}\int_{0}^{1}e^{\pm \i\lambda s_{1}^{p}|y|+\i\lambda s_{2}^{q}|x|}\left\langle  (M^+(\lambda))^{-1}\omega_{2,p}^{\pm}(\lambda,s_{1},\cdot,y),\,\omega_{1,q}^{-}(\lambda,s_{2},\cdot,x)\right\rangle \d s_{1}\d s_{2},
\end{aligned}
\end{equation*}
$\Upsilon_{2}(\lambda,x,y)$ can be written as
\begin{equation*}\label{eq4.66.1}\small
\begin{aligned}
	\sum_{p,q=0,1}\int_{0}^{1}\int_{0}^{1}e^{\i\lambda s_{1}^{p}|y|-\i\lambda s_{2}^{q}|x|}\left\langle  ((M^+(\lambda))^{-1}-(M^-(\lambda))^{-1})\omega_{1,p}^{+}(\lambda,s_{1},\cdot,y),\,\omega_{1,q}^{+}(\lambda,s_{2},\cdot,x)\right\rangle \d s_{1}\d s_{2},
\end{aligned}
\end{equation*}
and	$\Upsilon_{3}(\lambda,x,y)$ can be written as
\begin{equation*}\label{eq4.66}\small
\begin{aligned}
	\sum_{\pm}\sum_{p,q=0,1}\int_{0}^{1}\int_{0}^{1}e^{-\i\lambda s_{1}^{p}|y|\mp \i\lambda s_{2}^{q}|x|}\left\langle  (M^-(\lambda))^{-1}\omega_{1,p}^{-}(\lambda,s_{1},\cdot,y),\,\omega_{2,q}^{\pm}(\lambda,s_{2},\cdot,x)\right\rangle \d s_{1}\d s_{2}.
\end{aligned}
\end{equation*}

\noindent\emph{Case 1:} $|t|^{-\frac{1}{2m}}(|x|+|y|)\leq1$.\

If the zero energy of $H$ is regular, then Theorem \ref{thm3.4}, \eqref{eq4.62.1} and \eqref{eq4.63.2} yield
$$\lambda^{2m-1}\left\langle  (M^+(\lambda))^{-1}\omega_{2,p}^{\pm}(\lambda,s_{1},\cdot,y),\,\omega_{1,q}^{-}(\lambda,s_{2},\cdot,x)\right\rangle\in S_{\frac{n+1}{2}}^{n-1}((0,\lambda_{0})),$$
$$\lambda^{2m-1}\left\langle  (M^-(\lambda))^{-1}\omega_{1,p}^{-}(\lambda,s_{1},\cdot,y),\,\omega_{2,q}^{\pm}(\lambda,s_{2},\cdot,x)\right\rangle\in S_{\frac{n+1}{2}}^{n-1}((0,\lambda_{0})),$$
$$\lambda^{2m-1}\left\langle ((M^+(\lambda))^{-1}-(M^-(\lambda))^{-1})\omega_{1,p}^{+}(\lambda,s_{1},\cdot,y),\,\omega_{1,q}^{+}(\lambda,s_{2},\cdot,x)\right\rangle \in S_{\frac{n+1}{2}}^{n-1}((0,\lambda_{0})).$$
Thus, we use \eqref{eqA.5} in Lemma \ref{lemmaA.2} with $b=n-1$ to obtain
\begin{equation}\label{eq4.67}
	\left| \eqref{eq4.64}\right| \lesssim (1+|t|)^{-\frac{n}{2m}}\left(1+|t|^{-\frac{1}{2m}}|x-y|\right)^{-\frac{n(m-1)}{2m-1}}.
\end{equation}

On the other hand, if $0$ is an eigenvalue of $H$, then Theorem \ref{thm3.4}, \eqref{eq4.63} and \eqref{eq4.63.2} imply
$$\lambda^{2m-1}\left\langle  (M^+(\lambda))^{-1}\omega_{2,p}^{\pm}(\lambda,s_{1},\cdot,y),\,\omega_{1,q}^{-}(\lambda,s_{2},\cdot,x)\right\rangle\in S_{\frac{n+1}{2}}^{n-2m-1}((0,\lambda_{0})),$$
$$\lambda^{2m-1}\left\langle  (M^-(\lambda))^{-1}\omega_{1,p}^{-}(\lambda,s_{1},\cdot,y),\,\omega_{2,q}^{\pm}(\lambda,s_{2},\cdot,x)\right\rangle\in S_{\frac{n+1}{2}}^{n-2m-1}((0,\lambda_{0})),$$
$$\lambda^{2m-1}\left\langle ((M^+(\lambda))^{-1}-(M^-(\lambda))^{-1})\omega_{1,p}^{+}(\lambda,s_{1},\cdot,y),\,\omega_{1,q}^{+}(\lambda,s_{2},\cdot,x)\right\rangle \in S_{\frac{n+1}{2}}^{n-4m-1}((0,\lambda_{0})).$$
Thus, we use \eqref{eqA.5} in Lemma \ref{lemmaA.2} with $b=n-4m-1$ to obtain
\begin{equation}\label{eq4.67000}
\left| \eqref{eq4.64}\right| \lesssim (1+|t|)^{-\frac{n-4m}{2m}}\left(1+|t|^{-\frac{1}{2m}}|x-y|\right)^{-\frac{n(m-1)}{2m-1}}.
\end{equation}

\noindent\emph{Case 2:}  $|t|^{-\frac{1}{2m}}(|x|+|y|)\geq1$.\

If the zero energy of $H$ is regular, then Theorem \ref{thm3.4}, \eqref{eq4.62.1} and \eqref{eq4.63.1}  yield
{\small
$$\langle x\rangle^{\frac{n-1}{2}} \langle y\rangle^{\frac{n-1}{2}}\lambda^{2m-1}\left\langle  (M^+(\lambda))^{-1}\omega_{2,p}^{\pm}(\lambda,s_{1},\cdot,y),\,\omega_{1,q}^{-}(\lambda,s_{2},\cdot,x)\right\rangle\in S_{\frac{n+1}{2}}^{\frac{n-1}{2}}((0,\lambda_{0})),$$
$$\langle x\rangle^{\frac{n-1}{2}} \langle y\rangle^{\frac{n-1}{2}}\lambda^{2m-1}\left\langle  (M^-(\lambda))^{-1}\omega_{1,p}^{-}(\lambda,s_{1},\cdot,y),\,\omega_{2,q}^{\pm}(\lambda,s_{2},\cdot,x)\right\rangle\in S_{\frac{n+1}{2}}^{\frac{n-1}{2}}((0,\lambda_{0})),$$
$$\langle x\rangle^{\frac{n-1}{2}} \langle y\rangle^{\frac{n-1}{2}}\lambda^{2m-1}\left\langle ((M^+(\lambda))^{-1}-(M^-(\lambda))^{-1})\omega_{1,p}^{+}(\lambda,s_{1},\cdot,y),\,\omega_{1,q}^{+}(\lambda,s_{2},\cdot,x)\right\rangle \in S_{\frac{n+1}{2}}^{n-1}((0,\lambda_{0})).$$
}
Thus we apply Lemma \ref{lemmaA.2} with $b=\frac{n-1}{2}$ to get
\begin{equation} \label{eq4.68}
	\begin{aligned}
		\left| \eqref{eq4.64}\right|
		&\lesssim(1+|t|)^{-\frac{\frac{n+1}{2}}{2m}} \left( |t|^{-\frac{1}{2m}}(|x|+|y|)\right) ^{-\frac{m-\frac{n+1}{2}}{2m-1}}\langle x\rangle^{-\frac{n-1}{2}} \langle y\rangle^{-\frac{n-1}{2}}\\
		&\lesssim |t|^{-\frac{n}{2m}}\left(1+|t|^{-\frac{1}{2m}}|x-y|\right)^{-\frac{n(m-1)}{2m-1}},
	\end{aligned}
\end{equation}
where in the last inequality, we have used the identity
$$
\mbox{$\frac{n-1}{2}+\frac{m-\frac{n+1}{2}}{2m-1}=\frac{n(m-1)}{2m-1}$}.
$$

On the other hand, if $0$ is an eigenvalue of $H$, then Theorem \ref{thm3.4}, \eqref{eq4.62.1} and \eqref{eq4.63.1} yield
\begin{equation*}\small
	\begin{split}
		\langle x\rangle^{\frac{n-1}{2}} \langle y\rangle^{\frac{n-1}{2}}\lambda^{2m-1}\left\langle  (M^+(\lambda))^{-1}\omega_{2,p}^{\pm}(\lambda,s_{1},\cdot,y),\,\omega_{1,q}^{-}(\lambda,s_{2},\cdot,x)\right\rangle\in S_{\frac{n+1}{2}}^{\frac{n-1}{2}-2m}((0,\lambda_{0})),\\
		\langle x\rangle^{\frac{n-1}{2}} \langle y\rangle^{\frac{n-1}{2}}\lambda^{2m-1}\left\langle  (M^-(\lambda))^{-1}\omega_{1,p}^{-}(\lambda,s_{1},\cdot,y),\,\omega_{2,q}^{\pm}(\lambda,s_{2},\cdot,x)\right\rangle\in S_{\frac{n+1}{2}}^{\frac{n-1}{2}-2m}((0,\lambda_{0})),\\
		\langle x\rangle^{\frac{n-1}{2}} \langle y\rangle^{\frac{n-1}{2}}\lambda^{2m-1}\left\langle ((M^+(\lambda))^{-1}-(M^-(\lambda))^{-1})\omega_{1,p}^{+}(\lambda,s_{1},\cdot,y),\,\omega_{1,q}^{+}(\lambda,s_{2},\cdot,x)\right\rangle \in S_{\frac{n+1}{2}}^{n-4m-1}((0,\lambda_{0})).
	\end{split}
\end{equation*}
Since $n-4m-1\geq \frac{n-1}{2}-2m$ when $n>4m$, we can apply Lemma \ref{lemmaA.2} with $b=\frac{n-1}{2}-2m$ to get
\begin{equation} \label{eq4.69}
\begin{aligned}
	\left| \eqref{eq4.64}\right|
	&\lesssim(1+|t|)^{-\frac{\frac{n+1}{2}-2m}{2m}} \left( |t|^{-\frac{1}{2m}}(|x|+|y|)\right) ^{-\min\left\{\frac{3m-\frac{n+1}{2}}{2m-1},0\right\}}\langle x\rangle^{-\frac{n-1}{2}} \langle y\rangle^{-\frac{n-1}{2}}\\
	&\lesssim |t|^{-\frac{n-2m}{2m}}\left(1+|t|^{-\frac{1}{2m}}|x-y|\right)^{-\frac{n(m-1)}{2m-1}}.\\
\end{aligned}
\end{equation}
Here, we have used $\langle x\rangle^{-1}\langle y\rangle^{-1}\le (|x|+|y|)^{-1}\le |x-y|^{-1}$ and
$\mbox{$\frac{n(m-1)}{2m-1}\le  \min\{\frac{3m-\frac{n+1}{2}}{2m-1},0\}+\frac{n-1}{2}$}.$

Now it follows from \eqref{eq4.67}-\eqref{eq4.69} that the kernel of $\Omega_{r}^{+, low}-\Omega_{r}^{-, low}$ satisfies that 
\eqref{equ-ker-remai}.

\section{The proof of Theorem \ref{thmmain} (high energy part)}\label{section4}

\subsection{Overview}\

Given $K\in\mathbb{N}_+$, we apply the resolvent identity
\begin{equation*}
	R^\pm(\lambda)=\sum_{k=0}^{2K-1}(-1)^kR_0^\pm(\lambda)(VR_0^\pm(\lambda))^k+(R_0^\pm(\lambda) V)^KR^\pm(\lambda)(VR_0^\pm(\lambda))^K,
\end{equation*}
to the Stone's formula of $e^{-\mathrm{i}tH}P_{ac}(H)\tilde{\chi}(H)$, then
\begin{equation*}
	e^{-\mathrm{i}tH}P_{ac}(H)\tilde{\chi}(H)=\sum_{k=0}^{2K-1}\Omega_k^{high}+\Omega_{K,r}^{+,high}-\Omega_{K,r}^{-,high},
\end{equation*}
where
\begin{equation*}\label{Omegak}
	\Omega_k^{high}=\frac{(-1)^k}{2\pi\i}\int_0^{+\infty}e^{-\mathrm{i}t\lambda}\tilde{\chi}(\lambda)\left(R_0^+(\lambda)(VR_0^+(\lambda))^k-R_0^-(\lambda)(VR_0^-(\lambda))^k\right)\d\lambda,
\end{equation*}
\begin{equation}\label{OmegaKr}
	\Omega_{K,r}^{\pm,high}=\frac{1}{2\pi\i}\int_0^{+\infty}e^{-\mathrm{i}t\lambda}\tilde{\chi}(\lambda)(R_0^\pm(\lambda)V)^KR^\pm(\lambda)(VR_0^\pm(\lambda))^K\d\lambda.
\end{equation}
These integrals converge in weak* sense when $V$ the assumption of Theorem \ref{thmmain}, and $K$ will be chosen sufficiently large later.

The distribution kernel of $\Omega_0^{high}$ is
\begin{equation*}
	\Omega_0^{high}(t,x,y)=\mathscr{F}^{-1}\left(\tilde{\chi}(|\cdot|^{2m})e^{-\i t|\cdot|^{2m}}\right)(x-y),
\end{equation*}
by the Fourier representation of the spectral measure of $(-\Delta)^{m}$, and the estimate for $\Omega_0^{high}(t,x,y)$ has already been shown in \cite[Lemma 2.1]{HHZ}.

The distribution kernel of $\Omega_k^{high}$ when $k\geq1$ is formally the repeated integral
\begin{equation*}
	\begin{split}
		&\Omega_k^{high}(t,x,y)\\
		=&\frac{(-1)^k}{2\pi\i}\int_0^{+\infty}e^{-\mathrm{i}t\lambda}\tilde{\chi}(\lambda)\left(\int_{\mathbb{R}^{kn}}\left(\mbox{$\prod\limits_{i=0}^kR_0^+(\lambda)(r_i)-\prod\limits_{i=0}^kR_0^-(\lambda)(r_i)$}\right)\mbox{$\prod\limits_{i=1}^kV(x_i)$}dx_1\cdots dx_k\right)\d\lambda,
	\end{split}
\end{equation*}
where $r_i=x_i-x_{i+1}$, $x_0=x$ and $x_{k+1}=y$. For every fix $k\in\{1,\cdots,2K-1\}$, let
\begin{equation*}
	X=|r_0|+\cdots+|r_k|,\quad T=|t|^\frac{1}{2m}+|t|,\quad\mathbb{K}=\{1,\cdots,k\},\quad\mathbb{K}_0=\{0,\cdots,k\}.
\end{equation*}
We fix a sufficiently small $\delta>0$ which will be chosen later, take $\phi\in C_c^\infty(\mathbb{R})$ where $\operatorname{supp}\phi\subset[-1,1]$ and $\phi=1$ in $[-\frac12,\frac12]$, and further decompose $\Omega_k^{high}(t,x,y)$ into
\begin{equation}\label{Omegak1origin}
	\begin{split}
		&\Omega_k^{high,1}(t,x,y)
		=\frac{(-1)^k}{2\pi\i}\int_0^{+\infty}e^{-\mathrm{i}t\lambda}\tilde{\chi}(\lambda)\\
		&\left(\int_{\mathbb{R}^{kn}}\left(\mbox{$\prod\limits_{i\in\mathbb{K}_0}R_0^+(\lambda)(r_i)-\prod\limits_{i\in\mathbb{K}_0}R_0^-(\lambda)(r_i)$}\right)\mbox{$\prod\limits_{i\in\mathbb{K}}V(x_i)$}\phi(\mbox{$\frac{X}{\delta T}$})dx_1\cdots dx_k\right)\d\lambda,
	\end{split}
\end{equation}
and
\begin{equation}\label{Omegak2origin}
	\begin{split}
		&\Omega_k^{high,2}(t,x,y)=\frac{(-1)^k}{2\pi\i}\int_0^{+\infty}e^{-\mathrm{i}t\lambda}\tilde{\chi}(\lambda)\\
		&\times\left(\int_{\mathbb{R}^{kn}}\left(\mbox{$\prod\limits_{i\in\mathbb{K}_0}R_0^+(\lambda)(r_i)-\prod\limits_{i\in\mathbb{K}_0}R_0^-(\lambda)(r_i)$}\right)\mbox{$\prod\limits_{i\in\mathbb{K}}V(x_i)$}(1-\phi(\mbox{$\frac{X}{\delta T}$}))dx_1\cdots dx_k\right)\d\lambda.
	\end{split}
\end{equation}

To prove Theorem \ref{thmmain} (high energy part), we summarize what remains to show:

\begin{itemize}
	\item[(\romannumeral1)] Show the existence of a sufficiently large $K$ such that
	\begin{equation*}\label{estOmegaK}
		|\Omega_{K,r}^{\pm,high}(t,x,y)|\lesssim|t|^{-\frac{n}{2m}}(1+|t|^{-\frac{1}{2m}}|x-y|)^{-\frac{n(m-1)}{2m-1}},\quad|t|>0,\,x,y\in\mathbb{R}^n.
	\end{equation*}
	The proof will not be presented in this paper, because it is exactly the same one in \cite{CHHZ}, and $K>[\frac{n-4}{4}]+2$ suffices to give such estimate.
	
	\item[(\romannumeral2)] For any fixed $k\geq1$, show the existence of a sufficiently small $\delta>0$ depending on $\lambda_0$ such that
	\begin{equation}\label{Omegahigh1}
		|\Omega_k^{high,1}(t,x,y)|\begin{cases}
			\lesssim_N|t|^{-N},\quad|t|\gtrsim1,\,0<|x-y|\lesssim\delta|t|,\,N\in\mathbb{N}_+,\\
			\lesssim|t|^{-\frac{n}{2m}},\quad0<|t|\lesssim1,\,0<|x-y|\lesssim\delta|t|^\frac{1}{2m}.
		\end{cases}
	\end{equation}
	The proof is a slight adjustment of that in the low dimensional case \cite{CHHZ}, and we will present the proof later in Section \ref{section4.2} for the sake of completeness. 

	\item[(\romannumeral3)] For any fixed $\delta>0$ and $k\geq1$, show
	\begin{equation}\label{Ogema2bound}
		|\Omega_k^{high,2}(t,x,y)|\lesssim\begin{cases}
			|t|^{-\frac{n}{2}}(1+|t|^{-1}|x-y|)^{-\frac{n(m-1)}{2m-1}},\quad&|t|\gtrsim1,\,x,y\in\mathbb{R}^n,\\
			|t|^{-\frac{n}{2m}}(1+|t|^{-\frac{1}{2m}}|x-y|)^{-\frac{n(m-1)}{2m-1}},\quad&0<|t|\lesssim1,\,x,y\in\mathbb{R}^n.
		\end{cases}
	\end{equation}
	This estimate is the subtlest one in the paper where the regularity of $V$ is involved, and we will first sketch the modules of its proof in Section \ref{section4.3}, leaving the lengthy technical details for a particularly complicated part in the next Section \ref{section5}.
\end{itemize}

We also remark that the RHS of \eqref{Omegahigh1} and \eqref{Ogema2bound} are stronger than claimed in Theorem \ref{thmmain} (high energy part).

\subsection{Estimate for $\Omega_k^{high,1}(t,x,y)$}\label{section4.2}\

To show \eqref{Omegahigh1}, we only need to assume $V\in L^\infty$. Since $\Omega_k^{high,1}(t,x,y)=0$ when $|x-y|>\delta T$, we only need to consider its estimates when $0<|x-y|\leq\delta T$. We start by scaling the formal expression \eqref{Omegak1origin} of $\Omega_k^{high,1}(t,x,y)$:
\begin{equation*}
	\begin{split}
		\Omega_k^{high,1}(t,x,y)=&\frac{(-1)^km}{\pi\i}\int_0^{+\infty}e^{-\mathrm{i}t\lambda^{2m}}\tilde{\chi}(\lambda^{2m})\lambda^{2m-1}\\
		&\times\left(\int_{\mathbb{R}^{kn}}\left(\mbox{$\prod\limits_{i\in\mathbb{K}_0}R_0^+(\lambda^{2m})(r_i)-\prod\limits_{i\in\mathbb{K}_0}R_0^-(\lambda^{2m})(r_i)$}\right)\mbox{$\prod\limits_{i\in\mathbb{K}}V(x_i)$}\phi(\mbox{$\frac{X}{\delta T}$})dx_1\cdots dx_k\right)\d\lambda.
	\end{split}
\end{equation*}
Note that
\begin{equation*}
	\prod\limits_{i\in\mathbb{K}_0}R_0^+(\lambda^{2m})(r_i)-\prod\limits_{i\in\mathbb{K}_0}R_0^-(\lambda^{2m})(r_i)=\sum\limits_{i\in\mathbb{K}_0}\left(R_0^+(\lambda^{2m})(r_i)-R_0^-(\lambda^{2m})(r_i)\right)\prod\limits_{i'\in\mathbb{K}_0\setminus\{i\}}R_0^{\delta_{i,i'}}(\lambda^{2m})(r_{i'}),
\end{equation*}
where
\begin{equation*}
	\delta_{i,i'}=\begin{cases}
		-,\quad&i'<i,\\
		+,&i'>i,
	\end{cases}
\end{equation*}
and we know by \eqref{eq2.10} that $R_0^+(\lambda^{2m})(r_i)-R_0^-(\lambda^{2m})(r_i)$ is a finite linear combination of the form
\begin{equation*}
	\lambda^{n-2m}\int_0^1e^{\pm\i s\lambda|r_i|}(1-s)^{n-3-j}\d s,\quad\mbox{$j=0,\cdots,\frac{n-3}{2}$}.
\end{equation*}
Also note that when $\lambda>0$, $R_0^{\pm}(\lambda^{2m})(r_{i'})$ is a finite linear combination of the form
\begin{equation*}\label{R0term1}
	e^{\pm\mathrm{i}\lambda|r_{i'}|}|r_{i'}|^{-(n-2-l)}\lambda^{-(2m-2-l)}f_l(\lambda,r_{i'}),\quad\mbox{$l=0,\cdots,\frac{n-3}{2}$},
\end{equation*}
where
\begin{equation*}\label{R0term2}
	|\partial_\lambda^j f_l(\lambda,r_{i'})|\lesssim_j\lambda^{-j},\quad\lambda>0,\,r_{i'}\in\mathbb{R}^n\setminus\{0\},\,j\in\mathbb{N}_0.
\end{equation*}
We know from the above facts that $\Omega_k^{high,1}(t,x,y)$ is the linear combination of
\begin{equation}\label{Omegak1i}
	\begin{split}
		&\Omega_{k,i,j,\vec{l},\vec{\delta}}^{high,1}(t,x,y)\\
		:=&\int_{\mathbb{R}^{kn}}\phi(\mbox{$\frac{X}{\delta T}$})V(x_1)\cdots V(x_k)\prod\limits_{i'\in\mathbb{K}_0\setminus\{i\}}|r_{i'}|^{-(n-2-l_{i'})}\\
		&\times\int_0^1\left(\int_0^{+\infty}e^{-\mathrm{i}t\lambda^{2m}+\mathrm{i}\lambda(\delta_0|r_0|\cdots+\delta_k|r_k|)}g_{\vec{l},i}(\lambda,r_0,\cdots,r_k)\tilde{\chi}(\lambda^{2m})\d\lambda\right)(1-s)^{n-3-j}\d s\d x_1\cdots\d x_k,\\
		&\mbox{$i=0,\cdots,k,\quad j=0,\cdots,\frac{n-3}{2},$}\\
		&\mbox{$\vec{l}=\{l_0,\cdots,l_k\}\setminus\{l_i\}\in\left\{0,\cdots,\frac{n-3}{2}\right\}^k,\quad\vec{\delta}=\{\delta_0,\cdots,\delta_k\}\in\{-1,1\}^{k+1}.$}
	\end{split}
\end{equation}	
where
\begin{equation*}
	\left|\partial_\lambda^qg_{\vec{l},i}(\lambda,r_0,\cdots,r_k)\right|\lesssim_q\lambda^{n-1-\sum_{i'\in\mathbb{K}_0\setminus\{i\}}(2m-2-l_{i'})},\quad\lambda>0,\,r_0,\cdots,r_k\in\mathbb{R}^n\setminus\{0\},\,q\in\mathbb{N}_0.
\end{equation*}

To deal with the integral in $\lambda$, note that
\begin{equation*}
	g_{\vec{l},i}(\cdot,r_0,\cdots,r_k)\tilde{\chi}(\cdot^{2m})\in S^{n-1-\sum_{i'\in\mathbb{K}_0\setminus\{i\}}(2m-2-l_{i'})}\left((\mbox{$\frac{\lambda_0}{2}$},+\infty)\right),
\end{equation*}
with seminorms uniformly bounded in $r_0,\cdots,r_k$. In particular, if $\sum_{i'\in\mathbb{K}_0\setminus\{i\}}(2m-2-l_{i'})\geq0$, we further have
\begin{equation*}
	g_{\vec{l},i}(\cdot,r_0,\cdots,r_k)\tilde{\chi}(\cdot^{2m})\in S^{n-1}\left((\mbox{$\frac{\lambda_0}{2}$},+\infty)\right).
\end{equation*}
Also note that when $X\leq\delta T$ we have $\left|\delta_0|r_0|\cdots+\delta_k|r_k|\right|\leq\delta T$, so if $\delta>0$ is chosen sufficiently small, we can apply Lemma \ref{lemmaA.222} to obtain when $X\leq\delta T$ that
\begin{equation*}
	\begin{split}
			&\left|\int_{\mathbb{R}}e^{-\mathrm{i}t\lambda^{2m}+\mathrm{i}\lambda(\delta_0|r_0|\cdots+\delta_k|r_k|)}g_{\vec{l},i}(\lambda,r_0,\cdots,r_k)\tilde{\chi}(\lambda^{2m})\d\lambda\right|\\
			\lesssim&\begin{cases}
			C_N|t|^{-N},\quad&|t|\gtrsim1,\,N\in\mathbb{N}_+,\\
			|t|^{-\frac{n}{2m}},&0<|t|\lesssim1,\,\,\text{if}\,\,\sum_{i'\in\mathbb{K}_0\setminus\{i\}}(2m-2-l_{i'})\geq0,\\
			|t|^{-\frac{n-\sum_{i'\in\mathbb{K}_0\setminus\{i\}}(2m-2-l_{i'})}{2m}},&0<|t|\lesssim1,\,\,\text{if}\,\,\sum_{i'\in\mathbb{K}_0\setminus\{i\}}(2m-2-l_{i'})<0.
		\end{cases}
	\end{split}
\end{equation*}
Combining this and \eqref{Omegak1i}, we have the long time estimate:
\begin{equation}\label{Omegahigh1t>1}
	\begin{split}
		|\Omega_{k,i,j,\vec{l},\vec{\delta}}^{high,1}(t,x,y)|&\lesssim_N|t|^{-N}\int_{\{X\lesssim\delta T\}}\prod\limits_{i'\in\mathbb{K}_0\setminus\{i\}}|r_{i'}|^{-(n-2-l_{i'})}\d x_1\cdots \d x_k\\
		&\lesssim_N|t|^{-N+kn-\sum_{i'\in\mathbb{K}_0\setminus\{i\}}(n-2-l_{i'})}\\
		&\lesssim_{N'}|t|^{-N'},\quad|t|\gtrsim1,\,0<|x-y|\lesssim\delta|t|,\,N'\in\mathbb{N}_+,
	\end{split}
\end{equation}
the short time estimate when $\sum_{i'\in\mathbb{K}_0\setminus\{i\}}(2m-2-l_{i'})\geq0$:
\begin{equation}\label{Omegahigh1t<1>0}
	\begin{split}
		|\Omega_{k,i,j,\vec{l},\vec{\delta}}^{high,1}(t,x,y)|&\lesssim|t|^{-\frac{n}{2m}}\int_{\{X\lesssim\delta T\}}\prod\limits_{i'\in\mathbb{K}_0\setminus\{i\}}|r_{i'}|^{-(n-2-l_{i'})}\d x_1\cdots \d x_k\\
		&\lesssim|t|^{-\frac{n}{2m}+kn-\sum_{i'\in\mathbb{K}_0\setminus\{i\}}(n-2-l_{i'})}\\
	&=|t|^{-\frac{n}{2m}+\sum_{i'\in\mathbb{K}_0\setminus\{i\}}(2+l_{i'})}\\
&\lesssim|t|^{-\frac{n}{2m}},\quad0<|t|\lesssim1,\,0<|x-y|\lesssim\delta|t|^{\frac{1}{2m}},
	\end{split}
\end{equation}
and the short time estimate when $\sum_{i'\in\mathbb{K}_0\setminus\{i\}}(2m-2-l_{i'})\leq0$:
\begin{equation}\label{Omegahigh1t<1<0}
	\begin{split}
		|\Omega_{k,i,j,\vec{l},\vec{\delta}}^{high,1}(t,x,y)|&\lesssim|t|^{-\frac{n-\sum_{i'\in\mathbb{K}_0\setminus\{i\}}(2m-2-l_{i'})}{2m}}\int_{\{X\lesssim\delta T\}}\prod\limits_{i'\in\mathbb{K}_0\setminus\{i\}}|r_{i'}|^{-(n-2-l_{i'})}\d x_1\cdots \d x_k\\
		&\lesssim|t|^{-\frac{n-\sum_{i'\in\mathbb{K}_0\setminus\{i\}}(2m-2-l_{i'})}{2m}+kn-\sum_{i'\in\mathbb{K}_0\setminus\{i\}}(n-2-l_{i'})}\\
		&=|t|^{-\frac{n}{2m}+k+\frac{2m-1}{2m}\sum_{i'\in\mathbb{K}_0\setminus\{i\}}(2+l_{i'})}\\
		&\lesssim|t|^{-\frac{n}{2m}},\quad0<|t|\lesssim1,\,0<|x-y|\lesssim\delta|t|^{\frac{1}{2m}},
	\end{split}
\end{equation}
by scaling, and the estimates starting from the integral in $x_i$ using Lemma \ref{lmEd}. Now the estimate \eqref{Omegahigh1} for $\Omega^{high,1}(t,x,y)$ has been shown by \eqref{Omegahigh1t>1} \eqref{Omegahigh1t<1>0} and \eqref{Omegahigh1t<1<0}.

\subsection{The sketch of the estimate for $\Omega_k^{high,2}(t,x,y)$}\label{section4.3}\

To show \eqref{Ogema2bound}, we need to assume $V\in C^{\frac{n+1}{2}-2m}(\mathbb{R}^n)$ and \eqref{derest}. We start by the formal expression \eqref{Omegak2origin} of $\Omega_k^{high,2}(t,x,y)$ which also has the form
\begin{equation*}\label{Omegak2origin'}
	\begin{split}
		&\Omega_k^{high,2}(t,x,y)=\frac{(-1)^km}{\pi\i}\int_0^{+\infty}e^{-\mathrm{i}t\lambda^{2m}}\lambda^{2m-1}\tilde{\chi}(\lambda^{2m})\\
		&\times\left(\int_{\mathbb{R}^{kn}}\left(\mbox{$\prod\limits_{i\in\mathbb{K}_0}R_0^+(\lambda^{2m})(r_i)-\prod\limits_{i\in\mathbb{K}_0}R_0^-(\lambda^{2m})(r_i)$}\right)\mbox{$\prod\limits_{i\in\mathbb{K}}V(x_i)$}(1-\phi(\mbox{$\frac{X}{\delta T}$}))\d x_1\cdots \d x_k\right)\d\lambda.
	\end{split}
\end{equation*}
When $\lambda>0$, we know from \eqref{eq2.8} that $R_0^\pm(\lambda^{2m})(r_i)$ is a finite linear combination of the form
\begin{equation*}
	e^{\pm\mathrm{i}\lambda|r_i|}\lambda^{-(2m-2-l)}|r_i|^{-(n-2-l)}f_l(\lambda;r_i),\quad\mbox{$l=0,\cdots,\frac{n-3}{2}$},
\end{equation*}
where
\begin{equation*}
	|\partial_\lambda^j\partial_{r_i}^\alpha f_l(\lambda;r_i)|\lesssim_{j,\alpha}\lambda^{-j}|r_i|^{-|\alpha|},\quad\lambda,|r_i|>0,\,\alpha\in\mathbb{N}_0^n,
\end{equation*}
therefore $\Omega_k^{high,2}(t,x,y)$ is a finite linear combination of the form
\begin{equation}\label{Ivecl}
	\begin{split}
		I^{\vec{l}}(t;x,y)=&\int_0^{+\infty}e^{-\mathrm{i}t\lambda^{2m}}\tilde{\chi}(\lambda^{2m})\lambda^{2m-1-\sum_{i\in\mathbb{K}_0}(2m-2-l_i)}\times\\
		&\left(\int_{\mathbb{R}^{kn}}e^{\pm\mathrm{i}\lambda X}f(\lambda,r_0,\cdots,r_k)\mbox{$\prod\limits_{i\in\mathbb{K}_0}|r_i|^{-(n-2-l_i)}$}\mbox{$\prod\limits_{i\in\mathbb{K}}V(x_i)$}(1-\phi(\mbox{$\frac{X}{\delta T}$}))\d x_1\cdots \d x_k\right)\d\lambda,
	\end{split}
\end{equation}
where $\vec{l}=(l_0,\cdots,l_k)$ with $l_i\in\{\mbox{$0,\cdots,\frac{n-3}{2}$}\}$, and
\begin{equation*}\label{fsymbol}
	\left|\partial_\lambda^j\partial_{r_0}^{\alpha_0}\cdots\partial_{r_k}^{\alpha_k}f(\lambda;r_0,\cdots,r_k)\right|\lesssim_{j,\alpha_0,\cdots,\alpha_k}\lambda^{-j}|r_0|^{-|\alpha_0|}\cdots|r_k|^{-|\alpha_k|},\quad\lambda,|r_i|>0.
\end{equation*}

We split the discussion for $I^{\vec{l}}(t;x,y)$ into two cases according to $\vec{l}$.

\noindent\emph{Case 1}: \textit{There exists $i_0\in\mathbb{K}_0$ with}
\begin{equation}\label{lgood}
	2m-2-l_i\geq0,\quad i\in\mathbb{K}_0\setminus\{i_0\}.
\end{equation}

This case is irrelevant to the regularity of $V$, and somehow similar to that in \cite{CHHZ}. For $j\in\mathbb{K}_0$, let $D_j=\{(x_1,\cdots,x_k)\in\mathbb{R}^{kn};\,|r_j|=\max_{i\in\mathbb{K}_0}|r_i|\}$, then $X\sim|r_j|$ holds in $D_j$. We first rewrite \eqref{Ivecl}
\begin{equation*}
	\begin{split}
		I^{\vec{l}}(t;x,y)=&\sum_{j\in\mathbb{K}_0}\int_{D_j}\mbox{$\prod\limits_{i\in\mathbb{K}_0}|r_i|^{-(n-2-l_i)}$}\mbox{$\prod\limits_{i\in\mathbb{K}}V(x_i)$}(1-\phi(\mbox{$\frac{X}{\delta T}$}))\\
		&\times\left(\int_0^{+\infty}e^{-\mathrm{i}t\lambda^{2m}\pm\mathrm{i}\lambda X}\tilde{\chi}(\lambda^{2m})\lambda^{2m-1-\sum_{i\in\mathbb{K}_0}(2m-2-l_i)}f(\lambda,r_0,\cdots,r_k)\d\lambda\right)\d x_1\cdots \d x_k\\
		:=&\sum_{j\in\mathbb{K}_0}I_j^{\vec{l}}(t;x,y),
	\end{split}
\end{equation*}
and it follows that
\begin{equation}\label{symbolinitial}
	\begin{split}
		\tilde{\chi}(\lambda^{2m})\lambda^{2m-1-\sum_{i\in\mathbb{K}_0}(2m-2-l_i)}f(\lambda,r_0,\cdots,r_k)\in&S^{2m-1-\sum_{i\in\mathbb{K}_0}(2m-2-l_i)}\left((\mbox{$\frac{\lambda_0}{2},+\infty$})\right)\\
		\subset&S^{1+l_{i_0}}\left((\mbox{$\frac{\lambda_0}{2},+\infty$})\right),
	\end{split}
\end{equation}
where the inclusion is due to assumption \eqref{lgood}, and every relevant seminorm is bounded uniformly in parameters $r_0,\cdots,r_k$ by \eqref{fsymbol}. Since $l_{i_0}\geq-1$ implies $1+l_{i_0}\geq-\frac12$, we apply Lemma \ref{lemmaA.222} whenever $0<|t|\lesssim1$ or $|t|\gtrsim1$ to get
\begin{equation*}
	\begin{split}
		&\left|\int_0^{+\infty}e^{-\mathrm{i}t\lambda^{2m}\pm\mathrm{i}\lambda X}\tilde{\chi}(\lambda^{2m})\lambda^{2m-1-\sum_{i\in\mathbb{K}_0}(2m-2-l_i)}f(\lambda,r_0,\cdots,r_k)\d\lambda\right|\\
		\lesssim&|t|^{-\frac12+\mu_{1+l_{i_0}}}X^{-\mu_{1+l_{i_0}}},\quad \mbox{$X\geq\frac{\delta T}{2}\gtrsim T$},
	\end{split}
\end{equation*}
where $\mu_{1+l_{i_0}}$ is defined by \eqref{mu}. Consequently
\begin{equation*}
	\begin{split}
		&|I_j^{\vec{l}}(t;x,y)|\\
		\lesssim&|t|^{-\frac12+\mu_{1+l_{i_0}}}\int_{\{X\gtrsim T\}\cap D_j}X^{-\mu_{1+l_{i_0}}}\mbox{$\prod\limits_{i\in\mathbb{K}}\langle x_i\rangle^{-\frac{n+1}{2}-}$}\mbox{$\prod\limits_{i\in\mathbb{K}_0}|r_i|^{-(n-2-l_i)}$}\d x_1\cdots \d x_k\\
		\sim&|t|^{-\frac12+\mu_{1+l_{i_0}}}\int_{\{X\gtrsim T\}\cap D_j}X^{-\mu_{1+l_{i_0}}-(n-2-l_j)}\mbox{$\prod\limits_{i\in\mathbb{K}}\langle x_i\rangle^{-\frac{n+1}{2}-}$}\mbox{$\prod\limits_{i\in\mathbb{K}_0\setminus\{j\}}|r_i|^{-(n-2-l_i)}$}\d x_1\cdots \d x_k,
	\end{split}
\end{equation*}
where we note that
\begin{equation}\label{Xindexrange}
	\mu_{1+l_{i_0}}+(n-2-l_j)=\mbox{$\frac{n(m-1)}{2m-1}$}+\mbox{$\frac{1}{2m-1}(\frac{n-3}{2}-l_{i_0})$}+(\mbox{$\frac{n-3}{2}-l_j$})\geq0.
\end{equation}

To show long time estimate, we have when $X\sim|r_j|\gtrsim T\sim|t|\gtrsim1$ that
\begin{equation}\label{cal1}
	\begin{split}
		&|t|^{-\frac12+\mu_{1+l_{i_0}}}X^{-\mu_{1+l_{i_0}}-(n-2-l_j)}\\
		\lesssim&|t|^{-\frac12+\mu_{1+l_{i_0}}}(|t|+|x-y|)^{-\mu_{1+l_{i_0}}-\frac{n-1}{2}}\langle r_j\rangle^{-(\frac{n-3}{2}-l_j)}\\
		\sim&\langle r_j\rangle^{-(\frac{n-3}{2}-l_j)}\begin{cases}
			|t|^{-\frac n2},\quad&|t|>|x-y|,\\
			|t|^{-\frac n2+\frac{n(m-1)}{2m-1}-\frac{1}{2m-1}(\frac{n-3}{2}-l_{i_0})}|x-y|^{-\frac{n(m-1)}{2m-1}},&|t|\leq|x-y|,
		\end{cases}\\
		\lesssim&|t|^{-\frac n2}(1+|t|^{-1}|x-y|)^{-\frac{n(m-1)}{2m-1}}\langle r_j\rangle^{-(\frac{n-3}{2}-l_j)},
	\end{split}
\end{equation}
which implies when $|t|\gtrsim1$ and $x,y\in\mathbb{R}^n$ that
\begin{equation*}
	\begin{split}
		|I^{\vec{l}}(t;x,y)|\lesssim&|t|^{-\frac n2}(1+|t|^{-1}|x-y|)^{-\frac{n(m-1)}{2m-1}}\\
		&\times\sum_{j\in\mathbb{K}_0}\int_{\mathbb{R}^{kn}}\langle r_j\rangle^{-(\frac{n-3}{2}-l_j)}\mbox{$\prod\limits_{i\in\mathbb{K}}\langle x_i\rangle^{-\frac{n+1}{2}-}$}\mbox{$\prod\limits_{i\in\mathbb{K}_0\setminus\{j\}}|r_i|^{-(n-2-l_i)}$}\d x_1\cdots \d x_k.
	\end{split}
\end{equation*}
The fact that the integrals above are bounded uniformly in $x_0$ and $x_{k+1}$ is a consequence of Lemma \ref{lmEd} if we first estimate the integral in $x_j$ to get
\begin{equation*}\label{intxj}
	\int_{\mathbb{R}^n}\frac{\langle x_j\rangle^{-\frac{n-1}{2}-}\d x_j}{|x_{j-1}-x_j|^{n-2-l_{j-1}}\langle x_j-x_{j+1}\rangle^{\frac{n-3}{2}-l_j}}\lesssim1,\quad|x_{j-1}-x_{j+1}|>0,
\end{equation*}
and then estimate the repeated integral in variables remained starting from $x_{j-1}$ to $x_1$ or from $x_{j+1}$ to $x_k$.

To show short time estimate, recall \eqref{Xindexrange}, it follows when $0<|t|\lesssim1$ and $X\sim|r_j|\gtrsim T\sim|t|^\frac{1}{2m}$ that, if $l_{i_0}\leq l_j$, we have
\begin{equation}\label{cal2}
	\begin{split}
		&|t|^{-\frac12+\mu_{1+l_{i_0}}}X^{-\mu_{1+l_{i_0}}-(n-2-l_j)}\\
		\lesssim&|t|^{-\frac12+\mu_{1+l_{i_0}}}(|t|^{\frac{1}{2m}}+|x-y|)^{-\mu_{1+l_{i_0}}-(n-2-l_j)}\\
		\sim&\begin{cases}
			|t|^{-\frac{n}{2m}+\frac{1}{2m}(l_j-l_{i_0})},\quad&|t|^\frac{1}{2m}>|x-y|,\\
			|t|^{-\frac n2+\frac{n(m-1)}{2m-1}+\frac{1}{2m}(l_j-l_{i_0})}|x-y|^{-\frac{n(m-1)}{2m-1}},&|t|^\frac{1}{2m}\leq|x-y|,
		\end{cases}\\
		\lesssim&|t|^{-\frac{n}{2m}}(1+|t|^{-\frac{1}{2m}}|x-y|)^{-\frac{n(m-1)}{2m-1}},
	\end{split}
\end{equation}
and if $l_{i_0}>l_j$, we similarly have
\begin{equation*}
	\begin{split}
		|t|^{-\frac12+\mu_{1+l_{i_0}}}X^{-\mu_{1+l_{i_0}}-(n-2-l_j)}=&|t|^{-\frac12+\mu_{1+l_{i_0}}}X^{-\mu_{1+l_{i_0}}-(n-2-l_{i_0})}X^{-(l_{i_0}-l_j)}\\
		\lesssim&|t|^{-\frac12+\mu_{1+l_{i_0}}}X^{-\mu_{1+l_{i_0}}-(n-2-l_{i_0})}|r_{i_0}|^{-(l_{i_0}-l_j)}\\
		\lesssim&|t|^{-\frac{n}{2m}}(1+|t|^{-\frac{1}{2m}}|x-y|)^{-\frac{n(m-1)}{2m-1}}|r_{i_0}|^{-(l_{i_0}-l_j)}.
	\end{split}
\end{equation*}
This implies when $0<|t|\lesssim1$ and $x,y\in\mathbb{R}^n$ that
\begin{equation*}
	\begin{split}
		&|I_j^{\vec{l}}(t;x,y)|\\
		\lesssim&|t|^{-\frac{n}{2m}}(1+|t|^{-\frac{1}{2m}}|x-y|)^{-\frac{n(m-1)}{2m-1}}\\
		&\times\begin{cases}
			\int_{\mathbb{R}^{kn}}\mbox{$\prod\limits_{i\in\mathbb{K}}\langle x_i\rangle^{-\frac{n+1}{2}-}$}\mbox{$\prod\limits_{i\in\mathbb{K}_0\setminus\{j\}}|r_i|^{-(n-2-l_i)}$}\d x_1\cdots \d x_k,\quad&\text{if\,}l_{i_0}\leq l_j,\\
			\int_{\mathbb{R}^{kn}}|r_{i_0}|^{-(n-2-l_j)}\mbox{$\prod\limits_{i\in\mathbb{K}}\langle x_i\rangle^{-\frac{n+1}{2}-}$}\mbox{$\prod\limits_{i\in\mathbb{K}_0\setminus\{j,i_0\}}|r_i|^{-(n-2-l_i)}$}\d x_1\cdots \d x_k,\quad&\text{if\,}l_{i_0}>l_j,
		\end{cases}
	\end{split}
\end{equation*}
and the fact that the integrals above are bounded uniformly in $x_0$ and $x_{k+1}$ is also a consequence of Lemma \ref{lmEd} if we first estimate the integral in $x_j$ to get
\begin{equation*}
	\int_{\mathbb{R}^n}\frac{\langle x_j\rangle^{-\frac{n+1}{2}-}\d x_j}{|x_{j-1}-x_j|^{n-2-l_{j-1}}}\lesssim1,\quad|x_{j-1}-x_j|>0,~\text{if}~i_0\neq j-1,
\end{equation*}
or
\begin{equation*}
	\int_{\mathbb{R}^n}\frac{\langle x_j\rangle^{-\frac{n+1}{2}-}\d x_j}{|x_{j-1}-x_j|^{n-2-l_j}}\lesssim1,\quad|x_{j-1}-x_j|>0,~\text{if}~i_0=j-1,
\end{equation*}
and then estimate the remained repeated integral in the same way.

Now we have proved
\begin{equation*}\label{Iest}
	|I^{\vec{l}}(t;x,y)|\lesssim\begin{cases}
		|t|^{-\frac{n}{2}}(1+|t|^{-1}|x-y|)^{-\frac{n(m-1)}{2m-1}},\quad&|t|\gtrsim1,\,x,y\in\mathbb{R}^n,\\
		|t|^{-\frac{n}{2m}}(1+|t|^{-\frac{1}{2m}}|x-y|)^{-\frac{n(m-1)}{2m-1}},\quad&0<|t|\lesssim1,\,x,y\in\mathbb{R}^n,
	\end{cases}
\end{equation*}
under the assumption \eqref{lgood}.

\noindent\emph{Case 2}: \textit{There are at least two different indices $i_1,i_2\in\mathbb{K}_0$ with}
\begin{equation}\label{lbad}
	l_{i_1}+2-2m>0,\,l_{i_2}+2-2m>0.
\end{equation}

This case is relevant to the regularity of $V$, and we need the whole next section to establish the complicated techniques.

\section{An integration by parts scheme for estimating $\Omega_k^{high,2}(t,x,y)$}\label{section5}

\subsection{Heuristics}\label{section5.1}\

This section takes on our final task left at the end of the previous section, that is to obtain estimate \eqref{lgood} for $I^{\vec{l}}(t;x,y)$ (defined by \eqref{Ivecl}) under the assumption \eqref{lbad}, which completes the estimate \eqref{Ogema2bound} for $\Omega_k^{high,2}(t,x,y)$.

We first quickly give an impression of what we shall do in a rough sense. If we write
\begin{equation}\label{U}
	\begin{split}
		&U(\lambda,x_0,x_{k+1},t)\\
		=&\int_{\mathbb{R}^{kn}}e^{\pm\mathrm{i}\lambda X}f(\lambda,r_0,\cdots,r_k)\mbox{$\prod\limits_{i\in\mathbb{K}_0}|r_i|^{-(n-2-l_i)}$}\mbox{$\prod\limits_{i\in\mathbb{K}}V(x_i)$}(1-\phi(\mbox{$\frac{X}{\delta T}$}))\d x_1\cdots \d x_k,
	\end{split}
\end{equation}
then
\begin{equation*}
	I^{\vec{l}}(t;x,y)=\int_0^{+\infty}e^{-\mathrm{i}t\lambda^{2m}}\tilde{\chi}(\lambda^{2m})\lambda^{2m-1-\sum_{i\in\mathbb{K}_0}(2m-2-l_i)}U(\lambda,x_0,x_{k+1},t)\d\lambda.
\end{equation*}
Heuristically, we want to show that $U(\lambda,x_0,x_{k+1},t)$ can be written as a linear combination of the form
\begin{equation}\label{decaylambda}
		\lambda^{-J}\int_{\mathbb{R}^{nk}}e^{\pm\mathrm{i}\lambda X}g(\lambda,x_0,\cdots,x_{k+1},t)\d x_1\cdots\d x_k,
\end{equation}
where $g$ has bounded derivatives in $\lambda$ for fixed $(x_0,\cdots,x_{k+1},t)$, so that $I^{\vec{l}}(t;x,y)$ can be written as a linear combination of the form
\begin{equation*}
	\int_{\mathbb{R}^{nk}}\left(\int_0^{+\infty}e^{-\mathrm{i}t\lambda^{2m}\pm\mathrm{i}\lambda X}\tilde{\chi}(\lambda^{2m})\lambda^{2m-1-\sum_{i\in\mathbb{K}_0}(2m-2-l_i)-J}g(\lambda,x_0,\cdots,x_{k+1},t)\d\lambda\right)\d x_1\cdots\d x_k,
\end{equation*}
by changing the order of integration in the sense of oscillatory integral. 

If $J$ is large enough, which means the amplitude
\begin{equation}\label{amplitude}
	\tilde{\chi}(\lambda^{2m})\lambda^{2m-1-\sum_{i\in\mathbb{K}_0}(2m-2-l_i)-J}g(\lambda,x_0,\cdots,x_{k+1},t)
\end{equation}
 decays fast enough in $\lambda$, the mechanism run in Section \ref{section4.3} would first implies a good estimate at least in $t$ for the integral in $\lambda$, for example a bound like
\begin{equation}\label{firstinlambda}
	\begin{split}
		&\left|\int_0^{+\infty}e^{-\mathrm{i}t\lambda^{2m}\pm\mathrm{i}\lambda X}\tilde{\chi}(\lambda^{2m})\lambda^{2m-1-\sum_{i\in\mathbb{K}_0}(2m-2-l_i)-J}g(\lambda,x_0,\cdots,x_{k+1},t)\d\lambda\right|\\
		\lesssim&|t|^{-\frac{n}{2m}}G(x_0,\cdots,x_{k+1},t),
	\end{split}
\end{equation}
which intuitively leads to a correct dispersive type estimate, so that we can further hope the integration of $G(x_0,\cdots,x_{k+1})$ in $(x_1,\cdots,x_k)$ to give the spatial counterpart of the estimate in \eqref{lgood}.

There will be two major difficulties when bootstrapping the above heuristics:

\begin{itemize}
	\item[(\romannumeral1)] If $\lambda^{-J}$ in \eqref{decaylambda} comes from integrating \eqref{U} by parts $J$ times in $(x_1,\cdots,x_k)$, then $V$ is of course required to have some regularity. In an extreme case where $l_0=\cdots=l_k=\frac{n-3}{2}$, referring to \eqref{symbolinitial} in the previous section, if we want \eqref{amplitude} to lie in $S^{\frac{n-1}{2}}((\frac{\lambda_0}{2},+\infty))$, then we need $J=-k(2m-2-\frac{n-3}{2})=k(\frac{n+1}{2}-2m)$, which roughly means averagely we have to integrate by parts $\frac{n+1}{2}-2m$ times in every $x_i$ ($i=1,\cdots,k$), and the regularity $V\in C^{\frac{n+1}{2}-2m}$ is necessary for such purpose. The first difficulty is to show that why $V\in C^{\frac{n+1}{2}-2m}$ is actually sufficient, as expected in the Introduction that such regularity condition is possibly the sharp one.
	
	\item[(\romannumeral2)] If $g$ has bounded derivatives in $\lambda$ for fixed $(x_0,\cdots,x_{k+1},t)$, then the oscillatory integral estimate \eqref{firstinlambda} has little to do with $g$, and the behavior of $G$ in variables $(x_1,\cdots,x_k)$ should be consistent with $g$. On the other hand, $g$ is at least as singular in $(x_1,\cdots,x_k)$ as the integrand in \eqref{U}, and so is $G$. However, besides the already existing point singularities $\prod_{i\in\mathbb{K}_0}|r_i|^{-(n-2-l_i)}$, we will see that $G$ or $g$ can also have singularities at some line segments, which are introduced when integrating \eqref{U} by parts, while in fact more point singularities will also be introduced. The second difficulty is to show that why such mixture of point and line singularities are so integrable to give desired estimates.
\end{itemize}

In the rest of this section, we will rigorously establish an integration by parts scheme to serve the above heuristics. There two main technical results. The first one that explains \eqref{decaylambda} in details is Proposition \ref{stepmuprop}. The second one that explains the estimate of the above $G$ is Proposition \ref{intmainest}. We will finally apply these two results to complete the estimate of $\Omega_k^{high,2}(t,x,y)$ in Section \ref{section5.5}.

\subsection{The integration by parts scheme}\label{section5.2}\

Recall $k\in\mathbb{N}_+$, $\mathbb{K}=\{1,\cdots,k\}$ and $\mathbb{K}_0=\{0,\cdots,k\}$, we consider the oscillatory integral in the form of
\begin{equation*}\label{U^l}
	\begin{split}
		U^{\vec{l}}=\int_{\mathbb{R}^{nk}}&e^{\mathrm{i}\lambda X}\prod_{i\in\mathbb{K}}V(x_i)\prod_{i\in\mathbb{K}_0}|r_i|^{-(n-2-l_i)}f(\lambda,r_0,\cdots,r_k)\phi(X/T)\d x_1\cdots \d x_k,
	\end{split}
\end{equation*}
where $r_i=x_i-x_{i+1}$, $X=|r_0|+\cdots+|r_k|$, $\vec{l}=\{l_1,\cdots,l_k\}$ with $0\leq l_i\leq\frac{n-3}{2}$, $\phi\in C^\infty(\mathbb{R})$ is bounded with $\phi'\in\C_c^\infty(\mathbb{R})$, $T>0$, $V\in C^{\frac{n+1}{2}-2m}(\mathbb{R}^n)$ satisfies \eqref{derest}, and $f$ satisfies 
\begin{equation*}
	\left|\partial_\lambda^j\partial_{r_0}^{\alpha_0}\cdots\partial_{r_k}^{\alpha_k}f(\lambda;r_0,\cdots,r_k)\right|\lesssim_{j,\alpha_0,\cdots,\alpha_k}\lambda^{-j}|r_0|^{-|\alpha_0|}\cdots|r_k|^{-|\alpha_k|},\quad\lambda,|r_i|>0.
\end{equation*}

Below are the assumptions and notations throughout the rest of this section.

\begin{itemize}
	
	\item Assume \eqref{lbad}, i.e. there are at least two different indices $i_1,i_2\in\mathbb{K}_0$ such that
	\begin{equation}\label{assumptionl}
		l_{i_1}+2-2m>0,\,l_{i_2}+2-2m>0.
	\end{equation}
	\item  Let $\sigma$ be a fixed permutation of $\mathbb{K}_0$ such that $L_k\geq L_{k-1}\geq\cdots\geq L_0$ where
	\begin{equation}\label{defL_i}
		L_i=\max\{0,l_{\sigma(i)}+2-2m\},\quad i\in\mathbb{K}_0,
	\end{equation}
	and we define $k_0=\min\{i\in\mathbb{K}_0;\,L_i>0\}$.
	\item If $A$ is a non-empty finite subset of $\mathbb{Z}$, we define
	\begin{equation}\label{NL}
		\begin{split}
			N(A,i)&=\min\{j\in A;\,j\geq i\},\quad i\leq\max A,\\
			L(A,i)&=\max\{j\in A;\,j<i\},\quad i>\min A.
		\end{split}
	\end{equation}
	\item If $A$ is a finite subset of $\mathbb{Z}$, we define
	\begin{equation}\label{DiA}
		D_iA=\begin{cases}
			A\setminus\{N(A,i)\},\quad&A\neq\emptyset~\text{and}~i\leq\max A,\\
			A,&\text{otherwise}.
		\end{cases}
	\end{equation}
	One checks that $D_iD_jA=D_jD_iA$ always holds, so it is reasonable to denote
	\begin{equation}\label{DIA}
		D_IA=\left(\prod_{i\in I}D_i\right)A,\quad I\subset\mathbb{Z},
	\end{equation}
	and it is also true that $D_{I_1}D_{I_2}A=D_{I_2}D_{I_1}A$ for any $I_1,I_2\subset\mathbb{Z}$, but it may not be equal to $\prod_{i\in I_1\cup I_2}D_iA$ if $I_1\cap I_2\neq\emptyset$. It obviously follows that $D_{I_1}A\subset D_{I_2}B$ if $I_1\supset I_2$ and $A\subset B$.
	
	\item Denoted by
	\begin{equation}\label{Eij}
		E_{i,j}=\frac{x_i-x_{i+1}}{|x_i-x_{i+1}|}-\frac{x_j-x_{j+1}}{|x_j-x_{j+1}|},\quad i,j\in\mathbb{K}_0,\,i\neq j,
	\end{equation}
	if $F$ is a non-empty finite set of $E_{i,j}$ with $i<j$, we define the norm of $F$ to be
	\begin{equation}\label{defFnorm}
		\|F\|=\left(\sum_{E_{i,j}\in F}|E_{i,j}|^2\right)^\frac12.
	\end{equation}
	If $F=\{E_{i_1,j_1},\cdots,E_{i_r,j_r}\}$ with $j_1<\cdots<j_r$, we sometimes interpret $F$ to be the vector $F=(E_{i_1,j_1},\cdots,E_{i_r,j_r})\in\mathbb{R}^{rn}$ for convenience.
\end{itemize}

Now we introduce our first main technical result that will be used in Section \ref{section5.5} when studying a specific type of oscillatory integrals.

\begin{proposition}\label{stepmuprop}
	For every $\mu\in\{1,\cdots,k-k_0\}$, $U^{\vec{l}}$ is a finite linear combination of oscillatory integrals in the form of
	\begin{equation}\label{stepmu}
		\begin{split}
			\lambda^{-J}\int_{\mathbb{R}^{nk}}&e^{\mathrm{i}\lambda X}\prod_{i\in\mathbb{K}}V^{(\alpha_i)}(x_i)\prod_{i\in\mathbb{K}_0}|r_i|^{-(n-2-l_i+d_i)}\prod_{i=1}^{s}\|F_i\|^{-{p_i}}\\
			&\times g(\lambda,r_0,\cdots,r_k,F_1,\cdots,F_s)\psi(X/T)\d x_1\cdots \d x_k,
		\end{split}
	\end{equation}
	and every such integral is equipped with two sequences of indices
	\begin{equation}\label{musequences}
		\begin{split}
			\emptyset=&I_{0,1}^*\subset\cdots\subset I_{s,1}^*=\{\mbox{$i\in\mathbb{K};\,|\alpha_i|=\frac{n+1}{2}-2m$}\},\\
			\left\{i\in\mathbb{K}_0;\,l_i+2-2m\leq0\right\}=&I_{0,2}^*\subset\cdots\subset I_{s,2}^*=\{i\in\mathbb{K}_0;\,d_i=\max\{0,l_i+2-2m\}\},
		\end{split}
	\end{equation}
	satisfying the following constraints:
	\begin{itemize}
		
		\item[(1)] $J=\sum_{i\in\mathbb{K}}|\alpha_i|+\sum_{i\in\mathbb{K}_0}d_i$, and it follows that
		\begin{equation}\label{stepmuconstraint1}
			\begin{cases}
				1\leq s\leq\mu,\\
				|\alpha_i|\leq\frac{n+1}{2}-2m,\quad&i\in\mathbb{K},\\
				0\leq d_i\leq\max\{0,l_i+2-2m\},\quad&i\in\mathbb{K}_0,\\
				L_{k_0}+\cdots+L_{k_0+\mu-1}\leq J\leq L_{k_0}+\cdots+L_{k-1}.\\
			\end{cases}
		\end{equation}
		If $s<\mu$, it further follows that $J=L_{k_0}+\cdots+L_{k-1}$.
		\item[(2)] If $J<L_{k_0}+\cdots+L_{k-1}$, then for $i=1,\cdots,s$, there exists $\tau^{(i)}\in\mathbb{K}_0$ for either
		\begin{equation}\label{stepmuconstraint2}
			\begin{cases}
				I_{i,1}^*\setminus I_{i-1,1}^*=\{\tau^{(i)}\}\\
				I_{i,2}^*=I_{i-1,2}^*	
			\end{cases}\quad\text{or}\quad\begin{cases}
				I_{i,1}^*=I_{i-1,1}^*\\
				I_{i,2}^*\setminus I_{i-1,2}^*=\{\tau^{(i)}\}
			\end{cases}
		\end{equation}
		to hold. If $J=L_{k_0}+\cdots+L_{k-1}$, $s\geq2$ and $1\leq i\leq s-1$, such $\tau^{(i)}$ also exists.
		
		\item[(3)] Denoted by
		\begin{equation}\label{defI_I^*}
			I_i^*=D_{I_{i,2}^*}(\mathbb{K}\setminus I_{i,1}^*)=D_{I_{i,2}^*}D_{I_{i,1}^*}\mathbb{K}=D_{I_{i,1}^*}D_{I_{i,2}^*}\mathbb{K},\quad i=0,\cdots,s,
		\end{equation}
		it follows that $I_{i-1}^*\neq\emptyset$ whenever $\tau^{(i)}$ exists.

		\item[(4)] Denoted by
		\begin{equation}\label{defF_i}
			F_i=\left\{E_{j_1,j_2};\,j_2\in I_{i-1}^*,\,j_1=L(\mathbb{K}_0\setminus I_{i-1,2}^*,j_2)\right\},\quad i=1,\cdots,s,
		\end{equation}
		it follows that
		\begin{equation}\label{stepmuconstraint4}
			p_i+\cdots+p_s\leq\sum_{j\in I_{F_i}^1}|\alpha_j|+\sum_{j\in I_{F_i}^2}2d_j,\quad i=1,\cdots,s,
		\end{equation}
		where
		\begin{equation}\label{defI_F12}
			\begin{split}
				I_{F_i}^1&=\{j\in\mathbb{K};\,j_1<j\leq j_2,~E_{j_1,j_2}\in F_i\},\\
				I_{F_i}^2&=\{j\in\mathbb{K}_0;\,E_{j,j'}~\text{or}~E_{j',j}\in F_i\}.
			\end{split}
		\end{equation}
		It also follows for $i=1,\cdots,s$ that
		\begin{equation}\label{I_i^*I_F^}
			I_{i,1}^*\setminus I_{i-1,1}^*\subset I_{F_i}^1,\quad I_{i,2}^*\setminus I_{i-1,2}^*\subset I_{F_i}^2.
		\end{equation}
		
		\item[(5)] $g\in C^\infty(\mathbb{R}^{1+(k+1)n+\sum_{i=0}^{s-1}(k-k_0-i)n}\setminus\{0\})$ satisfies
		\begin{equation}\label{Pvariables}
			|\partial_\lambda^j\partial_{r_0}^{\alpha_0}\cdots\partial_{r_k}^{\alpha_k}\partial_{F_1}^{\beta_1}\cdots\partial_{F_s}^{\beta_s}g|\lesssim\lambda^{-j}|r_0|^{-|\alpha_0|}\cdots|r_k|^{-|\alpha_k|}\|F_1\|^{-|\beta_1|}\cdots\|F_s\|^{-|\beta_s|}.
		\end{equation}
		We also have $\psi\in C^\infty(\mathbb{R})$ bounded and $\mathrm{supp}\psi\subset\mathrm{supp}\phi$.
	\end{itemize}
\end{proposition}

\begin{remark}\label{stepmuproprk1}
	In the third constraint, that $I_{i-1}^*\neq\emptyset$ whenever $\tau^{(i)}$ exists is a consequence of the second constraint, because $i-1<s$ always holds in such cases, and we know by $\#\,I_{0,1}^*=0$ and $\#\,I_{0,2}^*=k_0$ that
	\begin{equation*}
		\#I_{i-1}^*=k-k_0-(i-1)>k-k_0-s\geq k-k_0-\mu\geq0.
	\end{equation*}
	If $\mu<k-k_0$ and $J<L_{k_0}+\cdots+L_{k-1}$, we then have $I_s^*\neq\emptyset$ since $\tau^{(s)}$ must exist and so
	\begin{equation*}
		\# I_s^*=k-k_0-s\geq k-k_0-\mu\geq1.
	\end{equation*}
\end{remark}

\begin{remark}\label{stepmuproprk2}
	That $F_i$ in \eqref{defF_i} is well defined is a consequence of $I_{i-1}^*\neq\emptyset$, which holds in the context for $\#I_0^*=k-k_0\geq1$ and the other cases are explained in Remark \ref{stepmuproprk1}. To see this, we first note that $\mathbb{K}_0\setminus I_{i-1,2}^*\neq\emptyset$ for it contains the non-empty $I_{i-1}^*$. If $j=\min I_{i-1}^*$, it follows that $j>\min(\mathbb{K}_0\setminus I_{i-1,2}^*)$, for otherwise we must have $\{0,\cdots,j-1\}\subset I_{i-1,2}^*$, and then $I_{i-1}^*\subset D_{I_{i-1,2}^*}\mathbb{K}\subset D_{\{0,\cdots,j-1\}}\mathbb{K}=\mathbb{K}\setminus\{1,\cdots,j\}$ yields a contradiction. Therefore $j_1$ is well defined for each $j_2\in I_{i-1}^*$ in \eqref{defF_i}. If $E_{j_1,j_2},E_{j_3,j_4}\in F_i$ and $E_{j_1,j_2}\neq E_{j_3,j_4}$, it is easy to check that either $j_1<j_2\leq j_3<j_4$ or $j_3<j_4\leq j_1<j_2$ must hold, so $\|F_i\|$ is also well defined.
\end{remark}

\begin{remark}
	It follows by definition that
	\begin{equation}\label{defI_F12'}
		\begin{split}
			I_{F_i}^1=&\bigcup_{j\in I_{i-1}^*}\{L(\mathbb{K}_0\setminus I_{i-1,2}^*,j)+1,\cdots,j\},\\
			I_{F_i}^2=&\bigcup_{j\in I_{i-1}^*}\{L(\mathbb{K}_0\setminus I_{i-1,2}^*,j),j\}.
		\end{split}
	\end{equation}
	Therefore \eqref{I_i^*I_F^} implies that
	\begin{equation}\label{stepmuconstraint3}
		\tau^{(i)}\leq\max I_{i-1}^*,
	\end{equation}
	holds whenever $\tau^{(i)}$ defined in constraint 3) exists.
\end{remark}

\begin{lemma}\label{lmIF}
	If integral \eqref{stepmu} in Proposition \ref{stepmuprop} is given with $s\geq2$, then
	\begin{equation}
		I_{F_i}^1\supset I_{F_{i+1}}^1,\,I_{F_i}^2\supset I_{F_{i+1}}^2,\,\|F_i\|\geq\|F_{i+1}\|,\quad i=1,\cdots,s-1.
	\end{equation}
	The conclusion also holds for $i=s$ if $\mu<k-k_0$ and $J<L_{k_0}+\cdots+L_{k-1}$, while $I_s^*\neq\emptyset$ by Remark \ref{stepmuproprk1} and then $F_{s+1}$ can be defined in the way of \eqref{defF_i} by Remark \ref{stepmuproprk2}.
\end{lemma}

\begin{proof}
	We first note that the conclusion holds in the special case $F_{i+1}\subset F_i$. Also note that when $s\geq2$ and $1\leq i\leq s-1$, we always have the existence of $\tau^{(i)}$, and $I_i^*=D_{\tau^{(i)}}I_{i-1}^*$ holds by constraint 2), meaning $I_i^*\subset I_{i-1}^*$ and
	\begin{equation}\label{Ii-1-Ii}
		I_{i-1}^*\setminus I_i^*=\{N(I_{i-1}^*,\tau^{(i)})\},
	\end{equation}
	where $N(I_{i-1}^*,\tau^{(i)})$ is well defined because of \eqref{stepmuconstraint3}. By Remark \ref{stepmuproprk1}, \eqref{Ii-1-Ii} is also true for $i=s$ if $\mu<k-k_0$ and $J<L_{k_0}+\cdots+L_{k-1}$.
	
	If $I_{i,1}^*\setminus I_{i-1,1}^*=\{\tau^{(i)}\}$ and $I_{i,2}^*=I_{i-1,2}^*$, we of course have $F_{i+1}\subset F_i$ by definition.
	
	If $I_{i,1}^*=I_{i-1,1}^*$, $I_{i,2}^*\setminus I_{i-1,2}^*=\{\tau^{(i)}\}$, and $\tau^{(i)}=\max I_{i-1}^*$, then $\tau^{(i)}=N(I_{i-1}^*,\tau^{(i)})$ by definition. For any $E_{j_1,j_2}\in F_{i+1}$, we have $j_2\in I_i^*\subset I_{i-1}^*$, $j_2<\tau^{(i)}$ and
	\begin{equation}\label{j_1}
		j_1=L(\mathbb{K}_0\setminus(I_{i-1,2}^*\cup\{\tau^{(i)}\}),j_2)=L(\mathbb{K}_0\setminus I_{i-1,2}^*,j_2),
	\end{equation}
	which means $E_{j_1,j_2}\in F_i$. So $F_{i+1}\subset F_i$ is also true.
	
	If $I_{i,1}^*=I_{i-1,1}^*$, $I_{i,2}^*\setminus I_{i-1,2}^*=\{\tau^{(i)}\}$, $\tau^{(i)}<\max I_{i-1}^*$ and $\tau^{(i)}\notin I_{i-1}^*$, we must have $\tau^{(i)}\notin I_i^*$ and $\tau^{(i)}<N(I_{i-1}^*,\tau^{(i)})$. So for any $E_{j_1,j_2}\in F_{i+1}$, we have either $j_2<\tau^{(i)}$ or $j_2>N(I_{i-1}^*,\tau^{(i)})$. Now \eqref{j_1} of course holds when $j_2<\tau^{(i)}$, and it also holds when $j_2>N(I_{i-1}^*,\tau^{(i)})$ because $\tau^{(i)}<N(I_{i-1}^*,\tau^{(i)})\in I_{i-1}^*\subset\mathbb{K}_0\setminus I_{i-1,2}^*$, which combined with the assumption $I_{i,2}^*\setminus I_{i-1,2}^*=\{\tau^{(i)}\}$ also implies $N(I_{i-1}^*,\tau^{(i)})\in I_{i-1}^*\subset\mathbb{K}_0\setminus I_{i,2}^*$. Therefore $E_{j_1,j_2}\in F_i$, and $F_{i+1}\subset F_i$ still holds.
	
	Lastly, if $I_{i,1}^*=I_{i-1,1}^*$, $I_{i,2}^*\setminus I_{i-1,2}^*=\{\tau^{(i)}\}$, $\tau^{(i)}<\max I_{i-1}^*$ and $\tau^{(i)}\in I_{i-1}^*$, we must have $\tau^{(i)}=N(I_{i-1}^*,\tau^{(i)})$ and the existence of $N(I_{i-1}^*,\tau^{(i)}+1)=N(I_i^*,\tau^{(i)})\in I_i^*$. Splitting $F_i=F_i^{(1)}\cup F_i^{(2)}$ where
	\begin{equation}\label{F_isplit}
		\begin{split}
			F_i^{(1)}=\{E_{j_1,j_2}\in F_i;\,j_2<\tau^{(i)}\,\text{or}\,j_2>N(I_{i-1}^*,\tau^{(i)}+1)\},\\
			F_i^{(2)}=\{E_{j_1,j_2}\in F_i;\,j_2=\tau^{(i)}\,\text{or}\,j_2=N(I_{i-1}^*,\tau^{(i)}+1)\},
		\end{split}
	\end{equation}
	the same discussion in the previous case shows that
	\begin{equation}
		F_i^{(1)}=\{E_{j_1,j_2}\in F_{i+1};\,j_2<\tau^{(i)}\,\text{or}\,j_2>N(I_{i-1}^*,\tau^{(i)}+1)\},
	\end{equation}
	and therefore
	\begin{equation}\label{F_i+1}
		F_{i+1}=F_i^{(1)}\cup\{E_{\tilde{j},N(I_{i-1}^*,\tau^{(i)}+1)}\},
	\end{equation}
	where $\tilde{j}=L(\mathbb{K}_0\setminus I_{i,2}^*,N(I_{i-1}^*,\tau^{(i)}+1))$. If
	\begin{equation*}
		L(\mathbb{K}_0\setminus I_{i-1,2}^*,N(I_{i-1}^*,\tau^{(i)}+1))>\tau^{(i)},
	\end{equation*}
	we must have
	\begin{equation*}
		L(\mathbb{K}_0\setminus I_{i-1,2}^*,N(I_{i-1}^*,\tau^{(i)}+1))=L(\mathbb{K}_0\setminus(I_{i-1,2}^*\cup\{\tau^{(i)}\}),N(I_{i-1}^*,\tau^{(i)}+1))=\tilde{j},
	\end{equation*}
	which means $E_{\tilde{j},N(I_{i-1}^*,\tau^{(i)}+1)}\in F_i$ and therefore $F_{i+1}\subset F_i$. If
	\begin{equation}\label{tildej}
		L(\mathbb{K}_0\setminus I_{i-1,2}^*,N(I_{i-1}^*,\tau^{(i)}+1))=\tau^{(i)},
	\end{equation}
	which is the only possibility left since $N(I_{i-1}^*,\tau^{(i)}+1)>\tau^{(i)}\in\mathbb{K}_0\setminus I_{i-1,2}^*$, we must have
	\begin{equation*}
		\tilde{j}=L(\mathbb{K}_0\setminus(I_{i-1,2}^*\cup\{\tau^{(i)}\}),N(I_{i-1}^*,\tau^{(i)}+1))<\tau^{(i)},
	\end{equation*}
	and consequently
	\begin{equation*}
		\tilde{j}=L(\mathbb{K}_0\setminus I_{i,2}^*,\tau^{(i)})=L(\mathbb{K}_0\setminus(I_{i-1,2}^*\cup\{\tau^{(i)}\}),\tau^{(i)})=L(\mathbb{K}_0\setminus I_{i-1,2}^*,\tau^{(i)}),
	\end{equation*}
	which together with \eqref{F_isplit} and \eqref{tildej} implies
	\begin{equation}\label{F_i}
		F_i=F_i^{(1)}\cup\{E_{L(\mathbb{K}_0\setminus I_{i-1,2}^*,\tau^{(i)}),\tau^{(i)}},E_{\tau^{(i)},N(I_{i-1}^*,\tau^{(i)}+1)}\}=F_i^{(1)}\cup\{E_{\tilde{j},\tau^{(i)}},E_{\tau^{(i)},N(I_{i-1}^*,\tau^{(i)}+1)}\},
	\end{equation}
	so it is now obvious to see $I_{F_i}^1=I_{F_{i+1}}^1$ and $I_{F_i}^2\supset I_{F_{i+1}}^2$ if we compare \eqref{F_i+1} with \eqref{F_i}, while $\|F_i\|\geq\|F_{i+1}\|$ is a consequence of the triangle inequality applied to $E_{\tilde{j},N(I_{i-1}^*,\tau^{(i)}+1)}=E_{\tilde{j},\tau^{(i)}}+E_{\tau^{(i)},N(I_{i-1}^*,\tau^{(i)}+1)}$.
\end{proof}

\begin{remark}
	The proof also shows that either $F_{i+1}\subset F_i$ holds, or there exist $j_1<j_2<j_3$ such that $E_{j_1,j_2},E_{j_2,j_3}\in F_i$, $E_{j_1,j_3}\in F_{i+1}$ and $F_i\setminus\{E_{j_1,j_2},E_{j_2,j_3}\}=F_{i+1}\setminus\{E_{j_1,j_3}\}$.
\end{remark}

Next, we introduce some lemmas for the proof of Proposition \ref{stepmuprop}, and formally the key one is Lemma \ref{lmintonce} which exploits the pattern if we integrate \eqref{stepmu} by parts once and once again.

\begin{lemma}\label{lmI_FI^*}
	If integral \eqref{stepmu} in Proposition \ref{stepmuprop} is given, then
	\begin{equation*}
		I_{F_i}^1\subset\mathbb{K}\setminus I_{i-1,1}^*,~I_{F_i}^2\subset\mathbb{K}_0\setminus I_{i-1,2}^*,\quad i=1,\cdots,s.
	\end{equation*}
	The conclusion also holds for $i=s+1$ if $I_s^*\neq\emptyset$ while $F_{s+1}$ can be defined in the way of \eqref{defF_i} by Remark \ref{stepmuproprk2}.
\end{lemma}

\begin{proof}
	To show the first inclusion, we pick up any $j\in I_{F_i}^1$, then there exists $i_0\in I_{i-1}^*\subset\mathbb{K}_0\setminus I_{i-1,2}^*$ with
	\begin{equation*}
		L(\mathbb{K}_0\setminus I_{i-1,2}^*,i_0)<j\leq i_0,
	\end{equation*}
	which implies $D_j(\mathbb{K}_0\setminus I_{i-1,2}^*)=(\mathbb{K}_0\setminus I_{i-1,2}^*)\setminus\{i_0\}$, so it is impossible for $j\in I_{i-1,1}^*$ to hold, because otherwise
	\begin{equation*}
		I_{i-1}^*=D_{I_{i-1,1}^*}D_{I_{i-1,2}^*}\mathbb{K}\subset D_j(\mathbb{K}_0\setminus I_{i-1,2}^*)=(\mathbb{K}_0\setminus I_{i-1,2}^*)\setminus\{i_0\}\notni i_0,
	\end{equation*}
	which is a contradiction. The second inclusion is obvious by \eqref{defI_F12'}.
\end{proof}

\begin{lemma}\label{lmchangeofvariables}
	Given $I_1\subset\mathbb{K}$, $I_2\subset\mathbb{K}_0$, $I=D_{I_2}(\mathbb{K}\setminus I_1)\neq\emptyset$, $y_0=x_0$, $y_{k+1}=x_{k+1}$ and the change of variables
	\begin{equation*}
		y_i=\begin{cases}
			x_i-x_{i+1},\quad&i\in\mathbb{K}\cap I_2,\\
			x_i,&i\in\mathbb{K}\setminus I_2.
		\end{cases}
	\end{equation*}
	Then for each $i\in I$, the following statements hold:
	\begin{itemize}
		\item [1)] If $j\in\mathbb{K}_0$, $j\neq i$ and $j\neq L(\mathbb{K}_0\setminus I_2,i)$, then $x_j-x_{j+1}$ is independent of $y_i$ in the $y$-coordinates.
		
		\item [2)] $\nabla_{y_i}X=-E_{L(\mathbb{K}_0\setminus I_2,i),i}$ holds where $E_{L(\mathbb{K}_0\setminus I_2,i),i}$ is defined.
	\end{itemize}
\end{lemma}

\begin{proof}
	For each $i\in I$, we first note that $L(\mathbb{K}_0\setminus I_2,i)$ is well defined for the same reason explained in Remark \ref{stepmuproprk2}. To show 1), one checks that
	\begin{equation*}
		x_j-x_{j+1}=\begin{cases}
			y_j,&j\in\mathbb{K}\cap I_2,\\
			y_j-\sum_{\tau=j+1}^{j^*}y_\tau,\quad&j\in(\mathbb{K}\setminus I_2)\cup\{0\},
		\end{cases}
	\end{equation*}
	where
	\begin{equation*}
		j^*=N((\mathbb{K}\setminus I_2)\cup\{k+1\},j+1).
	\end{equation*}
	The case of $j\in\mathbb{K}\cap I_2$ is obvious because $I\subset\mathbb{K}\setminus I_2$. If $j\in\mathbb{K}_0\setminus(\mathbb{K}\cap I_2)$, $j\neq i$ and $j\neq L(\mathbb{K}_0\setminus I_2,i)$, since $L(\mathbb{K}_0\setminus I_2,i)=L(\mathbb{K}_0\setminus(\mathbb{K}\cap I_2),i)$ always holds, we have either $j<j+1\leq\cdots\leq j^*<i$ or $i<j<j+1\leq\cdots\leq j^*$, and therefore $x_j-x_{j+1}=y_j-\sum_{\tau=j+1}^{j^*}y_\tau$ is also independent of $y_i$.
	
	To show 2), by the conclusion of 1) and the fact that $L(\mathbb{K}_0\setminus I_2,i)^*=i$, we derive
	\begin{equation}\label{nablaX}
		\begin{split}
			\nabla_{y_i}X&=\nabla_{y_i}|x_i-x_{i+1}|+\nabla_{y_i}|x_{L(\mathbb{K}_0\setminus I_2,i)}-x_{L(\mathbb{K}_0\setminus I_2,i)+1}|\\
			&=\nabla_{y_i}\left|\mbox{$y_i-\sum_{\tau=i+1}^{i^*}y_\tau$}\right|+\nabla_{y_i}\left|y_{L(\mathbb{K}_0\setminus I_2,i)}-\mbox{$\sum_{\tau=L(\mathbb{K}_0\setminus I_2,i)+1}^{i}y_\tau$}\right|\\
			&=\frac{\mbox{$y_i-\sum_{\tau=i+1}^{i^*}y_\tau$}}{\left|\mbox{$y_i-\sum_{\tau=i+1}^{i^*}y_\tau$}\right|}-\frac{y_{L(\mathbb{K}_0\setminus I_2,i)}-\mbox{$\sum_{\tau=L(\mathbb{K}_0\setminus I_2,i)+1}^{i}y_\tau$}}{\left|y_{L(\mathbb{K}_0\setminus I_2,i)}-\mbox{$\sum_{\tau=L(\mathbb{K}_0\setminus I_2,i)+1}^{i}y_\tau$}\right|}\\
			&=\frac{x_i-x_{i+1}}{|x_i-x_{i+1}|}-\frac{x_{L(\mathbb{K}_0\setminus I_2,i)}-x_{L(\mathbb{K}_0\setminus I_2,i)+1}}{|x_{L(\mathbb{K}_0\setminus I_2,i)}-x_{L(\mathbb{K}_0\setminus I_2,i)+1}|}\\
			&=-E_{L(\mathbb{K}_0\setminus I_2,i),i}.
		\end{split}
	\end{equation}
\end{proof}

\begin{lemma}\label{lmintonce}
	Given $\dot{s}\in\{0,\cdots,k-k_0-1\}$, $\dot{\alpha}_i\in\mathbb{N}_0^n$ for $i\in\mathbb{K}$, $\dot{d}_i\in\mathbb{N}_0$ for $i\in\mathbb{K}_0$, and two sequences of indices
	\begin{equation}\label{intonceseq}
		I_{0,1}^*\subset\cdots\subset I_{\dot{s},1}^*\subset\mathbb{K},\quad I_{0,2}^*\subset\cdots\subset I_{\dot{s},2}^*\subset\mathbb{K}_0,
	\end{equation}
	with $I_i^*\neq\emptyset$ for $i=0,\cdots,\dot{s}$ where $I_i^*$ is defined in the way of \eqref{defI_I^*} so that $F_1,\cdots,F_{\dot{s}+1}$ can be defined in the way of \eqref{defF_i} by Remark \ref{stepmuproprk2}, and we assume
	\begin{equation}\label{seqalpha}
		\begin{cases}
			I_{F_1}^1\supset\cdots\supset I_{F_{\dot{s}+1}}^1,\,I_{F_1}^2\supset\cdots\supset I_{F_{\dot{s}+1}}^2,\\			
			|\dot{\alpha}_i|\leq\mbox{$\frac{n+1}{2}$}-2m,&i\in\mathbb{K}\setminus I_{F_{\dot{s}+1}}^1,\\			
			0\leq\dot{d}_i\leq\max\{0,l_i+2-2m\},&i\in\mathbb{K}_0\setminus I_{F_{\dot{s}+1}}^2,\\
			|\dot{\alpha}_i|<\mbox{$\frac{n+1}{2}-2m$},&i\in I_{F_{\dot{s}+1}}^1,\\
			0\leq\dot{d}_i<\max\{0,l_i+2-2m\},&i\in I_{F_{\dot{s}+1}}^2,
		\end{cases}
	\end{equation}
	where $I_{F_i}^1$ and $I_{F_i}^2$ are defined in the way of \eqref{defI_F12} which is also equivalent to \eqref{defI_F12'}. Consider the expression
	\begin{equation}\label{dotJ}
		\begin{split}
			\lambda^{-\dot{J}}\int_{\mathbb{R}^{nk}}&e^{\mathrm{i}\lambda X}\prod_{i\in\mathbb{K}}V^{(\dot{\alpha}_i)}(x_i)\prod_{i\in\mathbb{K}_0}|r_i|^{-(n-2-l_i+\dot{d}_i)}\prod_{i=1}^{\dot{s}+1}\|F_i\|^{-{\dot{p}_i}}\\
			&\times\dot{g}(\lambda,r_0,\cdots,r_k,F_1,\cdots,F_{\dot{s}+1})\dot{\psi}(X/T)\d x_1\cdots \d x_k,
		\end{split}
	\end{equation}
	where $\dot{J}=|\dot{\alpha}_1|+\cdots+|\dot{\alpha}_k|+\dot{d}_0+\cdots+\dot{d}_k$, $\dot{g}(\lambda,r_0,\cdots,r_k,F_1,\cdots,F_{\dot{s}+1})$ is smooth and supported away from the origin in every variable satisfying estimates of the same type to \eqref{Pvariables}, $\dot{\psi}\in C^\infty(\mathbb{R})$ is bounded with $\dot{\psi}'\in C_0^\infty(\mathbb{R})$, and $\dot{p}_i\in\mathbb{N}_0$ with
	\begin{equation}\label{sumdotp}
		\dot{p}_i+\cdots+\dot{p}_{\dot{s}+1}\leq\sum_{j\in I_{F_i}^1}|\dot{\alpha}_j|+\sum_{j\in I_{F_i}^2}2\dot{d}_j,\quad i=1,\cdots,\dot{s}+1.
	\end{equation}
	Then integral \eqref{dotJ} is a finite linear combination of the form
	\begin{equation}\label{dotJ+1}
		\begin{split}
			\lambda^{-\dot{J}-1}\int_{\mathbb{R}^{nk}}&e^{\mathrm{i}\lambda X}\prod_{i\in\mathbb{K}}V^{(\ddot{\alpha}_i)}(x_i)\prod_{i\in\mathbb{K}_0}|r_i|^{-(n-2-l_i+\ddot{d}_i)}\prod_{i=1}^{\dot{s}+1}\|F_i\|^{-{\ddot{p}_i}}\\
			&\times \ddot{g}(\lambda,r_0,\cdots,r_k,F_1,\cdots,F_{\dot{s}+1})\ddot{\psi}(X/T)\d x_1\cdots \d x_k,
		\end{split}
	\end{equation}
	where $\ddot{g}$ and $\ddot{\psi}$ inherit the properties of $\dot{g}$ and $\dot{\psi}$ respectively, and furthermore, there either exists $i_0\in I_{F_{\dot{s}+1}}^1$ such that
	\begin{equation}\label{alt1}
		\begin{cases}
			|\ddot{\alpha}_{i_0}|=|\dot{\alpha}_{i_0}|+1,\\
			\ddot{\alpha}_i=\dot{\alpha}_i,\quad&i\in\mathbb{K}\setminus\{i_0\},\\
			\ddot{d}_i=\dot{d}_i,&i\in\mathbb{K}_0,
		\end{cases}
	\end{equation}
	or exists $i_0\in I_{F_{\dot{s}+1}}^2$ such that
	\begin{equation}\label{alt2}
		\begin{cases}
			\ddot{d}_{i_0}=\dot{d}_{i_0}+1,\\
			\ddot{\alpha}_i=\dot{\alpha}_i,\quad&i\in\mathbb{K},\\
			\ddot{d}_i=\dot{d}_i,&i\in\mathbb{K}_0\setminus\{i_0\},
		\end{cases}
	\end{equation}
	while in both cases we always have
	\begin{equation}\label{ddotp}
		\ddot{p}_i+\cdots+\ddot{p}_{\dot{s}+1}\leq\sum_{j\in I_{F_i}^1}|\ddot{\alpha}_j|+\sum_{j\in I_{F_i}^2}2\ddot{d}_j,\quad i=1,\cdots,\dot{s}+1.
	\end{equation}
\end{lemma}

\begin{proof}
	Set $y_0=x_0$, $y_{k+1}=x_{k+1}$, and the change of variables
	\begin{equation*}
		y_i=\begin{cases}
			x_i-x_{i+1},\quad&i\in\mathbb{K}\cap I_{\dot{s},2}^*\\
			x_i,&i\in\mathbb{K}\setminus I_{\dot{s},2}^*,
		\end{cases}
	\end{equation*}
	we have
	\begin{equation*}
		r_i=x_i-x_{i+1}=\begin{cases}
			y_i,&i\in\mathbb{K}\cap I_{\dot{s},2}^*,\\
			y_i-\sum_{\tau=i+1}^{i^*}y_\tau,\quad&i\in(\mathbb{K}\setminus I_{\dot{s},2}^*)\cup\{0\},
		\end{cases}
	\end{equation*}
	where $i^*=N((\mathbb{K}\setminus I_{\dot{s},2}^*)\cup\{k+1\},i+1)$, and	
	\begin{equation*}
		x_i=\sum_{\tau=i}^{(i-1)^*}y_\tau,\quad i\in\mathbb{K},
	\end{equation*}
	so such change of variables and its inverse only leave with universal constants. We denote $\tilde{V}=\prod_{i\in\mathbb{K}}V^{(\dot{\alpha}_i)}(x_i)$, $\tilde{r}=\prod_{i\in\mathbb{K}_0}|r_i|^{-(n-2-l_i+\dot{d}_i)}$, $\tilde{g}=\dot{g}(\lambda,r_0,\cdots,r_k,F_1,\cdots,F_{\dot{s}+1})$ and $\tilde{\psi}=\dot{\psi}(X/T)$ for convention.

	In the sequel, we will frequently use Lemma \ref{lmchangeofvariables} with $I_1=I_{\dot{s},1}^*$, $I_2=I_{\dot{s},2}^*$ and so $I=I_{\dot{s}}^*\neq\emptyset$. Let $\nabla_{y_{I_{\dot{s}}^*}}=(\nabla_{y_{i_1}},\cdots,\nabla_{y_{i_\nu}})$ where $i_1,\cdots,i_\nu\in I_{\dot{s}}^*$ be increasing. The second conclusion of Lemma \ref{lmchangeofvariables} shows that
	\begin{equation}\label{nablaX*}
		\nabla_{y_{I_{\dot{s}}^*}}X=-\left(E_{L(\mathbb{K}_0\setminus I_{\dot{s},2}^*,i_1),i_1},\cdots,E_{L(\mathbb{K}_0\setminus I_{\dot{s},2}^*,i_\nu),i_\nu}\right),
	\end{equation}
	and therefore $|\nabla_{y_{I_{\dot{s}}^*}}X|=\|F_{\dot{s}+1}\|$. Note that
	\begin{equation*}
		e^{\mathrm{i}\lambda X}=\mathrm{i}^{-1}\lambda^{-1}|\nabla_{y_{I_{\dot{s}}^*}}X|^{-2}\nabla_{y_{I_{\dot{s}}^*}}X\cdot\nabla_{y_{I_{\dot{s}}^*}}e^{\mathrm{i}\lambda X}=\mathrm{i}^{-1}\lambda^{-1}\|F_{\dot{s}+1}\|^{-2}\nabla_{y_{I_{\dot{s}}^*}}X\cdot\nabla_{y_{I_{\dot{s}}^*}}e^{\mathrm{i}\lambda X},
	\end{equation*}
	integration by parts in the $y$-coordinates gives
	\begin{equation}\label{intonce}
		-\mathrm{i}^{-1}\lambda^{-\dot{J}-1}\int_{\mathbb{R}^{nk}}e^{\mathrm{i}\lambda X}\nabla_{y_{I_{\dot{s}}^*}}\cdot\left(\tilde{g}\tilde{\psi}\tilde{V}\tilde{r}\|F_{\dot{s}+1}\|^{-(\dot{p}_{\dot{s}+1}+2)}\prod_{i=1}^{\dot{s}}\|F_i\|^{-{\dot{p}_i}}\nabla_{y_{I_{\dot{s}}^*}}X\right)\d y_1\cdots \d y_k,
	\end{equation}
	and there are five types of integrals derived from \eqref{intonce}.
	
	The first type is
	\begin{equation*}
		\mathrm{I}=\lambda^{-\dot{J}-1}\int_{\mathbb{R}^{nk}}e^{\mathrm{i}\lambda X}\tilde{g}\tilde{\psi}\tilde{r}\|F_{\dot{s}+1}\|^{-(\dot{p}_{\dot{s}+1}+1)}\prod_{i=1}^{\dot{s}}\|F_i\|^{-{\dot{p}_i}}\left(\nabla_{y_{I_{\dot{s}}^*}}\tilde{V}\cdot\frac{\nabla_{y_{I_{\dot{s}}^*}}X}{\|F_{\dot{s}+1}\|}\right)\d y_1\cdots \d y_k.
	\end{equation*}
	Note that
	\begin{equation}\label{nablaV}
		\nabla_{y_{I_{\dot{s}}^*}}\tilde{V}=\nabla_{y_{I_{\dot{s}}^*}}\mbox{$\left(\prod_{i\in\mathbb{K}}V^{(\dot{\alpha}_i)}(\sum_{\tau=i}^{(i-1)^*}y_\tau)\right)$},
	\end{equation}
	if $i\in\mathbb{K}\setminus I_{F_{\dot{s}+1}}^1$, it must follows that $\{i,\cdots,(i-1)^*\}\cap I_{\dot{s}}^*=\emptyset$, otherwise there exists $j\in\{i,\cdots,(i-1)^*\}\cap I_{\dot{s}}^*\subset\{i,\cdots,N((\mathbb{K}\setminus I_{\dot{s},2}^*)\cup\{k+1\},i)\}\cap(\mathbb{K}\setminus I_{\dot{s},2}^*)$, and then $j=N((\mathbb{K}\setminus I_{\dot{s},2}^*)\cup\{k+1\},i)=N(\mathbb{K}\setminus I_{\dot{s},2}^*,i)$ holds, which implies $L(\mathbb{K}\setminus I_{\dot{s},2}^*,j)<i\leq j$ and the contradiction $i\in I_{F_{\dot{s}+1}}^1$. So the gradient in \eqref{nablaV} only falls on $V^{(\dot{\alpha}_{i_0})}(\sum_{\tau=i_0}^{(i_0-1)^*}y_\tau)$ where $i_0\in I_{F_{\dot{s}+1}}^1$. On the other hand, \eqref{nablaX*} indicates that each one dimensional component of $\nabla_{y_{I_{\dot{s}}^*}} X/\|F_{\dot{s}+1}\|$ is a component of $E_{j_1,j_2}/\|F_{\dot{s}+1}\|$ for some $E_{j_1,j_2}\in F_{\dot{s}+1}$. We thus conclude that $\mathrm{I}$ is a finite linear combination of
	\begin{equation}\label{Icom}
		\begin{split}
			&\lambda^{-\dot{J}-1}\int_{\mathbb{R}^{nk}}e^{\mathrm{i}\lambda X}\tilde{r}\|F_{\dot{s}+1}\|^{-(\dot{p}_{\dot{s}+1}+1)}\prod_{i<\dot{s}+1}\|F_i\|^{-{\dot{p}_i}}\\
			&\quad\quad\quad\quad\times V^{(\ddot{\alpha}_{i_0})}(\mbox{$\sum_{\tau=i_0}^{(i_0-1)^*}y_\tau$})\prod_{i\in\mathbb{K}\setminus\{i_0\}}V^{(\dot{\alpha}_i)}(\mbox{$\sum_{\tau=i}^{(i-1)^*}y_\tau$})\tilde{g}\tilde{\psi}Q \d y_1\cdots \d y_k\\
			=&C\lambda^{-\dot{J}-1}\int_{\mathbb{R}^{nk}}e^{\mathrm{i}\lambda X}V^{(\ddot{\alpha}_{i_0})}(x_{i_0})\prod_{i\in\mathbb{K}\setminus\{i_0\}}V^{(\dot{\alpha}_i)}(x_i)\prod_{i\in\mathbb{K}_0}|r_i|^{-(n-2-l_i+\dot{d}_i)}\\
			&\quad\quad\quad\quad\times\|F_{\dot{s}+1}\|^{-(\dot{p}_{\dot{s}+1}+1)}\prod_{i<\dot{s}+1}\|F_i\|^{-{\dot{p}_i}}\tilde{g}\tilde{\psi}Q \d x_1\cdots \d x_k,
		\end{split}
	\end{equation}
	where $i_0\in I_{F_{\dot{s}+1}}^1$, $|\ddot{\alpha}_{i_0}|=|\dot{\alpha}_{i_0}|+1$, and $Q$ is a monomial in $E_{j_1,j_2}/\|F_{\dot{s}+1}\|$ for some $E_{j_1,j_2}\in F_{\dot{s}+1}$.
	
	The second type of integrals derived from \eqref{intonce} is
	\begin{equation*}
		\mathrm{II}=\lambda^{-\dot{J}-1}\int_{\mathbb{R}^{nk}}e^{\mathrm{i}\lambda X}\tilde{g}\tilde{\psi}\tilde{V}\|F_{\dot{s}+1}\|^{-(\dot{p}_{\dot{s}+1}+1)}\prod_{i<\dot{s}+1}\|F_i\|^{-{\dot{p}_i}}\left(\nabla_{y_{I_{\dot{s}}^*}}\tilde{r}\cdot\frac{\nabla_{y_{I_{\dot{s}}^*}}X}{\|F_{\dot{s}+1}\|}\right)\d y_1\cdots \d y_k.
	\end{equation*}
	If $i\in I_{\dot{s}}^*$ and $i_0\in\mathbb{K}_0$, Lemma \ref{lmchangeofvariables} indicates that $y_i$ is independent of $r_{i_0}$ unless $i_0=i$ or $i_0=L(\mathbb{K}_0\setminus I_{\dot{s},2}^*,i)$, where in either case we have $i_0\in I_{F_{\dot{s}+1}}^2$ by definition, and of course $i_0\in(\mathbb{K}\setminus I_{\dot{s},2}^*)\cup\{0\}$. One checks in a way similar to \eqref{nablaX} that
	\begin{equation*}\label{nablar_j}
		\nabla_{y_i}|r_{i_0}|^{-1}=\begin{cases}
			-|r_{i_0}|^{-2}\frac{r_{i_0}}{|r_{i_0}|},\quad&i_0=i,\\
			|r_{i_0}|^{-2}\frac{r_{i_0}}{|r_{i_0}|},&i_0=L(\mathbb{K}_0\setminus I_{\dot{s},2}^*,i),
		\end{cases}
	\end{equation*}
	so we conclude that $\mathrm{II}$ is a finite linear combination of
	\begin{equation}\label{IIcom}
		\begin{split}
			&\lambda^{-\dot{J}-1}\int_{\mathbb{R}^{nk}}e^{\mathrm{i}\lambda X}\tilde{V}\|F_{\dot{s}+1}\|^{-(\dot{p}_{\dot{s}+1}+1)}\prod_{i<\dot{s}+1}\|F_i\|^{-{\dot{p}_i}}\\
			&\quad\quad\quad\times|r_{i_0}|^{-(n-2-l_{i_0}+\dot{d}_{i_0}+1)}\prod_{i\in\mathbb{K}_0\setminus\{i_0\}}|r_i|^{-n-2-l_i+\dot{d}_i}\tilde{g}\tilde{\psi}Q \d y_1\cdots \d y_k\\
			=&C\lambda^{-\dot{J}-1}\int_{\mathbb{R}^{nk}}e^{\mathrm{i}\lambda X}|r_{i_0}|^{-(n-2-l_{i_0}+\dot{d}_{i_0}+1)}\prod_{i\in\mathbb{K}}V^{(\dot{\alpha}_i)}(x_i)\prod_{i\in\mathbb{K}_0\setminus\{i_0\}}|r_i|^{-n-2-l_i+\dot{d}_i}\\
			&\quad\quad\quad\times\|F_{\dot{s}+1}\|^{-(\dot{p}_{\dot{s}+1}+1)}\prod_{i<\dot{s}+1}\|F_i\|^{-{\dot{p}_i}}\tilde{g}\tilde{\psi}Q \d x_1\cdots \d x_k,
		\end{split}
	\end{equation}
	where $i_0\in I_{F_{\dot{s}+1}}^2$ and $Q$ is a polynomial in $(r_{i_0}/|r_{i_0}|,E_{j_1,j_2}/\|F_{\dot{s}+1}\|)$ for some $E_{j_1,j_2}\in F_{\dot{s}+1}$.
	
	The third type of integrals derived from \eqref{intonce} is
	\begin{equation*}
		\mathrm{III}=\lambda^{-\dot{J}-1}\int_{\mathbb{R}^{nk}}e^{\mathrm{i}\lambda X}\tilde{g}\tilde{\psi}\tilde{V}\tilde{r}\nabla_{y_{I_{\dot{s}}^*}}\left(\|F_{\dot{s}+1}\|^{-(\dot{p}_{\dot{s}+1}+2)}\prod_{i<\dot{s}+1}\|F_i\|^{-{\dot{p}_i}}\right)\cdot\nabla_{y_{I_{\dot{s}}^*}}X\d y_1\cdots \d y_k.
	\end{equation*}
	If $i\in I_{\dot{s}}^*$ and $j_1,j_2\in\mathbb{K}_0$, by Lemma \ref{lmchangeofvariables}, $E_{j_1,j_2}$ is independent of $y_i$ unless
	\begin{equation*}\label{j_1j_2}
		\{j_1,j_2\}\cap\{L(\mathbb{K}_0\setminus I_{\dot{s},2}^*,i),i\}\neq\emptyset,
	\end{equation*}
	so then for $i\in I_{\dot{s}}^*$ and $1\leq j_0\leq\dot{s}+1$, we have
	\begin{equation*}
		\nabla_{y_i}\|F_{j_0}\|^{-1}=-\mbox{$\frac12$}\|F_{j_0}\|^{-3}\sum_{\substack{E_{j_1,j_2}\in F_{j_0}\\\{j_1,j_2\}\cap\{L(\mathbb{K}_0\setminus I_{\dot{s},2}^*,i),i\}\neq\emptyset}}\nabla_{y_i}\left|E_{j_1,j_2}\right|^2,
	\end{equation*}
	and there are at most three non-trivial terms in the summation since $L(\mathbb{K}_0\setminus I_{\dot{s},2}^*,i)<i$:
	\begin{itemize}
		\item If $j_1=L(\mathbb{K}_0\setminus I_{\dot{s},2}^*,i)$ and $j_2\neq i$, then $y_i$ is independent of $r_{j_2}$ by Lemma \ref{lmchangeofvariables}, and similar to \eqref{nablaX}, we have
		\begin{equation*}
			\begin{split}
				\nabla_{y_i}\left|E_{j_1,j_2}\right|^2=&\nabla_{y_i}\left|\frac{y_{L(\mathbb{K}_0\setminus I_{\dot{s},2}^*,i)}-\mbox{$\sum_{\tau=L(\mathbb{K}_0\setminus I_{\dot{s},2}^*,i)+1}^{i}y_\tau$}}{|y_{L(\mathbb{K}_0\setminus I_{\dot{s},2}^*,i)}-\mbox{$\sum_{\tau=L(\mathbb{K}_0\setminus I_{\dot{s},2}^*,i)+1}^{i}y_\tau$}|}-\frac{r_{j_2}}{|r_{j_2}|}\right|^2\\
				=&-2\left(E_{L(\mathbb{K}_0\setminus I_{\dot{s},2}^*,i),j_2}\cdot r_{L(\mathbb{K}_0\setminus I_{\dot{s},2}^*,i)}\right)\frac{r_{L(\mathbb{K}_0\setminus I_{\dot{s},2}^*,i)}}{|r_{L(\mathbb{K}_0\setminus I_{\dot{s},2}^*,i)}|^3}-2\frac{E_{L(\mathbb{K}_0\setminus I_{\dot{s},2}^*,i),j_2}}{|r_{L(\mathbb{K}_0\setminus I_{\dot{s},2}^*,i)}|}.
			\end{split}
		\end{equation*}
		
		\item If $j_1=L(\mathbb{K}_0\setminus I_{\dot{s},2}^*,i)$ and $j_2=i$, then
		\begin{equation*}
			\begin{split}
				\nabla_{y_i}\left|E_{j_1,j_2}\right|^2=&\nabla_{y_i}\left|\frac{y_{L(\mathbb{K}_0\setminus I_{\dot{s},2}^*,i)}-\mbox{$\sum_{\tau=L(\mathbb{K}_0\setminus I_{\dot{s},2}^*,i)+1}^{i}y_\tau$}}{|y_{L(\mathbb{K}_0\setminus I_{\dot{s},2}^*,i)}-\mbox{$\sum_{\tau=L(\mathbb{K}_0\setminus I_{\dot{s},2}^*,i)+1}^{i}y_\tau$}|}-\frac{\mbox{$y_i-\sum_{\tau=i+1}^{i^*}y_\tau$}}{|\mbox{$y_i-\sum_{\tau=i+1}^{i^*}y_\tau$}|}\right|^2\\
				=&-2\left(E_{L(\mathbb{K}_0\setminus I_{\dot{s},2}^*,i),i}\cdot r_{L(\mathbb{K}_0\setminus I_{\dot{s},2}^*,i)}\right)\frac{r_{L(\mathbb{K}_0\setminus I_{\dot{s},2}^*,i)}}{|r_{L(\mathbb{K}_0\setminus I_{\dot{s},2}^*,i)}|^3}+2\left(E_{L(\mathbb{K}_0\setminus I_{\dot{s},2}^*,i),i}\cdot r_i\right)\frac{r_i}{|r_i|^3}\\
				&-2\left(|r_{L(\mathbb{K}_0\setminus I_{\dot{s},2}^*,i)}|^{-1}+|r_i|^{-1}\right)E_{L(\mathbb{K}_0\setminus I_{\dot{s},2}^*,i),i}.
			\end{split}
		\end{equation*}
		
		\item If $j_2=L(\mathbb{K}_0\setminus I_{\dot{s},2}^*,i)$, then $y_i$ is independent of $r_{j_1}$, and calculation in the first case above implies
		\begin{equation*}
			\nabla_{y_i}\left|E_{j_1,j_2}\right|^2=2\left(E_{j_1,L(\mathbb{K}_0\setminus I_{\dot{s},2}^*,i)}\cdot r_{L(\mathbb{K}_0\setminus I_{\dot{s},2}^*,i)}\right)\frac{r_{L(\mathbb{K}_0\setminus I_{\dot{s},2}^*,i)}}{|r_{L(\mathbb{K}_0\setminus I_{\dot{s},2}^*,i)}|^3}+2\frac{E_{j_1,L(\mathbb{K}_0\setminus I_{\dot{s},2}^*,i)}}{|r_{L(\mathbb{K}_0\setminus I_{\dot{s},2}^*,i)}|}.
		\end{equation*}
	\end{itemize}
	Since $I_{F_{\dot{s}+1}}^2=\cup_{i\in I_{\dot{s}}^*}\{L(\mathbb{K}_0\setminus I_{\dot{s},2}^*,i),i\}$, the above argument implies that when $1\leq j_0\leq\dot{s}+1$, each one dimensional component of $\nabla_{y_{I_{\dot{s}}^*}}\|F_{j_0}\|^{-1}$ is a sum of the form $\|F_{j_0}\|^{-2}|r_{i_0}|^{-1}Q$ where $i_0\in I_{F_{\dot{s}+1}}^2$ and $Q$ is a polynomial in $(r_{i_0}/|r_{i_0}|,E_{j_1,j_2}/\|F_{j_0}\|)$ for some $E_{j_1,j_2}\in F_{j_0}$. Hence if $\dot{s}\geq1$, $\mathrm{III}$ is then a finite linear combination of
	\begin{equation}\label{nabFs+1}
		\begin{split}
			&\lambda^{-\dot{J}-1}\int_{\mathbb{R}^{nk}}e^{\mathrm{i}\lambda X}\tilde{V}\tilde{r}|r_{i_0}|^{-1}\|F_{\dot{s}+1}\|^{-(\dot{p}_{\dot{s}+1}+2)}\prod_{i<\dot{s}+1}\|F_i\|^{-{\dot{p}_i}}\tilde{g}\tilde{\psi}Q_1 \d y_1\cdots \d y_k\\
			=&C\lambda^{-\dot{J}-1}\int_{\mathbb{R}^{nk}}e^{\mathrm{i}\lambda X}|r_{i_0}|^{-(n-2-l_{i_0}+\dot{d}_{i_0}+1)}\prod_{i\in\mathbb{K}}V^{(\dot{\alpha}_i)}(x_i)\prod_{i\in\mathbb{K}_0\setminus\{i_0\}}|r_i|^{-(n-2-l_i+\dot{d}_i)}\\
			&\quad\quad\quad\quad\times\|F_{\dot{s}+1}\|^{-(\dot{p}_{\dot{s}+1}+2)}\prod_{i<\dot{s}+1}\|F_i\|^{-{\dot{p}_i}}\tilde{g}\tilde{\psi}Q_1 \d x_1\cdots \d x_k,
		\end{split}
	\end{equation}
	and	
	\begin{equation}\label{IIIcom}
		\begin{split}
			\lambda^{-\dot{J}-1}\int_{\mathbb{R}^{nk}}&e^{\mathrm{i}\lambda X}\tilde{V}\tilde{r}|r_{i_0}|^{-1}\|F_{\dot{s}+1}\|^{-(\dot{p}_{\dot{s}+1}+1)}\|F_{j_0}\|^{-(\dot{p}_{j_0}+1)}\prod_{\substack{i<\dot{s}+1\\i\neq j_0}}\|F_i\|^{-{\dot{p}_i}}\tilde{g}\tilde{\psi}Q_2 \d y_1\cdots \d y_k\\
			=C\lambda^{-\dot{J}-1}\int_{\mathbb{R}^{nk}}&e^{\mathrm{i}\lambda X}|r_{i_0}|^{-(n-2-l_{i_0}+\dot{d}_{i_0}+1)}\prod_{i\in\mathbb{K}}V^{(\dot{\alpha}_i)}(x_i)\prod_{i\in\mathbb{K}_0\setminus\{i_0\}}|r_i|^{-(n-2-l_i+\dot{d}_i)}\\
			&\times\|F_{\dot{s}+1}\|^{-(\dot{p}_{\dot{s}+1}+1)}\|F_{j_0}\|^{-(\dot{p}_{j_0}+1)}\prod_{\substack{i<\dot{s}+1\\i\neq j_0}}\|F_i\|^{-{\dot{p}_i}}\tilde{g}\tilde{\psi}Q_2 \d x_1\cdots \d x_k,
		\end{split}
	\end{equation}
	where $i_0\in I_{F_{\dot{s}+1}}^2$, $Q_1$ is a polynomial in $(r_{i_0}/|r_{i_0}|,E_{j_1,j_2}/\|F_{\dot{s}+1}\|)$ for some $E_{j_1,j_2}\in F_{\dot{s}+1}$, $1\leq j_0\leq\dot{s}$, and $Q_2$ is a polynomial in $(r_{i_0}/|r_{i_0}|,E_{j_1,j_2}/\|F_{\dot{s}+1}\|,E_{j_1',j_2'}/\|F_{j_0}\|)$ for some $E_{j_1,j_2}\in F_{\dot{s}+1}$ and some $E_{j_1',j_2'}\in F_{j_0}$. If $\dot{s}=0$, we then only have terms like \eqref{nabFs+1} in the combination.
	
	The fourth type of integrals derived from \eqref{intonce} is
	\begin{equation*}
		\mathrm{IV}=\lambda^{-\dot{J}-1}\int_{\mathbb{R}^{nk}}e^{\mathrm{i}\lambda X}\tilde{g}\tilde{\psi}(\Delta_{y_{I_{\dot{s}}^*}}X)\tilde{V}\tilde{r}\|F_{\dot{s}+1}\|^{-(\dot{p}_{\dot{s}+1}+2)}\prod_{i=1}^{\dot{s}}\|F_i\|^{-{\dot{p}_i}}\d y_1\cdots \d y_k.
	\end{equation*}
	If $i\in I_{\dot{s}}^*$, it follows by \eqref{nablaX*} and \eqref{nablaX} that
	\begin{equation*}
		\begin{split}
			\Delta_{y_i}X=&-\nabla_{y_i}\cdot E_{L(\mathbb{K}_0\setminus I_{\dot{s},2}^*),i}=\nabla_{y_i}\cdot\left(\frac{\mbox{$y_i-\sum_{\tau=i+1}^{i^*}y_\tau$}}{|\mbox{$y_i-\sum_{\tau=i+1}^{i^*}y_\tau$}|}-\frac{y_{L(\mathbb{K}_0\setminus I_{\dot{s},2}^*,i)}-\mbox{$\sum_{\tau=L(\mathbb{K}_0\setminus I_{\dot{s},2}^*,i)+1}^{i}y_\tau$}}{|y_{L(\mathbb{K}_0\setminus I_{\dot{s},2}^*,i)}-\mbox{$\sum_{\tau=L(\mathbb{K}_0\setminus I_{\dot{s},2}^*,i)+1}^{i}y_\tau$}|}\right)\\
			=&(n-1)\left(|r_{L(\mathbb{K}_0\setminus I_{\dot{s},2}^*,i)}|^{-1}+|r_i|^{-1}\right),
		\end{split}
	\end{equation*}
	so we conclude by $I_{F_{\dot{s}+1}}^2=\cup_{i\in I_{\dot{s}}^*}\{L(\mathbb{K}_0\setminus I_{\dot{s},2}^*,i),i\}$ again that $\mathrm{IV}$ is a finite linear combination of
	\begin{equation}\label{IVcom}
		\begin{split}
			\lambda^{-\dot{J}-1}\int_{\mathbb{R}^{nk}}&e^{\mathrm{i}\lambda X}|r_{i_0}|^{-1}\tilde{V}\tilde{r}\|F_{\dot{s}+1}\|^{-(\dot{p}_{\dot{s}+1}+2)}\prod_{i<\dot{s}+1}\|F_i\|^{-{\dot{p}_i}}\tilde{g}\tilde{\psi} \d y_1\cdots \d y_k\\
			=C\lambda^{-\dot{J}-1}\int_{\mathbb{R}^{nk}}&e^{\mathrm{i}\lambda X}|r_{i_0}|^{-(n-2-l_{i_0}+\dot{d}_{i_0}+1)}\prod_{i\in\mathbb{K}}V^{(\dot{\alpha}_i)}(x_i)\prod_{i\in\mathbb{K}_0\setminus\{i_0\}}|r_i|^{-(n-2-l_i+\dot{d}_i)}\\
			&\times\|F_{\dot{s}+1}\|^{-(\dot{p}_{\dot{s}+1}+2)}\prod_{i<\dot{s}+1}\|F_i\|^{-{\dot{p}_i}}\tilde{g}\tilde{\psi} \d x_1\cdots \d x_k.
		\end{split}
	\end{equation}
	where $i_0\in I_{F_{\dot{s}+1}}^2$.
	
	The fifth type of integrals derived from \eqref{intonce} is
	\begin{equation*}
		\mathrm{V}=\lambda^{-\dot{J}-1}\int_{\mathbb{R}^{nk}}e^{\mathrm{i}\lambda X}\tilde{V}\tilde{r}\|F_{\dot{s}+1}\|^{-(\dot{p}_{\dot{s}+1}+1)}\prod_{i<\dot{s}+1}\|F_i\|^{-{\dot{p}_i}}\left(\nabla_{y_{I_{\dot{s}}^*}}(\tilde{g}\tilde{\psi})\cdot\frac{\nabla_{y_{I_{\dot{s}}^*}}X}{\|F_{\dot{s}+1}\|}\right)\d y_1\cdots \d y_k.
	\end{equation*}
	By the properties presumed for $\dot{g}$ and $\dot{\psi}$, with a mixture of the arguments for $\mathrm{II}$ and $\mathrm{III}$, we conclude without more details that $\mathrm{V}$ is a finite linear combination of the form \eqref{dotJ+1} with properties that come along.
	
	Now \eqref{ddotp} when $i=\dot{s}+1$ and the alternatives \eqref{alt1}, \eqref{alt2} are obviously seen from \eqref{Icom}, \eqref{IIcom}, \eqref{nabFs+1}, \eqref{IIIcom} ,\eqref{IVcom}, the discussion for $\mathrm{V}$, and \eqref{sumdotp}, while the other cases in \eqref{ddotp} is a result of \eqref{sumdotp} and the inclusions in \eqref{seqalpha}.
\end{proof}

Now we are ready to show Proposition \ref{stepmuprop}.

\begin{proof}[Proof of Proposition \ref{stepmuprop}]
	The proof will be an induction on $\mu$ by repeatedly using Lemma \ref{lmintonce}. To satisfy the condition that $\dot{g}$ being supported away from the origin in each of $r_1,\cdots,r_k,F_1,\cdots,F_{\dot{s}+1}$, we in principle should first introduce cutoffs in $|r_i|/\epsilon$ and $\|F_i\|/\epsilon$ whose derivatives have the same type of bounds in the discussion for $\mathrm{II}$ and $\mathrm{III}$ in the proof of Lemma \ref{lmintonce}, and let $\epsilon\rightarrow0$ when such application comes to an end. The convergence is actually a result of the fact that \eqref{stepmu} is absolutely convergent for $\mu=k-k_0$, which is a consequence of Proposition \ref{intmainest} studied later. To avoid distraction, we will pretend that such condition on $\dot{g}$ has been satisfied in the following application.
	
	We first prove the statement for $\mu=1$. Note that Lemma \ref{lmintonce} is first applicable to $U^l$, that is $\dot{s}=0$, $|\dot{\alpha}_i|\equiv0$, $\dot{d}_i\equiv0$, $\dot{p}_1=0$, $\dot{g}=f$, $\dot{\psi}=\phi$, $\dot{P}\equiv1$, $I_{0,1}^*=\emptyset$,
	\begin{equation}\label{I_02intitial}
		I_{0,2}^*=\{i\in\mathbb{K}_0;l_i+2-2m\leq0\}=\{i\in\mathbb{K}_0;0=\max\{0,l_i+2-2m\}\},
	\end{equation}
	and therefore $I_0^*=\{i\in\mathbb{K};l_i+2-2m>0\}\neq\emptyset$ by the initial assumption \eqref{assumptionl}. Inductively, after finitely many times of applying Lemma \ref{lmintonce} whenever applicable to terms in the combination, we must end up with the fact that $U^l$ is a finite linear combination of integrals in the form of
	\begin{equation}\label{J'}
		\begin{split}
			\lambda^{-\check{J}}\int_{\mathbb{R}^{nk}}&e^{\mathrm{i}\lambda X}\prod_{i\in\mathbb{K}}V^{(\check{\alpha}_i)}(x_i)\prod_{i\in\mathbb{K}_0}|r_i|^{-(n-2-l_i+\check{d}_i)}\|F_1\|^{-{\check{p}_1}}\\
			&\times\check{g}(\lambda,r_0,\cdots,r_k,F_1)\check{\psi}(X/T)\d x_1\cdots \d x_k,
		\end{split}
	\end{equation}
	satisfying
	\begin{equation*}
		\begin{cases}
			\check{J}=|\check{\alpha}_1|+\cdots+|\check{\alpha}_k|+\check{d}_0+\cdots+\check{d}_k,\\
			|\check{\alpha}_i|=0,\quad&i\in\mathbb{K}\setminus I_{F_1}^1,\\
			\check{d}_i=0,&i\in\mathbb{K}_0\setminus I_{F_1}^2,\\
			\check{p}_1\leq\sum_{j\in I_{F_1}^1}|\check{\alpha}_j|+\sum_{j\in I_{F_1}^2}2\check{d}_j,
		\end{cases}
	\end{equation*}
	and the subordinate parameters are subject to three cases:
	\begin{itemize}
		\item[\emph{Case 1}] \textit{There exists $i_0\in I_{F_1}^1$ such that $|\check{\alpha}_{i_0}|=\frac{n+1}{2}-2m$ and}
		\begin{equation*}
			\begin{cases}
				|\check{\alpha}_i|<\frac{n+1}{2}-2m,\quad&i\in I_{F_1}^1\setminus\{i_0\},\\
				0\leq \check{d}_i<\max\{0,l_i+2-2m\},&i\in I_{F_1}^2,\\
				L_{k_0}\leq\frac{n+1}{2}-2m\leq\check{J}\leq L_{k_0}+\cdots+L_{k-1}.
			\end{cases}
		\end{equation*}
		In this case, we define $I_{1,1}^*=\{i_0\}$, $I_{1,2}^*=I_{0,2}^*$. It is obvious that
		\begin{equation}\label{I11}
			I_{1,1}^*=\{i\in\mathbb{K};|\check{\alpha}_i|=\mbox{$\frac{n+2}{2}-2m$}\},
		\end{equation}
		It also follows by \eqref{defI_F12'} that
		\begin{equation}\label{IF12inclusion}
			I_{F_1}^2\subset\mathbb{K}_0\setminus I_{0,2}^*=\{i\in\mathbb{K}_0;l_i+2-2m>0\},
		\end{equation}
		which combining with \eqref{I_02intitial} implies
		\begin{equation}\label{I12}
			I_{1,2}^*=\{i\in\mathbb{K}_0;\check{d}_i=\max\{0,l_i+2-2m\}\}.
		\end{equation}
		
		\item[\emph{Case 2}] \textit{There exists $i_0\in I_{F_1}^2$ such that $0<\check{d}_{i_0}=\max\{0,l_{i_0}+2-2m\}$ and}
		\begin{equation*}
			\begin{cases}
				|\check{\alpha}_i|<\frac{n+1}{2}-2m,\quad&i\in I_{F_1}^1,\\
				0\leq\check{d}_i<\max\{0,l_i+2-2m\},&i\in I_{F_1}^2\setminus\{i_0\},\\
				L_{k_0}\leq\max\{0,l_{i_0}+2-2m\}\leq\check{J}\leq L_{k_0}+\cdots+L_{k-1}.
			\end{cases}
		\end{equation*}
		In this case, we define $I_{1,1}^*=I_{0,1}^*=\emptyset$ and $I_{1,2}^*=I_{0,2}^*\cup\{i_0\}$. Now \eqref{IF12inclusion} implies $i_0\notin I_{0,2}^*$ and \eqref{I12} holds in a similar way, while \eqref{I11} holds trivially.
		
		\item[\emph{Case 3}] \textit{It follows that}
		\begin{equation*}
			\begin{cases}
				|\check{\alpha}_i|<\frac{n+1}{2}-2m,\quad&i\in I_{F_1}^1,\\
				0\leq\check{d}_i<\max\{0,l_i+2-2m\},&i\in I_{F_1}^2,\\
				\check{J}=L_{k_0}+\cdots+L_{k-1}.
			\end{cases}
		\end{equation*}
		In this case, we just need to define $I_{1,1}^*=I_{0,1}^*$ and $I_{1,2}^*=I_{0,2}^*$.
	\end{itemize}
	In all three cases, \eqref{I_i^*I_F^} holds for $i=1$ by definition. Now all constraints are checked for $\mu=1$ in the statement if we equip integral \eqref{J'} with the sequences $I_{0,1}^*\subset I_{1,1}^*$ and $I_{0,2}^*\subset I_{1,2}^*$ defined respectively in the above three cases.
	
	By induction, we now suppose $k-k_0\geq2$ and validity of the statement for some $\mu\in\{1,\cdots,k-k_0-1\}$. First note that for every integral \eqref{stepmu} equipped with sequences \eqref{musequences} in the combination, if $J=L_{k_0}+\cdots+L_{k-1}$, it is then trivial that all subordinate constraints remain true with $\mu$ replaced by $\mu+1$.
	
	So discussion is only needed when $s=\mu$ and $J<L_{k_0}+\cdots+L_{k-1}$, while we recall that $\tau^{(\mu)}$ in \eqref{stepmuconstraint2} must exist. Now Lemma \ref{lmintonce} is first applicable to \eqref{stepmu}, that is, $\dot{s}=\mu$,
	\begin{equation*}
		\begin{cases}
			\dot{J}=J,\,\dot{p}_{\mu+1}=0,\\
			\dot{\alpha}_i=\alpha_i,\quad&i\in\mathbb{K},\\
			\dot{d}_i=d_i,&i\in\mathbb{K}_0\\
			\dot{p}_i=p_i,&i=0,\cdots,\mu,\\
		\end{cases}
	\end{equation*}
	sequences \eqref{intonceseq} given by \eqref{musequences}, $\dot{g}=g$, $\dot{\psi}=\psi$ and $\dot{P}=P$. This is because $I_0^*,\cdots,I_\mu^*$ are nonempty by Remark \ref{stepmuproprk1}, the inclusions in the first line of \eqref{seqalpha} are guaranteed by Lemma \ref{lmIF}, we also know by Lemma \ref{lmI_FI^*} that
	\begin{equation*}
		\begin{cases}
			I_{F_{\mu+1}}^1\subset\mathbb{K}\setminus I_{\mu,1}^*=\{i\in\mathbb{K};\,|\alpha_i|<\frac{n+1}{2}-2m\},\\
			I_{F_{\mu+1}}^2\subset\mathbb{K}_0\setminus I_{\mu,2}^*=\{i\in\mathbb{K}_0;\,0\leq d_i<\max\{0,l_i+2-2m\}\},
		\end{cases}
	\end{equation*}
	which checks the last two lines of \eqref{seqalpha}, and \eqref{sumdotp} trivially holds. Inductively using Lemma \ref{lmintonce}, we end up with the fact that integral \eqref{stepmu} is a finite linear combination of integrals in the form of
	\begin{equation}\label{tildeJ}
		\begin{split}
			\lambda^{-\tilde{J}}\int_{\mathbb{R}^{nk}}&e^{\mathrm{i}\lambda X}\prod_{i\in\mathbb{K}}V^{(\tilde{\alpha}_i)}(x_i)\prod_{i\in\mathbb{K}_0}|r_i|^{-(n-2-l_i+\tilde{d}_i)}\prod_{i=1}^{\mu+1}\|F_i\|^{-{\tilde{p}_i}}\\
			&\times\tilde{g}(\lambda,r_0,\cdots,r_k,F_1,\cdots,F_{\mu+1})\tilde{\psi}(X/T)\d x_1\cdots \d x_k,
		\end{split}
	\end{equation}
	satisfying
	\begin{equation}\label{tildeJconstraint}
		\begin{cases}
			\tilde{J}=|\tilde{\alpha}_1|+\cdots+|\tilde{\alpha}_k|+\tilde{d}_0+\cdots+\tilde{d}_k,\\
			\tilde{\alpha}_i=\alpha_i,\quad&i\in\mathbb{K}\setminus I_{F_{\mu+1}}^1,\\
			\tilde{d}_i=d_i,&i\in\mathbb{K}_0\setminus I_{F_{\mu+1}}^2,\\
			\tilde{p}_i+\cdots+\tilde{p}_{s+1}\leq\sum_{j\in I_{F_i}^1}|\tilde{\alpha}_j|+\sum_{j\in I_{F_i}^2}2\tilde{d}_j,&i=1,\cdots,\mu+1,
		\end{cases}
	\end{equation}
	and the subordinate parameters are subject to three cases:
	\begin{itemize}
		\item[\emph{Case 1'}] \textit{There exists $i_0\in I_{F_{\mu+1}}^1$ such that $|\tilde{\alpha}_{i_0}|=\frac{n+1}{2}-2m$ and}
		\begin{equation*}\label{dotJcase1}
			\begin{cases}
				|\tilde{\alpha}_i|<\frac{n+1}{2}-2m,\quad&i\in I_{F_{\mu+1}}^1\setminus\{i_0\},\\
				0\leq\tilde{d}_i<\max\{0,l_i+2-2m\},&i\in I_{F_{\mu+1}}^2,\\
				L_{k_0}+\cdots+L_{k_0+\mu}\leq\tilde{J}\leq L_{k_0}+\cdots+L_{k-1}.
			\end{cases}
		\end{equation*}
		In this case, we define $I_{\mu+1,1}^*=I_{\mu,1}^*\cup\{i_0\}$ and $I_{\mu+1,2}^*=I_{\mu,2}^*$.
		
		\item[\emph{Case 2'}] \textit{There exists $i_0\in I_{F_{\mu+1}}^2$ such that $0<\tilde{d}_{i_0}=\max\{0,l_{i_0}+2-2m\}$}
		\begin{equation*}\label{dotJcase2}
			\begin{cases}
				|\tilde{\alpha}_i|<\frac{n+1}{2}-2m,\quad&i\in I_{F_{\mu+1}}^1,\\
				0\leq\tilde{d}_i<\max\{0,l_i+2-2m\},&i\in I_{F_{\mu+1}}^2\setminus\{i_0\},\\
				L_{k_0}+\cdots+L_{k_0+\mu}\leq\tilde{J}\leq L_{k_0}+\cdots+L_{k-1}.
			\end{cases}
		\end{equation*}
		In this case, we define $I_{\mu+1,1}^*=I_{\mu,1}^*$ and $I_{\mu+1,2}^*=I_{\mu,2}^*\cup\{i_0\}$.
		
		\item[\emph{Case 3'}] \textit{It follows that}
		\begin{equation*}
			\begin{cases}
				|\tilde{\alpha}_i|<\frac{n+1}{2}-2m,\quad&i\in I_{F_{\mu+1}}^1,\\
				0\leq\tilde{d}_i<\max\{0,l_i+2-2m\},&i\in I_{F_{\mu+1}}^2,\\
				\tilde{J}=L_{k_0}+\cdots+L_{k-1}.
			\end{cases}
		\end{equation*}
		In this case, we define $I_{\mu+1,1}^*=I_{\mu,1}^*$ and $I_{\mu+1,2}^*=I_{\mu,2}^*$.
	\end{itemize}
	The remaining discussion is parallel to those from Case 1 to Case 3 previously. That $i_0\notin I_{\mu,1}^*$ in Case 1' and that $i_0\notin I_{\mu,2}^*$ in Case 2' are implied by Lemma \ref{lmI_FI^*}. We also conclude in all three cases that \eqref{I_i^*I_F^} holds for $i=\mu+1$ by definition, and that
	\begin{equation*}
		\begin{cases}
			I_{\mu+1,1}^*=\{i\in\mathbb{K};\,|\tilde{\alpha}_i|=\mbox{$\frac{n+1}{2}-2m$}\},\\
			I_{\mu+1,2}^*=\{i\in\mathbb{K}_0;\,\tilde{d}_i=\max\{0,l_i+2-2m\}\},
		\end{cases}
	\end{equation*}
	which follows by \eqref{tildeJconstraint}, the definitions of $I_{\mu,1}^*$ and $I_{\mu,2}^*$, and the consequence of Lemma \ref{lmI_FI^*} saying
	\begin{equation*}
		I_{\mu,1}^*\subset\mathbb{K}\setminus I_{F_{\mu+1}}^1,\quad I_{\mu,2}^*\subset\mathbb{K}_0\setminus I_{F_{\mu+1}}^2.
	\end{equation*}
	Now integral \eqref{tildeJ} satisfies all constraints in the statement for $\mu+1$ if we equip it with the sequences $I_{0,1}^*\subset\cdots\subset I_{\mu,1}^*\subset I_{\mu+1,1}^*$ and $I_{0,2}^*\subset\cdots\subset I_{\mu,2}^*\subset I_{\mu+1,2}^*$ defined respectively above, and the proof is complete.
\end{proof}

\subsection{Reduction of line singularities}\label{section5.3}\

We now turn to a reduction of the family of $\|F_i\|$ in the integral \eqref{stepmu}, and the main result in this part is Proposition \ref{lmlinereduction}. The key property of $F_1,\cdots,F_s$ we observe is that the subscripts of their elements have a nested pattern described by the following.

\begin{definition}\label{defadmissible}
	We call subset $A$ of $\mathbb{N}_0\times\mathbb{N}_0$ is admissible, if
	\begin{itemize}
		\item [1)] $(i,j)\in A$ implies $i<j$.
		
		\item [2)] $(i_1,j_1),(i_2,j_2)\in A$ and $i_1<j_2<j_1$ imply $i_1\leq i_2$.
	\end{itemize}
\end{definition}

If integral \eqref{stepmu} in Proposition \ref{stepmuprop} is given with $s\geq2$, then $\tau^{(1)},\cdots,\tau^{(s-1)}$ defined in \eqref{stepmuconstraint2} must exist, and it has been shown in \eqref{Ii-1-Ii} that
\begin{equation*}
	I_{i-1}^*\setminus I_i^*=\{j^{(i)}\}=\{N(I_{i-1}^*,\tau^{(i)})\},\quad i=1,\cdots,s-1,
\end{equation*}
We define for $j\in\mathbb{K}$ that
\begin{equation}\label{defh_j}
	\begin{split}
		\iota_j&=\begin{cases}
			j-1,&j\in\mathbb{K}\setminus I_0^*,\\
			L(\mathbb{K}_0\setminus I_{i-1,2}^*,j),\quad&j\in I_{i-1}^*\setminus I_i^*,\,i=1,\cdots,s-1,\\
			L(\mathbb{K}_0\setminus I_{s-1,2}^*,j),&j\in I_{s-1}^*,
		\end{cases}\\
		&=\begin{cases}
			j-1,&j\in\mathbb{K}\setminus I_0^*,\\
			L(\mathbb{K}_0\setminus I_{i-1,2}^*,j^{(i)}),\quad&j=j^{(i)},\,i=1,\cdots,s-1,\\
			L(\mathbb{K}_0\setminus I_{s-1,2}^*,j),&j\in I_{s-1}^*.
		\end{cases}
	\end{split}
\end{equation}
It immediately follows that $\iota_j<j$ and
\begin{equation}\label{EFiota}
	\begin{split}
		E_{\iota_{j^{(i)}},j^{(i)}}\in F_i,\quad&i=1,\cdots,s-1,\\
		E_{\iota_j,j}\in F_s,\quad&j\in I_{s-1}^*.
	\end{split}
\end{equation}

We note in advance that the following property is crucial for Proposition \ref{intmainest} in the next section.
\begin{proposition}\label{lmadmissible}
	If integral \eqref{stepmu} in Proposition \ref{stepmuprop} is given with $s\geq2$, then $\{(\iota_j,j);\,j\in\mathbb{K}\}$ is admissible.	
\end{proposition}

\begin{proof}
	Suppose $\iota_j<j'<j$ for some $j\in\mathbb{K}$. We first note that $j\in I_0^*$ must hold because $\iota_j\leq j-2$, and the conclusion is obviously trivial when $j'\in\mathbb{K}_0\setminus I_0^*$.
	
	If $j'\in I_0^*$ and $\iota_j=L(\mathbb{K}_0\setminus I_{i-1,2}^*,j)$ for some $i=1,\cdots,s$, we know by $\iota_j<j'<j$ that $j'\notin\mathbb{K}_0\setminus I_{i-1,2}^*$. Since $I_{i-1}^*\subset\mathbb{K}_0\setminus I_{i-1,2}^*$, then $j'\notin I_{i-1}^*$ holds and thus $j'\in I_0^*\setminus I_{i-1}^*$ follows. Therefore, there exists $i'\in\mathbb{K}$ with $1\leq i'\leq i-1$ such that $I_{i'-1}\setminus I_{i'}^*=\{j'\}$, and consequently
	\begin{equation*}
		\iota_{j'}=L(\mathbb{K}_0\setminus I_{i'-1,2}^*,j')\geq L(\mathbb{K}_0\setminus I_{i-1,2}^*,j')=L(\mathbb{K}_0\setminus I_{i-1,2}^*,j)=\iota_j,
	\end{equation*}
	where the inequality is a result of $I_{i'-1,2}^*\subset I_{i-1,2}^*$, and the second equality is a conclusion of $\iota_j<j'<j$ again.
\end{proof}

\begin{lemma}\label{lmFp}
	Suppose integral \eqref{stepmu} in Proposition \ref{stepmuprop} is given with $s\geq2$, and
	\begin{equation}\label{p<tildep}
		p_i+\cdots+p_s\leq \tilde{p}_i+\cdots+\tilde{p}_s,\quad i=1,\cdots,s,
	\end{equation}
	where $\tilde{p}_1,\cdots,\tilde{p}_s\geq0$. Then
	\begin{equation*}
		\|F_1\|^{-p_1}\cdots\|F_s\|^{-p_s}\lesssim\|F_1\|^{-\tilde{p}_1}\cdots\|F_s\|^{-\tilde{p}_s}.
	\end{equation*}
\end{lemma}

\begin{proof}
	Lemma \ref{lmIF} implies that $1\gtrsim\|F_1\|\geq\cdots\geq\|F_s\|$, so
	\begin{equation*}
		\begin{split}
			\|F_1\|^{-p_1}\cdots\|F_s\|^{-p_s}=&\|F_1\|^{-\tilde{p}_1}\|F_1\|^{-(p_1-\tilde{p}_1)}\|F_2\|^{-p_2}\cdots\|F_s\|^{-p_s}\\
			\lesssim&\|F_1\|^{-\tilde{p}_1}\times\begin{cases}
				\|F_2\|^{-p_2}\cdots\|F_s\|^{-p_s},\quad&\text{if}\,p_1\leq\tilde{p}_1,\\
				\|F_2\|^{-(p_1+p_2-\tilde{p}_1)}\cdots\|F_s\|^{-p_s},&\text{if}\,p_1>\tilde{p}_1,
			\end{cases}
		\end{split}
	\end{equation*}
	and the conclusion follows by an induction on $s$ if we use \eqref{p<tildep}.
\end{proof}

\begin{lemma}\label{lmJlower}
	If integral \eqref{stepmu} in Proposition \ref{stepmuprop} is given with $s\geq2$, it follows that
	\begin{equation}\label{Jlowerbd2}
		\sum_{j\in I_{F_i}^1}|\alpha_j|+\sum_{j\in I_{F_i}^2}d_j\leq J-(L_{k_0}+\cdots+L_{k_0+i-2}),\quad i=2,\cdots,s.
	\end{equation}
\end{lemma}

\begin{proof}
	The inclusions \eqref{musequences} imply
	\begin{equation*}\label{alphaL}
		|\alpha_j|=\mbox{$\frac{n+1}{2}-2m$}\geq L_\tau,\quad j\in I_{i-1,1}^*,
	\end{equation*}
	for any $k_0\leq\tau\leq k$ because $l_\tau\leq\frac{n-3}{2}$, and also imply
	\begin{equation*}\label{d_j>0}
		d_j=L_{\sigma^{-1}(j)}>0,\quad j\in I_{i-1,2}^*\setminus I_{0,2}^*,
	\end{equation*}
	because $d_j=\max\{0,l_j+2-2m\}=l_j+2-2m>0$ must hold. Since the constraint 2) in Proposition \ref{stepmuprop} also implies $\#I_{i-1,1}^*+\#(I_{i-1,2}^*\setminus I_{0,2}^*)=i-1$, so
	\begin{equation*}
		\begin{split}
			J\geq&\left(\sum_{j\in I_{i-1,1}^*}|\alpha_j|+\sum_{j\in I_{i-1,2}^*\setminus I_{0,2}^*}d_j\right)+\left(\sum_{j\in\mathbb{K}\setminus I_{i-1,1}^*}|\alpha_j|+\sum_{j\in\mathbb{K}_0\setminus I_{i-1,2}^*}d_j\right)\\
			\geq&\left(L_{k_0}+\cdots+L_{k_0+i-2}\right)+\left(\sum_{j\in I_{F_i}^1}|\alpha_j|+\sum_{j\in I_{F_i}^2}d_j\right),
		\end{split}
	\end{equation*}
	and the last line is a consequence of Lemma \ref{lmI_FI^*}.
\end{proof}

\begin{proposition}\label{lmlinereduction}
	Suppose integral \eqref{stepmu} in Proposition \ref{stepmuprop} is given, and $\{E_{\iota_j,j}\}_{j\in\mathbb{K}}$ is defined through \eqref{defh_j}.
	\begin{itemize}
		\item [1)] If $s\geq2$ and $I_{s-1,1}^*\neq\emptyset$, which guarantees the existence of $i_0\in\{1,\cdots,s-1\}$ with $I_{i_0,1}^*\setminus I_{i_0-1,1}^*=\{\tau^{(i_0)}\}$ and $I_{i_0,2}^*=I_{i_0-1,2}^*$, then
		\begin{equation}\label{FE1}
			\|F_1\|^{-p_1}\cdots\|F_s\|^{-p_s}\lesssim|E_{\iota_{j^{(i_0)}},j^{(i_0)}}|^{-(\frac{n+1}{2}-2m)}\prod_{j\in\mathbb{K}\setminus\{j^{(i_0)}\}}|E_{\iota_j,j}|^{-(n+1-4m)}.
		\end{equation}
		If $j^{(i_0)}<k$, then $j^{(i_0)}\leq\iota_j$ holds for all $j\in\{j^{(i_0)}+1,\cdots,k\}$.
		
		\item [2)] If $s\geq2$ and $I_{s-1,1}^*=\emptyset$, then
		\begin{equation*}
			\|F_1\|^{-p_1}\cdots\|F_s\|^{-p_s}\lesssim|E_{\iota_k,k}|^{-\min\{n+1-4m,\frac{n+1}{2}-2m+d_0+d_k\}}\prod_{j\in\mathbb{K}\setminus\{k\}}|E_{\iota_j,j}|^{-(n+1-4m)},
		\end{equation*}
		which also holds if $s=1$.
	\end{itemize}
\end{proposition}

\begin{proof}
	First note that \eqref{EFiota} implies
	\begin{equation}\label{Fj(i)}
		\begin{split}
			\|F_i\|\geq|E_{\iota_{j^{(i)}},j^{(i)}}|,\quad&i=1,\cdots,s-1,\\
			\|F_s\|\geq|E_{\iota_j,j}|,\quad&j\in I_{s-1}^*.
		\end{split}
	\end{equation}
	
	We first show 1). If $1\leq i<i_0$, since $I_{i_0,1}^*\setminus I_{i_0-1,1}^*=\{\tau^{(i_0)}\}$ implies $\tau^{(i_0)}\in I_{F_{i_0}}^1\subset I_{F_i}^1$ by \eqref{I_i^*I_F^} and Lemma \ref{lmIF}, we deduce
	\begin{equation*}
		\begin{split}
			p_i+\cdots+p_s=&2\left(\sum_{j\in I_{F_i}^1}|\alpha_j|+\sum_{j\in I_{F_i}^2}d_j\right)-\sum_{j\in I_{F_i}^1}|\alpha_j|\\
			\leq&\begin{cases}
				2J-|\alpha_{i_0}|,\quad&i=1,\\
				2(J-L_{k_0}-\cdots-L_{k_0+i-2})-|\alpha_{i_0}|,&1<i<i_0,
			\end{cases}\\
			\leq&2(L_{k_0+i-1}+\cdots+L_{k-1})-(\mbox{$\frac{n+1}{2}-2m$})\\
			\leq&(n+1-4m)(k-k_0-i)+(\mbox{$\frac{n+1}{2}-2m$}),
		\end{split}
	\end{equation*}
	where we have used \eqref{stepmuconstraint4}, \eqref{Jlowerbd2}, the last inequality in \eqref{stepmuconstraint1}, and the fact that $|\alpha_{i_0}|=\frac{n+1}{2}-2m$ for $\tau^{(i_0)}\in I_{i_0,1}^*\subset I_{s,1}^*$. If $i_0\leq i\leq s$, we always have by \eqref{Jlowerbd2} that
	\begin{equation}\label{trivialcontrol}
		\begin{split}
			p_i+\cdots+p_s\leq&2\left(\sum_{j\in I_{F_i}^1}|\alpha_j|+\sum_{j\in I_{F_i}^2}d_j\right)\\
			\leq&\begin{cases}
				2J,\quad&i=1,\\
				2(J-L_{k_0}-\cdots-L_{k_0+i-2}),&1<i\leq s,
			\end{cases}\\
			\leq&2(L_{k_0+i-1}+\cdots+L_{k-1})\\
			\leq&(n+1-4m)(k-k_0-i+1).
		\end{split}
	\end{equation}
	Applying Lemma \ref{lmFp} with
	\begin{equation*}
		\tilde{p}_i=\begin{cases}
			n+1-4m,\quad&i\notin\{i_0,s\},\\
			\frac{n+1}{2}-2m,&i=i_0,\\
			(n+1-4m)(k-k_0-s+1),&i=s,
		\end{cases}
	\end{equation*}
	and we get by using \eqref{Fj(i)}, $I_0^*\setminus I_{s-1}^*=\bigcup_{i=1}^{s-1}\{j^{(i)}\}$ and $\# I_{s-1}^*=k-k_0-s+1$ that
	\begin{equation}\label{Ficontrol1}
		\begin{split}
			\|F_1\|^{-p_1}\cdots\|F_s\|^{-p_s}\lesssim&\|F_{i_0}\|^{-(\frac{n+1}{2}-2m)}\left(\prod_{\substack{1\leq i\leq s-1\\i\neq i_0}}\|F_i\|^{-(n+1-4m)}\right)\|F_s\|^{-(n+1-4m)(k-k_0-s+1)}\\
			\leq&|E_{\iota_{j^{(i_0)}},j^{(i_0)}}|^{-(\frac{n+1}{2}-2m)}\prod_{\substack{1\leq i\leq s-1\\i\neq i_0}}|E_{\iota_{j^{(i)}},j^{(i)}}|^{-(n+1-4m)}\prod_{j\in I_{s-1}^*}|E_{\iota_j,j}|^{-(n+1-4m)}\\
			=&|E_{\iota_{j^{(i_0)}},j^{(i_0)}}|^{-(\frac{n+1}{2}-2m)}\prod_{j\in I_0^*\setminus\{j^{(i_0)}\}}|E_{\iota_j,j}|^{-(n+1-4m)},
		\end{split}
	\end{equation}
	which implies \eqref{FE1} for $|E_{\iota_j,j}|\lesssim1$. Now suppose $j^{(i_0)}<k$ and
	\begin{equation}\label{iotakjk}
		\iota_i<j^{(i_0)}<i,
	\end{equation}
	for some $i\in\{j^{(i_0)}+1,\cdots,k\}$, which indicates $\iota_i\leq i-2$, thus
	\begin{equation}\label{iinIi_0'-1*}
		i\in I_{i_0'-1}^*,
	\end{equation}
	and $\iota_i=L(\mathbb{K}_0\setminus I_{i_0'-1,2}^*,i)$ hold for some $i_0'\geq1$ by \eqref{defh_j}. Note that $i_0'>i_0$ must hold, because otherwise $i_0'<i_0$ for $j^{(i_0)}\neq i$, and then $j^{(i_0)}\in I_{i_0-1}^*\subset I_{i_0'-1}^*\subset\mathbb{K}_0\setminus I_{i_0'-1,2}^*$ yields the contradiction $\iota_i\geq j^{(i_0)}$ to \eqref{iotakjk}. On the other hand, we have shown at the beginning of the proof that $\tau^{(i_0)}\in I_{F_{i_0}}^1$, so there in the view of \eqref{defI_F12'} exists some $j\in I_{i_0-1}^*$ with
	\begin{equation}\label{taui_0range}
		L(\mathbb{K}_0\setminus I_{i_0-1,2}^*,j)<\tau^{(i_0)}\leq j.
	\end{equation}
	Note that $j\in I_{i_0-1}^*$ implies $N(I_{i_0-1}^*,\tau^{(i_0)})\leq j$, that is $j^{(i_0)}\leq j$. It is in fact impossible for $j^{(i_0)}<j$ to hold, otherwise, the fact by definition that $j^{(i_0)}\in I_{i_0-1}^*\subset\mathbb{K}_0\setminus I_{i_0-1,2}^*$ implies $j^{(i_0)}\leq L(\mathbb{K}_0\setminus I_{i_0-1,2}^*,j)<j$, and the contradiction $\tau^{(i_0)}\leq N(I_{i_0-1}^*,\tau^{(i_0)})=j^{(i_0)}\leq L(\mathbb{K}_0\setminus I_{i_0-1,2}^*,j)<\tau^{(i_0)}$ to \eqref{taui_0range}. Now we must have $j=j^{(i_0)}$, that is
	\begin{equation*}
		L(\mathbb{K}_0\setminus I_{i_0-1,2}^*,j^{(i_0)})<\tau^{(i_0)}\leq j^{(i_0)}.
	\end{equation*}
	This, and the fact $\{j^{(i_0)},\cdots,i-1\}\subset I_{i_0'-1,2}^*$ implied by \eqref{iotakjk}, together show $\{\tau^{(i_0)},\cdots,i-1\}\subset I_{i_0'-1,2}^*$, and consequently
	\begin{equation*}
		I_{i_0'-1}^*=D_{I_{i_0'-1,1}^*}D_{I_{i_0'-1,2}^*}\mathbb{K}\subset D_{\tau^{(i_0)}}D_{\{\tau^{(i_0)},\cdots,i-1\}}\mathbb{K},
	\end{equation*}
	which indicates the contradiction $i\notin I_{i_0'-1}^*$ to \eqref{iinIi_0'-1*}. Therefore, \eqref{iotakjk} is not true, and the proof of 1) is complete.

	To show 2), if $s\geq2$, $I_{s-1,1}^*=\emptyset$ and $\{0,k\}\cap I_{s-1,2}^*\neq\emptyset$, note that \eqref{trivialcontrol} is true for $i=1,\cdots,s$, we may apply Lemma \ref{lmFp} with
	\begin{equation*}
		\tilde{p}_i=\begin{cases}
			n+1-4m,\quad&i=1,\cdots,s-1,\\
			(k-k_0-s+1)(n+1-4m),&i=s,
		\end{cases}
	\end{equation*}
	and similar to \eqref{Ficontrol1}, we deduce
	\begin{equation*}
		\begin{split}
			\|F_1\|^{-p_1}\cdots\|F_s\|^{-p_s}\lesssim&\left(\prod_{1\leq i\leq s-1}\|F_i\|^{-(n+1-4m)}\right)\|F_s\|^{-(n+1-4m)(k-k_0-s+1)}\\
			\leq&\prod_{1\leq i\leq s-1}|E_{\iota_{j^{(i)}},j^{(i)}}|^{-(n+1-4m)}\prod_{j\in I_{s-1}^*}|E_{\iota_j,j}|^{-(n+1-4m)}\\
			\leq&\prod_{j\in\mathbb{K}}|E_{\iota_j,j}|^{-(n+1-4m)},
		\end{split}
	\end{equation*}
	and this is a stronger estimate for $n+1-4m\leq\frac{n+1}{2}-2m+d_0+d_k$ holds when $\{0,k\}\cap I_{s-1,2}^*\neq\emptyset$.
	
	If $I_{s-1,1}^*=\emptyset$ and $\{0,k\}\cap I_{s-1,2}^*=\emptyset$, then
	\begin{equation}\label{Is-1*}
		I_{s-1}^*=D_{I_{s-1,2}^*}\mathbb{K}=\mathbb{K}\setminus I_{s-1,2}^*\ni k,
	\end{equation}
	so it is easy to get from \eqref{defI_F12'} that
	\begin{equation}\label{IFs}
		\begin{split}
			I_{F_s}^1=&\bigcup_{j\in\mathbb{K}\setminus I_{s-1,2}^*}\{L(\mathbb{K}_0\setminus I_{s-1,2}^*,j)+1,\cdots,j\}=\mathbb{K},\\
			I_{F_s}^2=&\bigcup_{j\in\mathbb{K}\setminus I_{s-1,2}^*}\{L(\mathbb{K}_0\setminus I_{s-1,2}^*,j),j\}=\mathbb{K}_0\setminus I_{s-1,2}^*.
		\end{split}
	\end{equation}
	Next for $i=1,\cdots,s$,
	\begin{equation}\label{pips}
		\begin{split}
			p_i+\cdots+p_s\leq&\sum_{j\in I_{F_i}^1}|\alpha_j|+\sum_{j\in I_{F_i}^2}2d_j\\
			\leq&\sum_{j\in\mathbb{K}}|\alpha_j|+\sum_{j\in\mathbb{K}_0\setminus I_{i-1,2}^*}2d_j\\
			=&\sum_{j\in I_{F_s}^1}|\alpha_j|+\sum_{j\in\mathbb{K}_0\setminus I_{s-1,2}^*}2d_j+\sum_{j\in I_{s-1,2}^*\setminus I_{i-1,2}^*}2d_j\\
			\leq&\sum_{j\in I_{F_s}^1}|\alpha_j|+\sum_{j\in\mathbb{K}_0\setminus I_{s-1,2}^*}2d_j+(n+1-4m)(s-i),
		\end{split}
	\end{equation}
	where we have used Lemma \ref{lmI_FI^*} and the fact that $\#(I_{s-1,2}^*\setminus I_{i-1,2}^*)\leq s-i$. Further,
	\begin{equation}\label{sumcase1}
		\begin{split}
			&\sum_{j\in I_{F_s}^1}|\alpha_j|+\sum_{j\in\mathbb{K}_0\setminus I_{s-1,2}^*}2d_j=\sum_{j\in I_{F_s}^1}|\alpha_j|+\sum_{j\in I_{F_s}^2}d_j+\sum_{j\in\mathbb{K}_0\setminus I_{s-1,2}^*}d_j\\
			\leq&L_{k_0+s-1}+\cdots+L_{k-1}+\sum_{j\in\mathbb{K}_0\setminus I_{s-1,2}^*}d_j\\
			\leq&(\mbox{$\frac{n+1}{2}-2m$})(k-k_0-s+1)+\sum_{j\in\mathbb{K}_0\setminus I_{s-1,2}^*}d_j\\
			=&(\mbox{$\frac{n+1}{2}-2m$})(k-k_0-s)+\sum_{j\in\mathbb{K}\setminus(I_{s-1,2}^*\cup\{k\})}d_j+(\mbox{$\frac{n+1}{2}-2m$})+d_0+d_k\\
			\leq&(\mbox{$\frac{n+1}{2}-2m$})(k-k_0-s)+(\mbox{$\frac{n+1}{2}-2m$})\#\left(\mathbb{K}\setminus(I_{s-1,2}^*\cup\{k\})\right)+\mbox{$\frac{n+1}{2}-2m$}+d_0+d_k\\
			\leq&(n+1-4m)(k-k_0-s)+\mbox{$\frac{n+1}{2}-2m$}+d_0+d_k,
		\end{split}
	\end{equation}
	where we have used \eqref{IFs}, an application of \eqref{Jlowerbd2} that is similar to \eqref{trivialcontrol}, the assumption $\{0,k\}\cap I_{s-1,2}^*=\emptyset$, and the fact that $\#(\mathbb{K}\setminus(I_{s-1,2}^*\cup\{k\}))=k-k_0-s$ which is due to $\emptyset=I_{s-1,1}^*=\cdots=I_{0,1}^*$ and consequently $I_{s-1,2}^*=I_{0,2}^*\cup\{\tau^{(1)},\cdots,\tau^{(s-1)}\}$. On the other hand, we also have
	\begin{equation}\label{sumcase2}
		\begin{split}
			\sum_{j\in I_{F_s}^1}|\alpha_j|+\sum_{j\in\mathbb{K}_0\setminus I_{s-1,2}^*}2d_j\leq&2\left(\sum_{j\in I_{F_s}^1}|\alpha_j|+\sum_{j\in I_{F_s}^2}d_j\right)\\
			\leq&2(L_{k_0+s-1}+\cdots+L_{k-1})\\
			\leq&(n+1-4m)(k-k_0-s+1).
		\end{split}
	\end{equation}
	Combining \eqref{sumcase1}, \eqref{sumcase2} and \eqref{pips}, we derive for $i=i,\cdots,s$ that
	\begin{equation*}
		\begin{split}
			&p_i+\cdots+p_s\\
			\leq&(n+1-4m)(k-k_0-i)+\min(n+1-4m,\mbox{$\frac{n+1}{2}-2m$}+d_0+d_k),
		\end{split}
	\end{equation*}
	so we may apply Lemma \ref{lmFp} with
	\begin{equation*}
		\tilde{p}_i=\begin{cases}
			n+1-4m,\quad&i=1,\cdots,s-1,\\
			(n+1-4m)(k-k_0-s)+\min\{n+1-4m,\mbox{$\frac{n+1}{2}-2m$}+d_0+d_k\},&i=s,
		\end{cases}
	\end{equation*}
	to deduce
	\begin{equation*}
		\begin{split}
			&\|F_1\|^{-p_1}\cdots\|F_s\|^{-p_s}\\
			\lesssim&\left(\prod_{1\leq i\leq s-1}\|F_i\|^{-(n+1-4m)}\right)\|F_s\|^{-(n+1-4m)(k-k_0-s)-\min\{n+1-4m,\frac{n+1}{2}-2m+d_0+d_k\}}\\
			\leq&|E_{\iota_k,k}|^{-\min\{n+1-4m,\frac{n+1}{2}-2m+d_0+d_k\}}\prod_{1\leq i\leq s-1}|E_{\iota_{j^{(i)}},j^{(i)}}|^{-(n+1-4m)}\prod_{j\in I_{s-1}^*\setminus\{k\}}|E_{\iota_j,j}|^{-(n+1-4m)}\\
			\leq&|E_{\iota_k,k}|^{-\min\{n+1-4m,\frac{n+1}{2}-2m+d_0+d_k\}}\prod_{j\in\mathbb{K}\setminus\{k\}}|E_{\iota_j,j}|^{-(n+1-4m)},
		\end{split}
	\end{equation*}
	where in the second inequality we have used \eqref{Is-1*} and consequently $\#(I_{s-1}^*\setminus\{k\})=k-k_0-s$. Now the proof of 2) is complete.
\end{proof}

\subsection{Estimates for integrals with point and reduced line singularities}\label{section5.4}\

Since $\|F_1\|^{-p_1}\cdots\|F_s\|^{-p_s}$ in \eqref{stepmu} has been estimated in Proposition \ref{lmlinereduction} in particular forms, we now need to consider the estimates for integrals in the form of
\begin{equation}\label{Ix_0x_k}
	\begin{split}
		&I_k(x_0,x_{k+1};\beta_1,\cdots,\beta_k,a_0,\cdots,a_k,q_1,\cdots,q_k)\\
		=&\int_{\mathbb{R}^{kn}}\frac{\langle x_1\rangle^{-\beta_1-}\cdots\langle x_k\rangle^{-\beta_k-}}{|x_0-x_1|^{a_0}\cdots|x_k-x_{k+1}|^{a_k}}|E_{\eta_1,1}|^{-q_1}\cdots|E_{\eta_k,k}|^{-q_k}\d x_1\cdots \d x_k,
	\end{split}
\end{equation}
based on the three lemmas (Lemma \ref{propIkest1} to Lemma \ref{propIkest3}), which together prove the second main technical result Proposition \ref{intmainest} in Section \ref{section5}.

Before introducing these results, it is crucial to note that when there are a lot of line singularities $|E_{\eta_j,j}|^{-q_j}$ in the integral \eqref{Ix_0x_k}, the possibility of bounding such integrals  will come from further assuming $\{(\eta_j,j);\,j\in\mathbb{K}\}$ to be admissible by Definition \ref{defadmissible}, because then we are always able to choose a specific variable $x_\tau$ such that Proposition \ref{twoline} is first applicable in the integral of $x_\tau$ where at most two line singularities are relevant, and the admissibility of $\{(\eta_j,j);\,j\in\mathbb{K}\}$ will also allow such mechanism after each time of applying Proposition \ref{twoline}. The choice of such $x_\tau$ is asked to obey certain constraints for technical reasons.

\begin{proposition}\label{proad}
	Suppose that $\{(\eta_j,j);\,j\in\mathbb{K}\}$ in \eqref{Ix_0x_k} is admissible. Then we can find $\tau\in\mathbb{K}$ such that $x_{\tau}$ is independent of $E_{\eta_j,j}$ unless $j\in\{\tau-1,\tau\}$, and we also assert the following:
	
	\noindent (i) If there exists $s\in \mathbb{K}$ with $\eta_s \leq s-3$, then $\tau$ can be chosen to satisfy $\eta_s<\tau< s$.
	
	\noindent (ii) $\tau$ can always be chosen to satisfy either
	\begin{equation}\label{j_01r}
		\eta_k\leq\eta_{\tau+1}<\eta_\tau<\tau<\tau+1\leq k,
	\end{equation}
	or
	\begin{equation}\label{j_02r}
		\eta_k<\tau=k.
	\end{equation}		
\end{proposition}

\begin{proof}
	To prove (i), let
	\begin{equation}\label{maxtau}
		\tau=\max\{j\in\{1,\cdots,s-1\};~\eta_j=j-1~\text{or}~\eta_j=\eta_{j-1}\},
	\end{equation}
	which is well defined because $\eta_1=0$. If $\tau<s-1$, from the admissibility of $\{(\eta_j,j);\,j\in\mathbb{K}\}$ and that $\tau$ being the greatest, we must have
	\begin{equation}\label{tau<s-1}
		\eta_{\tau+1}<\eta_\tau<\tau<\tau+1,
	\end{equation}
	thus for $j=\tau+1,\cdots,k$, it follows from the admissibility again that either $\eta_j\geq\tau+1$ or $\eta_j\leq\eta_{\tau+1}$ holds, and consequently
	\begin{equation}\label{irrelevant}
		\eta_j\notin\{\tau-1,\tau\},\quad j=\tau+1,\cdots,k.
	\end{equation}
	If $\tau=s-1$, then \eqref{irrelevant} is also true, because the admissibility implies either $\eta_j\geq s$ or $\eta_j\leq\eta_s$, and recall that we have assumed $\eta_s\leq s-3=\tau-2$. It is now clear that
	\begin{equation}\label{2leqtauleqk-1}
		\eta_s<\tau< s,
	\end{equation}
	and $E_{\eta_j,j}$ depends on $x_\tau$ only when $j\in\{\tau-1,\tau\}$.	
	
	To prove (ii), if $\eta_k\geq k-2$, we can just choose $\tau=k$ because $\eta_j< k-2$ for each $j<k-2$. If $\eta_k\leq k-3$, we may use the first assertion with $s=k$ to choose $\tau$, and then \eqref{tau<s-1} implies either \eqref{j_01r} or \eqref{j_02r}.
\end{proof}

We now turn back to the estimates of \eqref{Ix_0x_k} in different regimes of assumption on indices.

\begin{lemma}\label{propIkest1}
	Suppose $n\geq4m+1$, $k\geq2$, $\{(\eta_j,j);\,j\in\mathbb{K}\}$ is admissible in \eqref{Ix_0x_k}, and
	\begin{equation}\label{indexcond1}
		\begin{cases}
			n-2m\leq a_i\leq n-2\,(1\leq i\leq k-1),\,\,0\leq a_0,a_k\leq n-2,\,\,a_0+a_k\geq\mbox{$\frac{n-1}{2}$},\\
			\beta_i\geq2m\,(i\in\mathbb{K}\setminus\{k\}),\,\,\beta_k\geq\mbox{$\frac{n+1}{2}$},\\
			q_i=n+1-4m\,(1\leq i\leq k-1),\,\,0\leq q_k\leq\min\{a_0+a_k,n+1-4m\}.
		\end{cases}
	\end{equation}
	Then
	\begin{equation*}
		I_k(x_0,x_{k+1};\beta_1,\cdots,\beta_k,a_0,\cdots,a_k,q_1,\cdots,q_k)\lesssim1,\quad|x_0-x_{k+1}|\gtrsim1.
	\end{equation*}
\end{lemma}

\begin{proof}
	If $k=2$, then $\{E_{\eta_1,1},E_{\eta_2},2\}=\{E_{0,1},E_{\eta_2,2}\}$, and one checks with Proposition \ref{twoline} in both situations that
	\begin{equation*}\label{I_2}
		\begin{split}
			&\int_{\mathbb{R}^{n}}\langle x_1\rangle^{-(\beta_1+\epsilon)}|r_0|^{-a_0}|r_1|^{-a_1}|E_{0,1}|^{-q_1}|E_{\eta_2,2}|^{-q_2}\d x_1\\
			=&\int_{\mathbb{R}^{n}}\frac{\langle x_1\rangle^{-\beta_1-}\d x_1}{|x_0-x_1|^{a_0}|x_1-x_2|^{a_1}|E_{x_0x_1x_1x_2}|^{n+1-4m}|E_{x_{\eta_2}x_{\eta_2+1}x_2x_3}|^{q_2}},\\
			\lesssim&|E_{x_0x_2x_2x_3}|^{-q_2}|x_0-x_2|^{-\min\{a_0,a_1\}},\quad|x_0-x_2|>0,
		\end{split}
	\end{equation*}
	and therefore
	\begin{equation*}
		I_2\lesssim\int_{\mathbb{R}^n}\frac{\langle x_2\rangle^{-\beta_2-}\d x_2}{|x_0-x_2|^{\min\{a_0,a_1\}}|x_2-x_3|^{a_2}|E_{x_0x_2x_2x_3}|^{q_2}}\lesssim1,\quad|x_0-x_3|\gtrsim1,
	\end{equation*}
	where we have used the facts that $\beta_2\geq\frac{n+1}{2}$,
	\begin{equation*}
		\min\{a_0,a_1\}+a_2+\beta_2-n\geq\min\{a_0+a_2,n-2m\}-\mbox{$\frac{n-1}{2}$}\geq0,
	\end{equation*}
	and
	\begin{equation*}
		\min\{a_0,a_1\}+a_2-q_2\geq\min\{a_0+a_2,n-2m\}-\min\{a_0+a_2,n+1-4m\}\geq0.
	\end{equation*}
	
	If $k\geq3$, we split the proof into three cases.
	
	\noindent\emph{Case 1: There exists $s\in\{3,\cdots,k\}$ with $\eta_s\leq s-3$.}

	By the first assertion of Proposition \ref{proad}, we can find $\tau\in\mathbb{K}$ with $\eta_s<\tau< s$, so that $E_{\eta_j,j}$ depends on $x_\tau$ only when $j\in\{\tau-1,\tau\}$, and then the integral with respect to $x_\tau$ reads
	\begin{equation}\label{intxtau}
		\int_{\mathbb{R}^n}\frac{\langle x_\tau\rangle^{-\beta_\tau-}\d x_\tau}{|x_{\tau-1}-x_\tau|^{a_{\tau-1}}|x_\tau-x_{\tau+1}|^{a_\tau}|E_{\eta_{\tau-1},\tau-1}|^{n+1-4m}|E_{\eta_\tau,\tau}|^{n+1-4m}}.
	\end{equation}
	If $\eta_\tau=\eta_{\tau-1}$, the triangle inequality implies
	\begin{equation}\label{Etautriangle}
		\begin{split}
			&|E_{\eta_{\tau-1},\tau-1}|^{-(n+1-4m)}|E_{\eta_\tau,\tau}|^{-(n+1-4m)}\\
			\lesssim&|E_{\eta_\tau,\tau}|^{-(n+1-4m)}|E_{\tau-1,\tau}|^{-(n+1-4m)}+|E_{\eta_{\tau-1},\tau-1}|^{-(n+1-4m)}|E_{\tau-1,\tau}|^{-(n+1-4m)},
		\end{split}
	\end{equation}
	we may thus just consider $\eta_\tau=\eta_{\tau-1}$ in the view of \eqref{maxtau}, and apply Proposition \ref{twoline} to \eqref{intxtau} and \eqref{Etautriangle}, to conclude that integral \eqref{intxtau} is bounded by
	\begin{equation*}\label{inttaubd}
		|E_{x_{\eta_{\tau-1}}x_{\eta_{\tau-1}+1}x_{\tau-1}x_{\tau+1}}|^{-(n+1-4m)}|x_{\tau-1}-x_{\tau+1}|^{-\min\{a_{\tau-1},a_\tau\}},
	\end{equation*}
	where we have used $\beta_\tau\geq2m$ which always holds. Consequently,
	\begin{equation}\label{Ik1st}
		\begin{split}
			I_k\lesssim&\int_{\mathbb{R}^{(k-1)n}}\frac{\langle x_1\rangle^{-\beta_1-}\cdots\langle x_{\tau-1}\rangle^{-\beta_{\tau-1}-}\langle x_{\tau+1}\rangle^{-\beta_{\tau+1}-}\cdots\langle x_k\rangle^{-\beta_k-}}{|r_0|^{a_0}\cdots|r_{\tau-2}|^{a_{\tau-2}}|x_{\tau-1}-x_{\tau+1}|^{\min\{a_{\tau-1},a_\tau\}}|r_{\tau+1}|^{a_{\tau+1}}\cdots|r_k|^{a_k}}\\
			&\quad\quad\times\frac{\d x_1\cdots \d x_{\tau-1}\d x_\tau\cdots \d x_k}{|E_{x_{\eta_{\tau-1}}x_{\eta_{\tau-1}+1}x_{\tau-1}x_{\tau+1}}|^{n+1-4m}\prod\limits_{j\in\mathbb{K}\setminus\{\tau-1,\tau\}}|E_{x_{\eta_j}x_{\eta_j+1}x_jx_{j+1}}|^{q_j}}.
		\end{split}
	\end{equation}
	Now we relabel
	\begin{equation*}\label{relabel1}
		y_j=\begin{cases}
			x_j,\quad&0\leq j\leq\tau-1,\\
			x_{j+1}&\tau\leq j\leq k,
		\end{cases}
		\quad\tilde{\beta}_j=\begin{cases}
			\beta_j,\quad&1\leq j\leq\tau-1,\\
			\beta_{j+1}&\tau\leq j\leq k-1,
		\end{cases}
	\end{equation*}
	and
	\begin{equation*}\label{relabel2}
		\tilde{a}_j=\begin{cases}
			a_j,\quad&0\leq j\leq\tau-2,\\
			\min(a_{\tau-1},a_\tau),&j=\tau-1,\\
			a_{j+1}&\tau\leq j\leq k-1,
		\end{cases}
		\quad\tilde{q}_j=\begin{cases}
			q_j,\quad&\text{if}~\tau\geq3,~1\leq j\leq\tau-2,\\
			n+1-4m,&j=\tau-1,\\
			q_{j+1},&\tau\leq j\leq k-1.
		\end{cases}
	\end{equation*}
	Note that $\eta_{\tau-1}<\tau-1$ implies $E_{x_{\eta_{\tau-1}}x_{\eta_{\tau-1}+1}x_{\tau-1}x_{\tau+1}}=E_{y_{\eta_{\tau-1}}y_{\eta_{\tau-1}+1}y_{\tau-1}y_\tau}$. If $1\leq j\leq\tau-2$, it is also obvious from the admissibility that $E_{x_{\eta_j}x_{\eta_j+1}x_jx_{j+1}}=E_{y_{\eta_j}y_{\eta_j+1}y_jy_{j+1}}$. If $\tau+1\leq j\leq k$, the argument from \eqref{tau<s-1} to \eqref{2leqtauleqk-1} also shows that
	\begin{equation*}
		E_{x_{\eta_j}x_{\eta_j+1}x_jx_{j+1}}=\begin{cases}E_{y_{\eta_j}y_{\eta_j+1}y_{j-1}y_j},\quad&\text{if}~\eta_j\leq\tau-2,\\
			E_{y_{\eta_j-1}y_{\eta_j}y_{j-1}y_j},&\text{if}~\eta_j\geq\tau+1.
		\end{cases}
	\end{equation*}
	In other words, if we denote
	\begin{equation*}\label{deftildeeta}
		\tilde{\eta}_j=\begin{cases}
			\eta_j,\quad&1\leq j\leq\tau-1,\\
			\eta_{j+1},&\text{if}~\tau\leq j\leq k-1~\text{and}~\eta_{j+1}\leq\tau-2,\\
			\eta_{j+1}-1,&\text{if}~\tau\leq j\leq k-1~\text{and}~\eta_{j+1}\geq\tau+1,
		\end{cases}
	\end{equation*}
	then \eqref{Ik1st} says
	\begin{equation}\label{I_ktoI_k-1}
		I_k\lesssim\int_{\mathbb{R}^{(k-1)n}}\frac{\langle y_1\rangle^{-\tilde{\beta}_1-}\cdots\langle y_{k-1}\rangle^{-\tilde{\beta}_{k-1}-}}{|y_0-y_1|^{\tilde{a}_0}\cdots|y_{k-1}-y_k|^{\tilde{a}_{k-1}}}|E_{\tilde{\eta}_1,1}|^{-\tilde{q}_1}\cdots|E_{\tilde{\eta}_{k-1},k-1}|^{-\tilde{q}_{k-1}}\d y_1\cdots \d y_{k-1}.
	\end{equation}
	
	We claim that $\{(\tilde{\eta}_j,j);~j\in\mathbb{K}\setminus\{k\}\}$ is admissible. In fact, that $\tilde{\eta}_j<j$ is obvious by definition, and to check condition 2) in Definition \ref{defadmissible}, it is only necessary to discuss when $\tau\leq j\leq k-1$ and $\tilde{\eta}_j<i<j$, where there are four possibilities:
	\begin{itemize}
		\item[1)] If $i\leq\tau-1$, we must know $\tilde{\eta}_j=\eta_{j+1}$, so $\tilde{\eta}_j\leq\eta_i=\tilde{\eta}_i$ is obvious.
		
		\item[2)] If $i\geq\tau$ and $\tilde{\eta}_i=\eta_{i+1}$, it is easy to see from $\tilde{\eta}_j<i<j$ that $\eta_{j+1}<i+1<j+1$, and therefore $\tilde{\eta}_j\leq\eta_{j+1}\leq\eta_{i+1}=\tilde{\eta}_i$.
		
		\item[3)] If $i\geq\tau$, $\tilde{\eta}_i=\eta_{i+1}-1$ and $\tilde{\eta}_j=\eta_{j+1}-1$, then $\eta_{j+1}<i+1<j+1$ and thus $\tilde{\eta}_j=\eta_{j+1}-1\leq\eta_{i+1}-1=\tilde{\eta}_i$.
		
		\item[4)] If $i\geq\tau$, $\tilde{\eta}_i=\eta_{i+1}-1$ and $\tilde{\eta}_j=\eta_{j+1}$, then $\eta_{j+1}\leq\tau-2$, $\eta_{i+1}\geq\tau+1$, and thus $\tilde{\eta}_j\leq\tilde{\eta}_i-2$.
	\end{itemize}
	
	Now it is routine to check the relevant conditions in \eqref{indexcond1} with respect to $\tilde{a}_i$, $\tilde{\beta}_i$ and $\tilde{q}_i$ for the RHS of \eqref{I_ktoI_k-1}. Thus the estimate of $I_k$ is reduced to that of $I_{k-1}$.
	
	\noindent\emph{Case 2: $\eta_s\geq s-2$ for all $s\in\{3,\cdots,k\}$, and $\eta_3=2$.}
	
	We first know $\eta_j\geq2$ for $j\geq3$ so that $E_{\eta_j,j}$ is independent of $x_1$. Since $\beta_1\geq2m$ always holds, we can apply Proposition \ref{twoline} to the integral with respect to $x_1$ whenever $\eta_2=0$ or $\eta_2=1$ to have
	\begin{equation}\label{I_kx_1}
		\begin{split}
			&\int_{\mathbb{R}^{n}}\langle x_1\rangle^{-\beta_1-}|r_0|^{-a_0}|r_1|^{-a_1}|E_{0,1}|^{-q_1}|E_{\eta_2,2}|^{-q_2}\d x_1\\
			=&\int_{\mathbb{R}^{n}}\frac{\langle x_1\rangle^{-\beta_1-}\d x_1}{|x_0-x_1|^{a_0}|x_1-x_2|^{a_1}|E_{x_0x_1x_1x_2}|^{n+1-4m}|E_{x_{\eta_2}x_{\eta_2+1}x_2x_3}|^{(n+1-4m)}},\\
			\lesssim&|E_{x_0x_2x_2x_3}|^{-(n+1-4m)}|x_0-x_2|^{-\min\{a_0,a_1\}},
		\end{split}
	\end{equation}
	and consequently
	\begin{equation*}\label{I_kbound}
		\begin{split}
			I_k\lesssim\int_{\mathbb{R}^{(k-1)n}}&\frac{\langle x_2\rangle^{-\beta_2-}\cdots\langle x_k\rangle^{-\beta_k-}}{|x_0-x_2|^{\min\{a_0,a_1\}}|x_2-x_3|^{a_2}\cdots|x_k-x_{k+1}|^{a_k}}\\
			&\times\frac{\d x_2\cdots \d x_k}{|E_{x_0x_2x_2x_3}|^{n+1-4m}|E_{\eta_3,3}|^{q_3}\cdots|E_{\eta_k,k}|^{q_k}},
		\end{split}
	\end{equation*}
	so the estimate is reduced to that of $I_{k-1}$ with the obvious relabeling
	\begin{equation*}
		\begin{gathered}
			y_j=\begin{cases}
				x_0,\quad&j=0,\\
				x_{j+1},&1\leq j\leq k,
			\end{cases}\quad\tilde{a}_j=\begin{cases}
				\min\{a_0,a_1\},\quad&j=0,\\
				a_{j+1},&1\leq j\leq k-1,
			\end{cases}\\
			\tilde{\eta}_j=\begin{cases}
				0,\quad&j=1,\\
				\eta_{j+1},&2\leq j\leq k-1,
			\end{cases}	\quad(\tilde{\beta}_j,\tilde{q}_j)=(\beta_{j+1},q_{j+1}),\quad 1\leq j\leq k-1,
		\end{gathered}
	\end{equation*}
	and an easy check of relevant conditions in \eqref{indexcond1} where we remark that
	\begin{equation*}
		\tilde{a}_0+\tilde{a}_{k-1}=\min\{a_0+a_k,a_1+a_k\}\geq\min\{a_0+a_k,n-2m\}\geq\max\{\mbox{$\frac{n-1}{2}$},q_k\}=\max\{\mbox{$\frac{n-1}{2}$},\tilde{q}_{k-1}\}.
	\end{equation*}
	
	\noindent\emph{Case 3: $\eta_s\geq s-2$ for all $s\in\{3,\cdots,k\}$, and $\eta_3=1$.}
	
	We must have $\{E_{\eta_1,1},E_{\eta_2,2},E_{\eta_3,3}\}=\{E_{0,1},E_{1,2},E_{1,3}\}$. Since $E_{1,3}=E_{1,2}+E_{2,3}$, it follows from homogeneity argument that
	\begin{equation*}
		|E_{1,2}|^{-(n+1-4m)}|E_{1,3}|^{-q_3}\lesssim|E_{1,2}|^{-(n+1-4m)}|E_{2,3}|^{-q_3}+|E_{2,3}|^{-(n+1-4m)}|E_{1,3}|^{-q_3},
	\end{equation*}
	and then
	\begin{equation*}
		\begin{split}
			I_k\lesssim&\int_{\mathbb{R}^{kn}}\frac{\langle x_1\rangle^{-\beta_1-}\cdots\langle x_k\rangle^{-\beta_k-}\d x_1\cdots \d x_k}{|r_0|^{a_0}\cdots|r_k|^{a_k}|E_{0,1}|^{n+1-4m}|E_{1,2}|^{n+1-4m}|E_{2,3}|^{q_3}\cdots|E_{\eta_k,k}|^{q_k}}\\
			&+\int_{\mathbb{R}^{kn}}\frac{\langle x_1\rangle^{-\beta_1-}\cdots\langle x_k\rangle^{-\beta_k-}\d x_1\cdots \d x_k}{|r_0|^{a_0}\cdots|r_k|^{a_k}|E_{0,1}|^{n+1-4m}|E_{2,3}|^{n+1-4m}|E_{1,3}|^{q_3}\cdots|E_{\eta_k,k}|^{q_k}}\\
			:=&I_k^1+I_k^2.
		\end{split}
	\end{equation*}
	Bounding $I_k^1$ has essentially been discussed, because the integral with respect to $x_1$ is exactly \eqref{I_kx_1} with $\eta_2=1$, and all consequences follow with no change.
	
	Therefore, we are left to bound $I_k^2$. We first apply Proposition \ref{twoline} to the integral with respect to $x_1$ to get
	\begin{equation*}
		\int_{\mathbb{R}^n}\frac{\langle x_1\rangle^{-\beta_1-}\d x_1}{|r_0|^{a_0}|r_1|^{a_1}|E_{0,1}|^{n+1-4m}|E_{1,3}|^{q_3}}\lesssim|E_{x_0x_2x_3x_4}|^{-q_3}|x_0-x_2|^{-\min\{a_0,a_1\}},
	\end{equation*}
	where $\beta_1\geq2m$ is used. If $k=3$, then
	\begin{equation*}
		\begin{split}
			I_k^2=I_3^2\lesssim&\int_{\mathbb{R}^{2n}}\frac{\langle x_2\rangle^{-\beta_2-}\langle x_3\rangle^{-\beta_3-}\d x_2\d x_3}{|x_0-x_2|^{\min\{a_0,a_1\}}|x_2-x_3|^{a_2}|x_3-x_4|^{a_3}|E_{x_0x_2x_3x_4}|^{q_3}|E_{x_2x_3x_3x_4}|^{n+1-4m}}\\
			\lesssim&\int_{\mathbb{R}}\frac{\langle x_2\rangle^{-\beta_2-}\d x_2}{|x_0-x_2|^{\min\{a_0,a_1\}}|x_2-x_4|^{\min\{a_2,a_3\}}|E_{x_0x_2x_2x_4}|^{q_3}}\\
			\lesssim&1,\quad|x_0-x_4|=|x_0-x_{k+1}|\geq1,
		\end{split}
	\end{equation*}
	where in the last step we have used the fact that
	\begin{equation*}
		\min\{a_0,a_1\}+\min\{a_2,a_3\}\geq\min\{a_0+a_3,n-2m\}\geq\max\{\mbox{$\frac{n-1}{2}$},q_3\}.
	\end{equation*}
	If $k\geq4$, we have $q_3=n+1-4m$ and
	\begin{equation}\label{Ikfinal}
		\begin{split}
			I_k^2\lesssim\int_{\mathbb{R}^{(k-1)n}}&\frac{\langle x_2\rangle^{-\beta_2-}\cdots\langle x_k\rangle^{-\beta_k-}}{|x_0-x_2|^{\min\{a_0,a_1\}}|r_2|^{a_2}\cdots|r_k|^{a_k}}\\
			&\times\frac{\d x_2\cdots \d x_k}{|E_{x_0x_2x_3x_4}|^{n+1-4m}|E_{x_2x_3x_3x_4}|^{n+1-4m}|E_{\eta_4,4}|^{q_4}\cdots|E_{\eta_k,k}|^{q_k}}
		\end{split}
	\end{equation}
	The triangle inequality implies
	\begin{equation}\label{triangle}
		\begin{split}
			&\frac{1}{|E_{x_0x_2x_3x_4}|^{n+1-4m}|E_{x_2x_3x_3x_4}|^{n+1-4m}}\\
			\lesssim&\frac{1}{|E_{x_0x_2x_2x_3}|^{n+1-4m}|E_{x_2x_3x_3x_4}|^{n+1-4m}}+\frac{1}{|E_{x_0x_2x_2x_3}|^{n+1-4m}|E_{x_0x_2x_3x_4}|^{n+1-4m}},
		\end{split}
	\end{equation}
	and we also note that $\eta_4=3$ must hold, because otherwise $2=\eta_4\leq\eta_3<3<4$ yields the contradiction $\eta_3=2$. So we may apply Proposition \ref{twoline} and \eqref{triangle} to the integral with respect to $x_2$ on the RHS of \eqref{Ikfinal} where $E_{\eta_4,4},\cdots,E_{\eta_k,k}$ are irrelevant, to get
	\begin{equation*}
		I_k^2\lesssim\int_{\mathbb{R}^{(k-2)n}}\frac{\langle x_3\rangle^{-\beta_3-}\cdots\langle x_k\rangle^{-\beta_k-}\d x_3\cdots \d x_k}{|x_0-x_3|^{\min\{a_0,a_1,a_2\}}|r_3|^{a_3}\cdots|r_k|^{a_k}|E_{x_0x_3x_3x_4}|^{n+1-4m}|E_{\eta_4,4}|^{q_4}\cdots|E_{\eta_k,k}|^{q_k}},
	\end{equation*}
	then the estimate is reduced to that of $I_{k-2}$ with the relabeling
	\begin{equation*}
		\begin{gathered}
			y_j=\begin{cases}
				x_0,\quad&j=0,\\
				x_{j+2},&1\leq j\leq k-1,
			\end{cases}\quad\tilde{a}_j=\begin{cases}
				\min\{a_0,a_1,a_2\},\quad&j=0,\\
				a_{j+2},&1\leq j\leq k-2,
			\end{cases}\\
			\tilde{\eta}_j=\begin{cases}
				0,\quad&j=1,\\
				\eta_{j+2}&2\leq j\leq k-2,
			\end{cases}\quad(\tilde{\beta}_j,\tilde{q}_j)=(\beta_{j+2},q_{j+2}),\quad 1\leq j\leq k-2,
		\end{gathered}
	\end{equation*}
	and the proof is complete with an easy check of relevant conditions in \eqref{indexcond1}.
\end{proof}

\begin{lemma}\label{propIkest2}
	Suppose $n\geq4m+1$, $k\geq2$, $\{(\eta_j,j);\,j\in\mathbb{K}\}$ is admissible in \eqref{Ix_0x_k}, $j_0\in\mathbb{K}\setminus\{k\}$ be fixed with $j_0\leq\eta_j\,(j_0+1\leq j\leq k)$, $q_i=n+1-4m$ for $i\in\mathbb{K}\setminus\{j_0\}$, $q_{j_0}=\frac{n+1}{2}-2m$, $n-2m\leq a_i\leq n-2\,(i\notin\{0,k\})$, $0\leq a_0,a_k\leq n-2$ and either
	\begin{equation}\label{indexcond2}
		a_0,a_k\geq\mbox{$\frac{n-1}{2}$},\,\,\beta_i\geq2m\,(i\in\mathbb{K}),
	\end{equation}
	or
	\begin{equation}\label{indexcond2'}
		a_0+a_k\geq\mbox{$\frac{n-1}{2}$},\,\,\beta_i\geq\mbox{$\frac{n+1}{2}$}\,(i\in\mathbb{K}).
	\end{equation}
	Then
	\begin{equation*}
		I_k(x_0,x_{k+1};\beta_1,\cdots,\beta_k,a_0,\cdots,a_k,q_1,\cdots,q_k)\lesssim1,\quad|x_0-x_{k+1}|\gtrsim1.
	\end{equation*}
\end{lemma}

\begin{proof}
	We only show the proof when \eqref{indexcond2} holds, for the other case when \eqref{indexcond2'} holds can be shown in a parallel way.
	
	If $k=2$, then $j_0=\eta_2=1$ and $\{E_{\eta_1,1},E_{\eta_2,2}\}=\{E_{0,1},E_{1,2}\}$. One checks with Proposition \ref{twoline} to deduce
	\begin{equation*}
		\begin{split}
			&\int_{\mathbb{R}^{n}}\langle x_2\rangle^{-\beta_2-}|r_1|^{-a_1}|r_2|^{-a_2}|E_{0,1}|^{-q_1}|E_{1,2}|^{-q_2}\d x_2\\
			=&\int_{\mathbb{R}^{n}}\frac{\langle x_2\rangle^{-\beta_2-}\d x_2}{|x_1-x_2|^{a_1}|x_2-x_3|^{a_2}|E_{x_0x_1x_1x_2}|^{\frac{n+1}{2}-2m}|E_{x_1x_2x_2x_3}|^{n+1-4m}},\\
			\lesssim&|E_{x_0x_1x_1x_3}|^{-(\frac{n+1}{2}-2m)}|x_1-x_3|^{-\min\{a_1,a_2\}},\quad|x_1-x_3|>0,
		\end{split}
	\end{equation*}
	and then
	\begin{equation*}
		I_2\lesssim\int_{\mathbb{R}^n}\frac{\langle x_1\rangle^{-\beta_1-}\d x_1}{|x_0-x_1|^{a_0}|x_1-x_3|^{\min\{a_1,a_2\}}|E_{x_0x_1x_1x_3}|^{\frac{n+1}{2}-2m}}\lesssim1,\quad|x_0-x_3|\gtrsim1,
	\end{equation*}
	where we have used $a_1\geq n-2m$ to make sure $\min\{a_1,a_2\}+a_0\geq\frac{n-1}{2}$.
	
	If $k\geq3$, by the second assertion of Proposition \ref{proad}, we can find $\tau\in\mathbb{K}$ such that $x_{\tau}$ is independent of $E_{\eta_j,j}$ unless $j\in\{\tau-1,\tau\}$, satisfying either
	\begin{equation*}\label{j_01}
		j_0\leq\eta_k\leq\eta_{\tau+1}<\eta_\tau<\tau<\tau+1\leq k,
	\end{equation*}
	or
	\begin{equation*}\label{j_02}
		j_0\leq\eta_k<\tau=k.
	\end{equation*}
	We split the argument into two cases.

	\noindent\emph{Case 1: $j_0=\tau-1$.}
	
	We must have $j_0=\eta_k=k-1$ in this case. We apply Proposition \ref{twoline} to the integral with respect to $x_k$ and get
	\begin{equation*}\label{intxk}
		\begin{split}
			&\int_{\mathbb{R}^n}\frac{\langle x_k\rangle^{-\beta_k-}\d x_k}{|x_{k-1}-x_k|^{a_{k-1}}|x_k-x_{k+1}|^{a_k}|E_{x_{\eta_{k-1}}x_{\eta_{k-1}+1}x_{k-1}x_k}|^{\frac{n+1}{2}-2m}|E_{x_{k-1}x_kx_kx_{k+1}}|^{n+1-4m}}\\
			&\lesssim|E_{x_{\eta_{k-1}}x_{\eta_{k-1}+1}x_{k-1}x_{k+1}}|^{-(\frac{n+1}{2}-2m)}|x_{k-1}-x_{k+1}|^{-\min\{a_{k-1},a_k\}},\quad|x_{k-1}-x_{k+1}|>0,
		\end{split}
	\end{equation*}
	where we have used
	\begin{equation*}
		a_{k-1}+a_k+\beta_k-n\geq a_{k},\quad a_{k-1}+a_k-(n+1-4m)\geq a_k.
	\end{equation*}
	Consequently
	\begin{equation*}
		\begin{split}
			I_k\lesssim\int_{\mathbb{R}^{(k-1)n}}&\frac{\langle x_1\rangle^{-\beta_1-}\cdots\langle x_{k-1}\rangle^{-\beta_{k-1}-}}{|x_0-x_1|^{a_0}\cdots|x_{k-2}-x_{k-1}|^{a_{k-2}}|x_{k-1}-x_{k+1}|^{\min\{a_{k-1},a_k\}}}\\
			&\times\frac{\d x_1\cdots \d x_{k-1}}{|E_{\eta_1,1}|^{n+1-4m}\cdots|E_{\eta_{k-2},k-2}|^{n+1-4m}|E_{x_{\eta_{k-1}}x_{\eta_{k-1}+1}x_{k-1}x_{k+1}}|^{\frac{n+1}{2}-2m}}.
		\end{split}
	\end{equation*}
	The desired estimate is then implied by Proposition \ref{propIkest1}, where in the relevant conditions \eqref{indexcond1} one mainly checks by $a_{k-1}\geq n-2m$ that
	\begin{equation*}
		\min\{a_0+\min\{a_{k-1},a_k\},n+1-4m\}\geq\min\{\mbox{$\frac{n-1}{2}$},n+1-4m\}\geq\mbox{$\frac{n+1}{2}$}-2m.
	\end{equation*}
	
	\noindent\emph{Case 2: $j_0<\tau-1$.}
	
	Since $j_0<j_0+1<\tau$ in this case, we have $q_{j_0+1}=\cdots=q_\tau=\cdots=q_k=n+1-4m$. The same type of argument from \eqref{tau<s-1} to \eqref{2leqtauleqk-1} shows that $E_{\eta_j,j}$ is irrelevant in the integral with respect to $x_\tau$ if $j\notin\{\tau-1,\tau\}$, which is exactly \eqref{intxtau}, and therefore the application of Proposition \ref{twoline} gives
	\begin{equation*}\label{I_kcase2est}
		\begin{split}
			I_k\lesssim&\int_{\mathbb{R}^{(k-1)n}}\frac{\prod\limits_{j\in\mathbb{K}\setminus\{\tau\}}\langle x_j\rangle^{-\beta_\tau-}}{|x_{\tau-1}-x_{\tau+1}|^{\min\{a_{\tau-1},a_\tau\}}|E_{x_{\eta_{\tau-1}}x_{\eta_{\tau-1}+1}x_{\tau-1}x_{\tau+1}}|^{n+1-4m}}\\
			&\quad\quad\times\frac{\prod\limits_{j\in\mathbb{K}\setminus\{\tau\}}\d x_j}{\prod\limits_{j\in\mathbb{K}_0\setminus\{\tau-1,\tau\}}|r_j|^{a_j}\prod\limits_{j\in\mathbb{K}\setminus\{\tau-1,\tau\}}|E_{x_{\eta_j}x_{\eta_j+1}x_jx_{j+1}}|^{q_j}}.
		\end{split}
	\end{equation*}
	If $\tau=k$, then $\min\{a_{\tau-1},a_\tau\}\geq\frac{n-1}{2}$. If $\tau<k$, one then has $1<\tau<k$ and $\min\{a_{\tau-1},a_\tau\}\geq n+1-4m$ by the assumptions in this lemma. Therefore the estimate can be immediately reduced to that of $I_{k-2}$, and the proof is complete.
\end{proof}

\begin{lemma}\label{propIkest3}
	Suppose $n\geq4m+1$, $k\geq2$, $\{(\eta_j,j);\,j\in\mathbb{K}\}$ is admissible in \eqref{Ix_0x_k}, and
	\begin{equation*}\label{indexcond3}
		\begin{cases}
			n-2m\leq a_i\leq n-2\,(a\leq i\leq k-1),\,\,0\leq a_0,a_k\leq n-2,\\
			\beta_i\geq2m\,(i\in\mathbb{K}),\\
			q_i=n+1-4m\,\,(i\in\mathbb{K}).
		\end{cases}
	\end{equation*}
	\begin{itemize}
		\item [1)] If $n-2\geq a_0+a_k\geq\frac{n-1}{2}$, then
		\begin{equation*}
			I_k(x_0,x_{k+1};\beta_1,\cdots,\beta_k,a_0,\cdots,a_k,q_1,\cdots,q_k)\lesssim1,\quad0<|x_0-x_{k+1}|\lesssim1.
		\end{equation*}
		
		\item [2)] If $i_0\in\mathbb{K}\setminus\{k\}$ and $a_0,a_k\geq\frac{n-1}{2}$, then
		\begin{equation*}
			\int_{\mathbb{R}^{kn}}\frac{\langle x_1\rangle^{-\beta_1-}\cdots\langle x_k\rangle^{-\beta_k-}\d x_1\cdots \d x_k}{\langle x_{i_0}-x_{i_0+1}\rangle^\frac{n-1}{2}\prod_{i\in\mathbb{K}_0\setminus\{i_0\}}|x_i-x_{i+1}|^{a_i}\prod_{i\in\mathbb{K}}|E_{\eta_i,i}|^{n+1-4m}}\lesssim1,\quad0<|x_0-x_{k+1}|\lesssim1.
		\end{equation*}
	\end{itemize}
\end{lemma}

\begin{proof}
	The same type of argument from \eqref{maxtau} to \eqref{I_ktoI_k-1} with the help of the second assertion of Proposition \ref{proad} shall finally reduce the estimate to the case of $k=2$. Without repeating the discussion, we only prove when $k=2$ in the following.
	
	To show 1), note that $\eta_2=0$ or $\eta_2=1$ holds, and
	\begin{equation*}
		\begin{split}
			&\int_{\mathbb{R}^{n}}\frac{\langle x_2\rangle^{-\beta_2-}\d x_2}{|x_1-x_2|^{a_1}|x_2-x_3|^{a_2}|E_{x_0x_1x_1x_2}|^{n+1-4m}|E_{x_{\eta_2}x_{\eta_2+1}x_2x_3}|^{n+1-4m}}\\
			\lesssim&|E_{x_0x_1x_1x_3}|^{-(n+1-4m)}|x_1-x_3|^{-\min\{a_1,a_2\}},\quad|x_1-x_3|>0,
		\end{split}
	\end{equation*}
	where in the case of $\eta_2=0$ we have used
	\begin{equation*}
		\begin{split}
			&\frac{1}{|E_{x_0x_1x_1x_2}|^{n+1-4m}|E_{x_0x_1x_2x_3}|^{n+1-4m}}\\
			\lesssim&\frac{1}{|E_{x_0x_1x_1x_2}|^{n+1-4m}|E_{x_1x_2x_2x_3}|^{n+1-4m}}+\frac{1}{|E_{x_0x_1x_2x_3}|^{n+1-4m}|E_{x_1x_2x_2x_3}|^{n+1-4m}}.
		\end{split}
	\end{equation*}
	Therefore Proposition \ref{twoline} implies
	\begin{equation*}
		I_2\lesssim\int_{\mathbb{R}^n}\frac{\langle x_1\rangle^{-\beta_1-}\d x_1}{|x_0-x_1|^{a_0}|x_1-x_3|^{\min\{a_1,a_2\}}|E_{x_0x_1x_1x_3}|^{n+1-4m}}\lesssim1,\quad0<|x_0-x_3|\lesssim1.
	\end{equation*}
	
	We next show 2). If $k=2$, then $i_0=1$, and one checks with Proposition \ref{twoline} using $a_0\geq\frac{n-1}{2}$ that
	\begin{equation*}
		\begin{split}
			&\int_{\mathbb{R}^n}\frac{\langle x_1\rangle^{-\beta_1-} \mathrm{d}x_1}{|x_0-x_1|^{a_0}\langle x_1-x_2\rangle^\frac{n-1}{2}|E_{x_0x_1x_1x_2}|^{n+1-4m}|E_{x_{\eta_2}x_{\eta_2+1}x_2x_3}|^{n+1-4m}}\\
			\lesssim&|E_{x_0x_2x_2x_3}|^{-(n+1-4m)},\quad|x_0-x_2|>0,
		\end{split}
	\end{equation*}
	and therefore
	\begin{equation*}
		\begin{split}
			I_2\lesssim\int_{\mathbb{R}^n}\frac{\langle x_2\rangle^{-\beta_2-}\mathrm{d}x_2}{|x_2-x_3|^{a_3}|E_{x_0x_2x_2x_3}|^{n+1-4m}}\lesssim1,\quad0<|x_0-x_3|\lesssim1.
		\end{split}
	\end{equation*}

\end{proof}

The following is our second main technical result in Section \ref{section5}.

\begin{proposition}\label{intmainest}
	We have
	\begin{equation*}
		I^{(1)}:=\int_{\mathbb{R}^{kn}}\frac{X^\frac{n-1}{2}|V^{(\alpha_1)}(x_1)|\cdots|V^{(\alpha_k)}(x_k)|\d x_1\cdots \d x_k}{|r_0|^{n-2-l_0+d_0}\cdots|r_k|^{n-2-l_k+d_k}\|F_1\|^{p_1}\cdots\|F_s\|^{p_s}}\lesssim1,\quad|x_0-x_{k+1}|\gtrsim1,
	\end{equation*}
	and
	\begin{equation*}
		I^{(2)}:=\int_{\mathbb{R}^{kn}}\frac{X^{n-L_k-2m}|V^{(\alpha_1)}(x_1)|\cdots|V^{(\alpha_k)}(x_k)|\d x_1\cdots \d x_k}{|r_0|^{n-2-l_0+d_0}\cdots|r_k|^{n-2-l_k+d_k}\|F_1\|^{p_1}\cdots\|F_s\|^{p_s}}\lesssim1,\quad|x_0-x_{k+1}|\lesssim1,
	\end{equation*}
	where $V$, $F_i$ and all the indices are the same ones in \eqref{stepmu}.
\end{proposition}

\begin{proof}
	To estimate $I^{(1)}$, Proposition \ref{lmlinereduction} first implies either
	\begin{equation}\label{I1decom1}
		\begin{split}
			I^{(1)}\lesssim\int_{\mathbb{R}^{kn}}&X^\frac{n-1}{2}\prod_{i\in\mathbb{K}}|V^{(\alpha_i)}(x_i)|\prod_{i\in\mathbb{K}_0}|r_i|^{-(n-2-l_i+d_i)}\\
			&\quad\quad\times|E_{\iota_{j^{(i_0)}},j^{(i_0)}}|^{-(\frac{n+1}{2}-2m)}\prod_{j\in\mathbb{K}\setminus\{j^{(i_0)}\}}|E_{\iota_j,j}|^{-(n+1-4m)}\d x_1\cdots \d x_k,
		\end{split}
	\end{equation}
	or
	\begin{equation}\label{I1decom2}
		\begin{split}
			I^{(1)}\lesssim\int_{\mathbb{R}^{kn}}&X^\frac{n-1}{2}\prod_{i\in\mathbb{K}}|V^{(\alpha_i)}(x_i)|\prod_{i\in\mathbb{K}_0}|r_i|^{-(n-2-l_i+d_i)}\\
			&\times|E_{\iota_k,k}|^{-\min\{n+1-4m,\frac{n+1}{2}-2m+d_0+d_k\}}\prod_{j\in\mathbb{K}\setminus\{k\}}|E_{\iota_j,j}|^{-(n+1-4m)}\d x_1\cdots \d x_k,
		\end{split}
	\end{equation}
	where in the first estimate, $j^{(i_0)}\leq\iota_j$ holds for all $j\in\{j^{(i_0)}+1,\cdots,k\}$ if $j^{(i_0)}<k$. Recall that $\{(\iota_j,j;);j\in\mathbb{K}\}$ is admissible by Proposition \ref{lmadmissible}.
	
	We now decompose the RHS of either of these estimates into $I^{(1)}_j$ ($j\in\mathbb{K}_0$) according to the regions $D_j=\{X\sim|r_j|\}$.
	
	If $j\notin\{0,k\}$, then in $D_j$
	\begin{equation*}
		\begin{split}
			&X^\frac{n-1}{2}\prod_{i\in\mathbb{K}}|V^{(\alpha_i)}(x_i)|\prod_{i\in\mathbb{K}_0}|r_i|^{-(n-2-l_i+d_i)}\\
			\lesssim&\langle x_1\rangle^{-\frac{n+1}{2}-}\langle x_k\rangle^{-\frac{n+1}{2}-}\prod_{i\in\mathbb{K}\setminus\{1,k\}}\langle x_i\rangle^{-2m-}|r_0|^{-(n-2-l_0+d_0)}|r_k|^{-(n-2-l_k+d_k)}\prod_{i\in\mathbb{K}_0}|r_i|^{-\max\{n-2-l_i+d_i,n-2m\}},
		\end{split}
	\end{equation*}
	where we have used the decay $|V^{(\alpha_i)}(x_i)|\lesssim\langle x_i\rangle^{-\frac{3n+1}{2}-2m-}$,
	\begin{equation}\label{fangsuo}
		\langle x_i\rangle^{-1}\langle x_{i+1}\rangle^{-1}\lesssim\langle x_1-x_{i+1}\rangle^{-1}\leq|x_1-x_{i+1}|^{-1},	
	\end{equation}
	$n-2-l_i+d_i\geq\frac{n-1}{2}$ and $n-2m-\min\{n-2-l_i+d_i,n-2m\}\leq\frac{n+1}{2}-2m$. Now the RHS of both \eqref{I1decom1} and \eqref{I1decom2} when restricted in $D_j$ can be estimated as
	\begin{equation*}
		\begin{split}
			I^{(1)}_j\lesssim\int_{\mathbb{R}^{kn}}&\langle x_1\rangle^{-\frac{n+1}{2}-}\langle x_k\rangle^{-\frac{n+1}{2}-}\prod_{i\in\mathbb{K}\setminus\{1,k\}}\langle x_i\rangle^{-2m-}|r_0|^{-(n-2-l_0+d_0)}\\
			&\times|r_k|^{-(n-2-l_k+d_k)}\prod_{i\in\mathbb{K}_0\setminus\{0,k\}}|r_i|^{-\max\{n-2-l_i+d_i,n-2m\}}\prod_{j\in\mathbb{K}}|E_{\iota_j,j}|^{-(n+1-4m)}\d x_1\cdots\d x_k,
		\end{split}
	\end{equation*}
	and the desired bound is checked by Lemma \ref{propIkest1}.
	
	Similarly in $D_0$ or $D_k$, it follows that
	\begin{equation*}
		\begin{split}
			X^\frac{n-1}{2}\prod_{i\in\mathbb{K}}|V^{(\alpha_i)}(x_i)|\prod_{i\in\mathbb{K}_0}|r_i|^{-(n-2-l_i+d_i)}\lesssim\prod_{i\in\mathbb{K}}\langle x_i\rangle^{-\frac{n+1}{2}-}|r_0|^{-a_0}|r_k|^{-a_k}\prod_{i\in\mathbb{K}_0}|r_i|^{-\max\{n-2-l_i+d_i,n-2m\}},
		\end{split}
	\end{equation*}
	where either
	\begin{equation*}
		a_0+a_k\geq n-2-l_0-d_0+d_k\geq\mbox{$\frac{n-1}{2}$}+d_0+d_k,
	\end{equation*}
	or
	\begin{equation*}
		a_0+a_k\geq n-2-l_k+d_0+d_k\geq\mbox{$\frac{n-1}{2}$}+d_0+d_k,
	\end{equation*}
	and we have used $n-2-l_i-\frac{n-1}{2}\geq0$. So for $I^{(1)}_k$, we have
	\begin{equation*}
		\begin{split}
			I^{(1)}_k\lesssim\int_{\mathbb{R}^{kn}}&\prod_{i\in\mathbb{K}}\langle x_i\rangle^{-\frac{n+1}{2}-}|r_0|^{-a_0}|r_k|^{-a_k}\prod_{i\in\mathbb{K}_0\setminus\{0,k\}}|r_i|^{-\max\{n-2-l_i+d_i,n-2m\}}\\
			&\times|E_{\iota_{j^{(i_0)}},j^{(i_0)}}|^{-(\frac{n+1}{2}-2m)}\prod_{j\in\mathbb{K}\setminus\{j^{(i_0)}\}}|E_{\iota_j,j}|^{-(n+1-4m)}\d x_1\cdots\d x_k,
		\end{split}
	\end{equation*}
	or
	\begin{equation*}
		\begin{split}
			I^{(1)}_k\lesssim\int_{\mathbb{R}^{kn}}&\prod_{i\in\mathbb{K}}\langle x_i\rangle^{-\frac{n+1}{2}-}|r_0|^{-a_0}|r_k|^{-a_k}\prod_{i\in\mathbb{K}_0\setminus\{0,k\}}|r_i|^{-\max\{n-2-l_i+d_i,n-2m\}}\\
			&\times|E_{\iota_k,k}|^{-\min\{n+1-4m,\frac{n+1}{2}-2m+d_0+d_k\}}\prod_{j\in\mathbb{K}\setminus\{k\}}|E_{\iota_j,j}|^{-(n+1-4m)}\d x_1\cdots\d x_k,
		\end{split}
	\end{equation*}
	and the desired bound for $I^{(1)}_k$ can be checked by Lemma \ref{propIkest1} and Lemma \ref{propIkest2}, while the bound for $I^{(1)}_0$ follows in the same way.

	To estimate $I^{(2)}$, let $L_k=l_{i_1}+2-2m$ for some $i_1\in\mathbb{K}_0$. First note that in $D_j$ we have
	\begin{equation*}
		X^{n-L_k-2m}|r_j|^{-(n-2-l_j+d_j)}\lesssim|r_j|^{-(l_{i_1}-l_j+d_j)}.
	\end{equation*}
	
	If $l_j+2-2m\leq0$, then $j\notin I_{0,2}^*$, $d_j=0$ and
	\begin{equation*}
		\begin{split}
			\sum_{i\in I_{0,2}^*}(n-2-l_i+d_i)+l_{i_1}-l_j\leq&\sum_{i\in I_{0,2}^*\setminus\{i_1\}}(n-2-l_i)+\sum_{i=k_0}^{k-1}L_i+(n-2-l_j)\\
			\leq&(n-2m)(k-k_0)+n-2-l_j\leq(n-2)(k-k_0+1).
		\end{split}
	\end{equation*}
	Since $\# I_{0,2}^*=k-k_0+1$ and $\frac{n-1}{2}\leq n-2-l_i+d_i\leq n-2m$ holds for $i\in I_{0,2}^*$, we have the existence of $\delta_i\geq0$ for $i\in I_{0,2}^*$ such that $\frac{n-1}{2}\leq n-2-l_i+d_i+\delta_i\leq n-2$ holds, and
	\begin{equation*}
		\sum_{i\in I_{0,2}^*}(n-2-l_i+d_i+\delta_i)=\sum_{i\in I_{0,2}^*}(n-2-l_i+d_i)+l_{i_1}-l_j.
	\end{equation*}
	This, together with \eqref{fangsuo} and the fact that $|r_i|\lesssim|r_j|$ for $i\in\mathbb{K}_0$, implies the following bounds. We decompose $I^{(2)}$ into $I^{(2)}_j$ ($j=0,\cdots,k$) by restricting the integral in $D_j$. If $j\notin\{0,k\}$, then $I^{(2)}_j$ is bounded by
	\begin{equation*}
		\begin{split}
			\int_{\mathbb{R}^{kn}}&\langle x_{j-1}-x_j\rangle^{-\frac{n-1}{2}}\prod_{i\in\mathbb{K}}\langle x_i\rangle^{-2m-}\prod_{i\in I_{0,2}^*}|r_i|^{-\max\{n-2-l_i+d_i+\delta_i,n-2m\}}\prod_{i\in\{0,k\}}|r_i|^{-(n-2-l_i+d_i)}\\
			&\times\prod_{i\in\mathbb{K}_0\setminus(I_{0,2}^*\cup\{j\}\cup\{0,k\})}|r_i|^{-\max\{n-2-l_i+d_i,n-2m\}}\prod_{i\in\mathbb{K}}|E_{\iota_i,i}|^{-(n+1-4m)}\d x_1\cdots\d x_k.
		\end{split}
	\end{equation*}
	If $j\in\{0,k\}$, then $I^{(2)}_j$ is bounded by
	\begin{equation*}
		\begin{split}
			\int_{\mathbb{R}^{kn}}&\prod_{i\in\mathbb{K}}\langle x_i\rangle^{-2m-}\prod_{i\in I_{0,2}^*}|r_i|^{-\max\{n-2-l_i+d_i+\delta_i,n-2m\}}\prod_{i\in\{0,k\}\setminus\{j\}}|r_i|^{-(n-2-l_i+d_i)}\\
			&\times\prod_{i\in\mathbb{K}_0\setminus(I_{0,2}^*\cup\{j\}\cup\{0,k\})}|r_i|^{-\max\{n-2-l_i+d_i,n-2m\}}\prod_{i\in\mathbb{K}}|E_{\iota_i,i}|^{-(n+1-4m)}\d x_1\cdots\d x_k.
		\end{split}
	\end{equation*}
	These bounds can be estimated by Lemma \ref{propIkest3}, and we remark that the conditions for $a_0$ and $a_k$ in the two cases of Lemma \ref{propIkest3} follow from the fact that $n-2-l_i+d_i\geq \frac{n-1}{2}$ when $i=0,k$, and the other conditions are easy to check.
	
	If $l_j+2-2m>0$, then $j\in I_{0,2}^*$ and
	\begin{equation*}
		\sum_{i\in I_{0,2}^*\setminus\{j\}}(n-2-l_i+d_i)+l_{i_1}-l_j+d_{j}\leq(n-2m)(k-k_0),
	\end{equation*}
	so the estimate can also be similarly reduced to the application of Lemma \ref{propIkest3}, and we save the parallel details here.
\end{proof}

\subsection{Completing the estimate for $\Omega_k^{high,2}(t,x,y)$}\label{section5.5}\

Recall that we are now left to obtain estimate \eqref{lgood} for $I^{\vec{l}}(t;x,y)$ defined by \eqref{Ivecl} under the assumption \eqref{lbad}, which completes the estimate \eqref{Ogema2bound} for $\Omega_k^{high,2}(t,x,y)$.

Apply Proposition \ref{stepmuprop} with $\mu=k-k_0$ (and $T$ replaced by $\delta T$ there) to the spatial integrals in \eqref{Ivecl}, and note that $J=L_{k_0}+\cdots+K_{k-1}$ must hold by \eqref{stepmuconstraint1}, we have
\begin{equation}\label{related}
	\begin{split}
		&I^{\vec{l}}(t;x,y)\\
		=&\int_0^{+\infty}e^{\mathrm{i}t\lambda^{2m}}\tilde{\chi}(\lambda^{2m})\lambda^{2m-1-\sum_{i\in\mathbb{K}_0}(2m-2-l_i)-L_{k_0}-\cdots-L_{k-1}}\\
		&\times\left(\int_{\mathbb{R}^{kn}}e^{\pm\mathrm{i}\lambda X}\frac{\left(\prod_{i\in\mathbb{K}}V^{(\alpha_i)}(x_i)\right)g(\lambda,r_0,\cdots,r_k,F_1,\cdots,F_s)\psi(\mbox{$\frac{X}{\delta T}$})}{\prod_{i\in\mathbb{K}_0}|r_i|^{n-2-l_i+d_i}\prod_{i=1}^{s}\|F_i\|^{p_i}}\d x_1\cdots \d x_k\right)\d\lambda\\
		=&\int_{\mathbb{R}^{kn}}\frac{\prod_{i\in\mathbb{K}}V^{(\alpha_i)}(x_i)}{\prod_{i\in\mathbb{K}_0}|r_i|^{n-2-l_i+d_i}\prod_{i=1}^{s}\|F_i\|^{p_i}}\psi(\mbox{$\frac{X}{\delta T}$})\d x_1\cdots \d x_k\\
		&\times\int_0^{+\infty}e^{\mathrm{i}t\lambda^{2m}\pm\mathrm{i}\lambda X}\tilde{\chi}(\lambda^{2m})\lambda^{2m-1-\sum_{i\in\mathbb{K}_0}(2m-2-l_i)-L_{k_0}-\cdots-L_{k-1}}g(\lambda,r_0,\cdots,r_k,F_1,\cdots,F_s)\d\lambda,
	\end{split}
\end{equation}
with all properties illustrated in Proposition \ref{stepmuprop}, where we note that $X\geq\frac12\delta T\sim T$ holds in $\mathrm{supp}\,\psi(\frac{X}{\delta T})$. By definition of $L_i$ in \eqref{defL_i}, we have
\begin{equation*}
	\sum_{i\in\mathbb{K}_0}(2m-2-l_i)+L_{k_0}+\cdots+L_{k-1}\geq-L_k=2m-2-l_{i_0},
\end{equation*}
where $i_0=\sigma(k)$ (see \eqref{defL_i}), and consequently
\begin{equation*}
		\tilde{\chi}(\lambda^{2m})\lambda^{2m-1-\sum_{i\in\mathbb{K}_0}(2m-2-l_i)-L_{k_0}-\cdots-L_{k-1}}g(\lambda,r_0,\cdots,r_k,F_1,\cdots,F_s)\in S^{1+l_{i_0}}\left(\mbox{$(\frac{\lambda_0}{2},+\infty)$}\right).
\end{equation*}
Thus Lemma \ref{lemmaA.222} gives
\begin{equation*}
	\begin{split}
		&\left|\int_0^{+\infty}e^{\mathrm{i}t\lambda^{2m}\pm\mathrm{i}\lambda X}\tilde{\chi}(\lambda^{2m})\lambda^{2m-1-\sum_{i\in\mathbb{K}_0}(2m-2-l_i)-L_{k_0}-\cdots-L_{k-1}}g(\lambda,r_0,\cdots,r_k,F_1,\cdots,F_s)\d\lambda\right|\\
		\lesssim&|t|^{-\frac12+\mu_{1+l_{i_0}}}X^{-\mu_{1+l_{i_0}}},\quad \mbox{$X\geq\frac{\delta T}{2}\gtrsim T$},
	\end{split}
\end{equation*}
and we recall that $\mu_{1+l_{i_0}}$ is defined by \eqref{mu}. Similar to the calculation in \eqref{cal1} and \eqref{cal2} (recall $l_{i_0}=\max_{j\in\mathbb{K}_0}l_j$ by definition of $L_k$ in \eqref{defL_i}), it follows when $X\gtrsim T$ that
\begin{equation*}
	|t|^{-\frac12+\mu_{1+l_{i_0}}}X^{-\mu_{1+l_{i_0}}-\frac{n-1}{2}}\text{~and~}|t|^{-\frac12+\mu_{1+l_{i_0}}}X^{-\mu_{1+l_{i_0}}-(n-L_k-2m)},
\end{equation*}
are both bounded by the function
\begin{equation}\label{est}
	\begin{cases}
		|t|^{-\frac n2}(1+|t|^{-1}|x-y|)^{-\frac{n(m-1)}{2m-1}},\quad&|t|\gtrsim1,\\
		|t|^{-\frac{n}{2m}}(1+|t|^{-\frac{1}{2m}}|x-y|)^{-\frac{n(m-1)}{2m-1}},&0<|t|\lesssim1.
	\end{cases}
\end{equation}
If $|x-y|\gtrsim1$, we conclude
\begin{equation*}
	|I^{\vec{l}}(t,x,y)|\lesssim\eqref{est}\times\int_{\mathbb{R}^{kn}}\frac{X^\frac{n-1}{2}|V^{(\alpha_1)}(x_1)|\cdots|V^{(\alpha_k)}(x_k)|\d x_1\cdots \d x_k}{|r_0|^{n-2-l_0+d_0}\cdots|r_k|^{n-2-l_k+d_k}\|F_1\|^{p_1}\cdots\|F_s\|^{p_s}},
\end{equation*}
and if $|x-y|\lesssim1$, we conclude
\begin{equation*}
	|I^{\vec{l}}(t,x,y)|\lesssim\eqref{est}\times\int_{\mathbb{R}^{kn}}\frac{X^{n-L_k-2m}|V^{(\alpha_1)}(x_1)|\cdots|V^{(\alpha_k)}(x_k)|\d x_1\cdots \d x_k}{|r_0|^{n-2-l_0+d_0}\cdots|r_k|^{n-2-l_k+d_k}\|F_1\|^{p_1}\cdots\|F_s\|^{p_s}},
\end{equation*}
so the proof is completed by Proposition \ref{intmainest}.

\begin{appendix}
	\renewcommand{\appendixname}{Appendix\,\,}
	\section{The proof of Proposition \ref{twoline}}\label{app-001}
	
	Recall the quantities $E_{xyyz}$ and $E_{ww'xy}$ defined in \eqref{linesing}, and we introduce two coordinates for $y\in\mathbb{R}^n$:
	\begin{equation}\label{l1}
		y=\begin{cases}
			\mbox{$\frac{x+z}{2}+s_y\frac{(x-z)}{|x-z|}+h_y$},\quad &s_y\in\mathbb{R},\,h_y\in(x-z)^\perp\cong\mathbb{R}^{n-1},\\
			\mbox{$x+\tilde{s}_y\frac{(w'-w)}{|w'-w|}+\tilde{h}_y$},&\tilde{s}_y\in\mathbb{R},\,\tilde{h}_y\in(w'-w)^\perp\cong\mathbb{R}^{n-1}.
		\end{cases}
	\end{equation}
	Set $\Gamma=\{y\in\mathbb{R}^n;\,0<|h_y|<\mbox{$\frac{|x-z|}{2}-|s_y|$}\}$, $\Gamma_+=\{y\in\Gamma;\,s_y\geq0\}$, $\Gamma_-=\{y\in\Gamma;\,s_y<0\}$ and $\tilde{\Gamma}=\{y\in\mathbb{R}^n;\,0<|\tilde{h}_y|<\tilde{s}_y\}$. It is elementary to show that
	\begin{equation}\label{l3}
		|E_{xyyz}|\sim
		\begin{cases}
			\frac{|h_y|}{\min\{|x-y|,|y-z|\}},\quad &y\in\Gamma,\\
			1,&y\notin\overline{\Gamma},
		\end{cases}
	\end{equation}
	and that
	\begin{equation}\label{l4}
		|E_{ww'xy}|\sim\mbox{$\frac{|\tilde{h}_y|}{|\tilde{s}_y|+|\tilde{h}_y|}$}\sim\mbox{$\frac{|\tilde{h}_y|}{|x-y|}$}\sim
		\begin{cases}
			\mbox{$\frac{|\tilde{h}_y|}{\tilde{s}_y}$},\quad&y\in\tilde{\Gamma},\\
			1,&y\notin\overline{\tilde{\Gamma}}.
		\end{cases}
	\end{equation}
	
	We first prove the special case $q=0$ of Proposition \ref{twoline}, that is the case where only one line singularity shows up.
	
	\begin{lemma}\label{oneline}
		Suppose $n\geq2$, $k_1,l_1\in[0,n)$, $k_2,l_2\in[0,+\infty)$, $\beta\in(0,+\infty)$, $k_2+l_2+\beta\geq n$, and $p\in[0,n-1)$. It follows uniformly in $y_0\in\mathbb{R}^n$ that
		\begin{equation}\label{l5}
			\begin{split}
				&\int_{\mathbb{R}^n}\frac{\langle|x-y|^{-k_1}\rangle\langle|y-z|^{-l_1}\rangle\langle y-y_0\rangle^{-\beta-}}{\langle x-y\rangle^{k_2}\langle y-z\rangle^{l_2}|E_{xyyz}|^p}\d y\\
				\lesssim&\begin{cases}
					\langle|x-z|^{-\max\{0,k_1+l_1-n\}}\rangle\langle x-z\rangle^{-\min\{k_2,l_2,k_2+l_2+\beta-n,k_2+l_2-p\}},\quad&k_1+l_1\neq n,\\
					\langle|x-z|^{0-}\rangle\langle x-z\rangle^{-\min\{k_2,l_2,k_2+l_2+\beta-n,k_2+l_2-p\}},&k_1+l_1=n.
				\end{cases}
			\end{split}
		\end{equation}
	\end{lemma}
	
	\begin{proof}
		By \eqref{l3} and Lemma \ref{lmEd}, the estimate for the integral over $\mathbb{R}^n\setminus\Gamma$ immediately follows. Now consider the integral over $\Gamma_+$. If $y\in\Gamma_+$, we have
		\begin{equation}\label{esim}
			\begin{cases}
				|x-y|\sim\frac{|x-z|}{2}-s_y\lesssim|x-z|,\\
				|y-z|\sim|x-z|,\\
				|E_{xyyz}|\sim\frac{|h_y|}{|x-y|},\\
				\langle y-y_0\rangle\sim\langle s_y-s_{y_0}\rangle+\langle h_y-h_{y_0}\rangle.
			\end{cases}
		\end{equation}
		
		When $|x-z|\leq2$, since $\langle x-y\rangle\sim\langle y-z\rangle\sim1$, $p<n-1$ and $k_1<n$, we have
		\begin{equation}\label{l7}
			\begin{split}
				\int_{\Gamma_+}\frac{\langle y-y_0\rangle^{-\beta-}\d y}{|x-y|^{k_1}|y-z|^{l_1}|E_{xyyz}|^p}\lesssim&|x-z|^{-l_1}\int_0^\frac{|x-z|}{2}\left(\mbox{$\frac{|x-z|}{2}-s_y$}\right)^{-(k_1-p)}\left(\int_{|h_y|<\frac{|x-z|}{2}-s_y}|h_y|^{-p}\d h_y\right)\d s_y\\
				\lesssim&|x-z|^{-(k_1+l_1-n)}.
			\end{split}
		\end{equation}
		
		When $|x-z|>2$ and $y\in\Gamma_+$, $\langle |y-z|^{-1}\rangle\sim1$ and $\langle y-z\rangle\sim|y-z|$ follow. Set $\theta=\min\{k_2,l_2,k_2+l_2+\beta-n\}\geq0$, we have
		\begin{equation*}
			\frac{1}{|x-y|^{k_2-p}|y-z|^{l_2}}\lesssim\frac{1}{|x-z|^\theta|x-y|^{k_2+l_2-p-\theta}},
		\end{equation*}
		and consequently we derive from \eqref{esim} that
		\begin{equation}\label{l8}
			\begin{split}
				&\int_{\Gamma_+}\frac{\langle|x-y|^{-k_1}\rangle\langle|y-z|^{-l_1}\rangle\langle y-y_0\rangle^{-\beta-}}{\langle x-y\rangle^{k_2}\langle y-z\rangle^{l_2}|E_{xyyz}|^p}\d y\\
				\lesssim&|x-z|^{-\theta}\int_0^{\frac{|x-z|}{2}-1}\left(\mbox{$\frac{|x-z|}{2}-s_y$}\right)^{-(k_2+l_2-p-\theta)}\left(\int_{|h_y|<\frac{|x-z|}{2}-s_y}\frac{\langle y-y_0\rangle^{-\beta-}\d h_y}{|h_y|^p}\right)\d s_y\\
				&+|x-z|^{-l_2}\int_{\frac{|x-z|}{2}-1}^{\frac{|x-z|}{2}}\left(\mbox{$\frac{|x-z|}{2}-s_y$}\right)^{-(k_1-p)}\left(\int_{|h_y|<\frac{|x-z|}{2}-s_y}\frac{\langle y-y_0\rangle^{-\beta-}\d h_y}{|h_y|^p}\right)\d s_y,
			\end{split}
		\end{equation}
		where the second term on the RHS is bounded by $|x-z|^{-l_2}$ if we neglect $\langle y-y_0\rangle^{-\beta-}$, so we are left to show
		\begin{equation*}
			\int_0^{\frac{|x-z|}{2}-1}\left(\mbox{$\frac{|x-z|}{2}-s_y$}\right)^{-(k_2+l_2-p-\theta)}\left(\int_{|h_y|<\frac{|x-z|}{2}-s_y}\frac{\langle y-y_0\rangle^{-\beta-}\d h_y}{|h_y|^p}\right)\d s_y\lesssim|x-z|^{-\min\{\theta,k_2+l_2-p\}}.
		\end{equation*}
		If $k_2+l_2-p\leq\theta$ which implies $\beta\geq n-p$, since
		\begin{equation*}
			\langle y-y_0\rangle^{-\beta-}\lesssim\langle s_y-s_{y_0}\rangle^{-(\beta-(n-p)+1)-}\langle h_y-h_{y_0}\rangle^{-(n-1-p)-},
		\end{equation*}
		by \eqref{esim} we have
		\begin{equation*}\label{l11}
			\begin{split}
				&\int_0^{\frac{|x-z|}{2}-1}\left(\mbox{$\frac{|x-z|}{2}-s_y$}\right)^{-(k_2+l_2-p-\theta)}\left(\int_{|h_y|<\frac{|x-z|}{2}-s_y}\frac{\langle y-y_0\rangle^{-\beta-}\d h_y}{|h_y|^p}\right)\d s_y\\
				\lesssim&|x-z|^{\theta-(k_2+l_2-p)}\int_\mathbb{R}\langle s_y-s_{y_0}\rangle^{-(1+\beta-(n-p))-}\d s_y\int_{\mathbb{R}^{n-1}}\frac{\langle h_y-h_{y_0}\rangle^{-(n-1-p)-}}{|h_y|^p}\d h_y\\
				\lesssim&|x-z|^{\theta-(k_2+l_2-p)}.
			\end{split}
		\end{equation*}
		If $k_2+l_2-p>\theta$, we use
		\begin{equation*}
			\begin{split}
				&\left(\mbox{$\frac{|x-z|}{2}-s_y$}\right)^{-(k_2+l_2-p-\theta)}\langle y-y_0\rangle^{-\beta-}\\
				\lesssim&\left(\mbox{$\frac{|x-z|}{2}-s_y$}\right)^{-(k_2+l_2+\beta-\theta-p)-}+\langle s_y-s_{y_0}\rangle^{-(k_2+l_2+\beta-n-\theta+1)-}\langle h_y-h_{y_0}\rangle^{-(n-1-p)-},
			\end{split}
		\end{equation*}
		to get
		\begin{equation*}\label{l13}
			\int_0^{\frac{|x-z|}{2}-1}\left(\mbox{$\frac{|x-z|}{2}-s_y$}\right)^{-(k_2+l_2-p-\theta)}\left(\int_{|h_y|<\frac{|x-z|}{2}-s_y}\frac{\langle y-y_0\rangle^{-\beta-}\d h_y}{|h_y|^p}\right)\d s_y\lesssim1.
		\end{equation*}
		Now the estimate for the integral over $\Gamma_+$ is shown, and the part over $\Gamma_-$ follows in a parallel way.
	\end{proof}

	Now we turn to prove Proposition \ref{twoline}.
	
	\begin{proof}[Proof of Proposition \ref{twoline}]
		We break the integral into three parts.
		
		\noindent\emph{Part 1: Integral over $\mathbb{R}^n\setminus\tilde{\Gamma}$}
		
		$|E_{ww'xy}|\sim1$ when $y\in\mathbb{R}^n\setminus\tilde{\Gamma}$, so \eqref{l5} implies the desired bound for the integral over $\mathbb{R}^n\setminus\tilde{\Gamma}$.

		\noindent\emph{Part 2: Integral over $\tilde{\Gamma}\setminus\Gamma$}
		
		If $y\in\tilde{\Gamma}\setminus\Gamma$, then $|E_{xyyz}|\sim1$, $|x-y|\sim\tilde{s}_y$, $|E_{ww'xy}|\sim\frac{|\tilde{h}_y|}{\tilde{s}_y}$, and
		\begin{equation}\label{l17}
			\begin{split}
				&\int_{\tilde{\Gamma}\setminus\Gamma}\frac{\langle|x-y|^{-k_1}\rangle\langle|y-z|^{-l_1}\rangle\langle y-y_0\rangle^{-\beta-}}{\langle x-y\rangle^{k_2}\langle y-z\rangle^{l_2}|E_{xyyz}|^p|E_{ww'xy}|^q}\d y\\
				\lesssim&\int_{\tilde{\Gamma}}\frac{\langle y-y_0\rangle^{-\beta-}\langle\tilde{s}_y^{-k_1}\rangle\langle\tilde{s}_y\rangle^{-k_2}\tilde{s}_y^q\langle|y-z|^{-l_1}\rangle\langle y-z\rangle^{-l_2}}{|\tilde{h}_y|^q}\d y.
			\end{split}
		\end{equation}
		We split the RHS of \eqref{l17} into 4 parts corresponded to the integration over
		\begin{equation*}
			\begin{split}
				&A=\left\{y\in\tilde{\Gamma};\,\tilde{s}_y\geq2|x-z|\right\},\\
				&B=\left\{y\in\tilde{\Gamma};\,0<\tilde{s}_y<2|x-z|,\,\mbox{$|y-z|\geq\frac12|x-z|$}\right\},\\
				&C=\left\{y\in\tilde{\Gamma};\,0<\tilde{s}_y<2|x-z|,\,\mbox{$|y-z|<\frac12|x-z|$},\,\mbox{$|\tilde{h}_y|\geq\frac12|\tilde{h}_z|$}\right\},\\
				&D=\left\{y\in\tilde{\Gamma};\,0<\tilde{s}_y<2|x-z|,\,\mbox{$|y-z|<\frac12|x-z|$},\,\mbox{$|\tilde{h}_y|<\frac12|\tilde{h}_z|$}\right\}.
			\end{split}
		\end{equation*}
		
		If $y\in A$, then $|y-z|\geq|\tilde{s}_y-\tilde{s}_z|\geq\tilde{s}_y-|x-z|\geq\frac{\tilde{s}_y}{2}$, and thus
		\begin{equation*}
			\begin{split}
				I_A&\triangleq\int_A\frac{\langle y-y_0\rangle^{-\beta-}\langle\tilde{s}_y^{-k_1}\rangle\langle\tilde{s}_y\rangle^{-k_2}\tilde{s}_y^q\langle|y-z|^{-l_1}\rangle\langle y-z\rangle^{-l_2}}{|\tilde{h}_y|^q}\d y\\
				&\lesssim\int_{2|x-z|}^{+\infty}\langle\tilde{s}_y-\tilde{s}_{y_0}\rangle^{-\beta-}\langle\tilde{s}_y^{-(k_1+l_1)}\rangle\langle\tilde{s}_y\rangle^{-(k_2+l_2)}\tilde{s}_y^q\left(\int_{|\tilde{h}_y|<\tilde{s}_y}\frac{\d\tilde{h}_y}{|\tilde{h}_y|^q}\right)\d\tilde{s}_y\\
				&\sim\int_{2|x-z|}^{+\infty}\langle\tilde{s}_y-\tilde{s}_{y_0}\rangle^{-\beta-}\langle\tilde{s}_y^{-(k_1+l_1)}\rangle\langle\tilde{s}_y\rangle^{-(k_2+l_2)}\tilde{s}_y^{n-1}\d\tilde{s}_y.
			\end{split}
		\end{equation*}
		It is quite elementary to get a sufficient bound
		\begin{equation*}
			I_A\lesssim\begin{cases}
				\langle|x-z|^{-\max\{0,k_1+l_1-n\}}\rangle\langle x-z\rangle^{-(k_2+l_2+\beta-n)},\quad&k_1+l_1\neq n,\\
				\langle|x-z|^{0-}\rangle\langle x-z\rangle^{-(k_2+l_2+\beta-n)},\quad&k_1+l_1=n.
			\end{cases}
		\end{equation*}
		
		Consider the integration over $B$. Set
		\begin{equation*}
			I_B\triangleq\int_B\frac{\langle y-y_0\rangle^{-\beta-}\langle\tilde{s}_y^{-k_1}\rangle\langle\tilde{s}_y\rangle^{-k_2}\tilde{s}_y^q\langle|y-z|^{-l_1}\rangle\langle y-z\rangle^{-l_2}}{|\tilde{h}_y|^q}\d y.
		\end{equation*}
		Note that when $y\in B$, we have $|y-z|\gtrsim|x-z|$, $\tilde{s}_y\sim|x-y|$ and $|\tilde{h}_y|<\tilde{s}_y<2|x-z|$, as how we treat \eqref{l7} and \eqref{l8}, it follows that
		\begin{equation*}\label{l22}
			I_B\lesssim\langle|x-z|^{-(k_1+l_1-n)}\rangle\langle x-z\rangle^{-\min\{k_2,l_2,k_2+l_2+\beta-n,k_2+l_2-q\}}.
		\end{equation*}
		
		Consider the integration over $C$. If $y\in C$, we have $\tilde{s}_y\sim|x-y|\sim|x-z|$, and it follows from \eqref{l4} that $|E_{ww'xy}|\sim\frac{|\tilde{h}_y|}{\tilde{s}_y}\gtrsim\frac{|\tilde{h}_z|}{|x-z|}\sim|E_{ww'xz}|$. Consequently,
		\begin{equation}\label{l23}
			\begin{split}
				I_C&\triangleq\int_C\frac{\langle y-y_0\rangle^{-\beta-}\langle\tilde{s}_y^{-k_1}\rangle\langle\tilde{s}_y\rangle^{-k_2}\tilde{s}_y^q\langle|y-z|^{-l_1}\rangle\langle y-z\rangle^{-l_2}}{|\tilde{h}_y|^q}\d y\\
				&\sim\int_C\frac{\langle|x-y|^{-k_1}\rangle\langle|y-z|^{-l_1}\rangle\langle y-y_0\rangle^{-\beta-}}{\langle x-y\rangle^{k_2}\langle y-z\rangle^{l_2}|E_{ww'xy}|^q}\d y\\
				&\lesssim|E_{ww'xz}|^{-q}\int_{\mathbb{R}^n}\frac{\langle|x-y|^{-k_1}\rangle\langle|y-z|^{-l_1}\rangle\langle y-y_0\rangle^{-\beta-}}{\langle x-y\rangle^{k_2}\langle y-z\rangle^{l_2}}\d y
			\end{split}
		\end{equation}
		and the estimate for $I_C$ immediately follows from Lemma \ref{lmEd}.
		
		Consider the integration over $D$. If $y\in D$, we have $\tilde{s}_y\sim|x-y|\sim|x-z|$ and $|y-z|\geq|\tilde{h}_y-\tilde{h}_z|\sim|\tilde{h}_z|$. Set
		\begin{equation}\label{ll24}
			I_D\triangleq\int_D\frac{\langle y-y_0\rangle^{-\beta-}\langle\tilde{s}_y^{-k_1}\rangle\langle\tilde{s}_y\rangle^{-k_2}\tilde{s}_y^q\langle|y-z|^{-l_1}\rangle\langle y-z\rangle^{-l_2}}{|\tilde{h}_y|^q}\d y.
		\end{equation}
		When $|x-z|\lesssim1$, first observe that $|\tilde{h}_z|\sim|E_{ww'xz}||x-z|\lesssim1$ and $|y-z|<\frac12|x-z|\lesssim1$ hold for $y\in D$. If $0\leq l_1\leq n-1$, we deduce
		\begin{equation*}\label{l24}
			\begin{split}
				I_D&\lesssim\int_{\tilde{s}_y\sim|x-z|}\frac{1}{\tilde{s}_y^{k_1-q}|\tilde{h}_z|^{l_1}}\left(\int_{|\tilde{h}_y|<\frac12|\tilde{h}_z|}\frac{\d\tilde{h}_y}{|\tilde{h}_y|^q}\right)\d\tilde{s}_y\\
				&\sim|x-z|^{-(k_1-q-1)}|\tilde{h}_z|^{-(l_1+q+1-n)}\\
				&\sim|x-z|^{-(k_1+l_1-n)}|E_{ww'xz}|^{-q}|E_{ww'xz}|^{n-1-l_1}\\
				&\leq|E_{ww'xz}|^{-q}|x-z|^{-(k_1+l_1-n)},
			\end{split}
		\end{equation*}
		and if $n-1<l_1<n$, we use $1\gtrsim|x-z|^{l_1}\gtrsim|y-z|^{l_1}\gtrsim|\tilde{s}_y-\tilde{s}_z|^{l_1-(n-1)}|\tilde{h}_z|^{n-1}$ when $y\in D$ to deduce
		\begin{equation*}\label{l25}
			\begin{split}
				I_D&\lesssim|x-z|^{-(k_1-q)}|\tilde{h}_z|^{-(n-1)}\int_{\tilde{s}_y\sim|x-z|}\frac{1}{|\tilde{s}_y-\tilde{s}_z|^{l_1-(n-1)}}\left(\int_{|\tilde{h}_y|<\frac12|\tilde{h}_z|}\frac{\d\tilde{h}_y}{|\tilde{h}_y|^q}\right)\d\tilde{s}_y\\
				&\lesssim|x-z|^{-(k_1+l_1-n-q)}|\tilde{h}_z|^{-q}\\
				&\sim|E_{ww'xz}|^{-q}|x-z|^{-(k_1+l_1-n)}.
			\end{split}
		\end{equation*}
		When $|x-z|\gtrsim1$, we decompose $I_D$ into $I_{D_1}$ and $I_{D_2}$ with respect to the region of integration, where $D_1=D\cap\{y\in\mathbb{R}^n;\,|y-z|<1\}$ and $D_2=D\cap\{y\in\mathbb{R}^n;\,|y-z|\geq1\}$. Recall $|E_{ww'xz}|\sim\frac{|\tilde{h}_z|}{|\tilde{s}_y|}$ holds when $y\in D$, we first see that $I_{D_1}$ has the bound
		\begin{equation*}
			\begin{split}
				I_{D_1}\lesssim&\int_{D\cap\{|y-z|<1\}}\frac{\d y}{\tilde{s}_y^{k_2-q}|\tilde{h}_y|^q|y-z|^{l_1}}\\
				\lesssim&|E_{ww'xz}|^{-q}|x-z|^{-k_2}\int_{D\cap\{|y-z|<1\}}\frac{|\tilde{h}_z|^q\d y}{|\tilde{h}_y|^q|y-z|^{l_1}},
			\end{split}
		\end{equation*}
		so if $0\leq l_1\leq n-1$, we use the facts that $|y-z|\gtrsim|\tilde{h}_z|$ and $|\tilde{h}_z|\sim|\tilde{h}_y-\tilde{h}_z|\leq|y-z|<1$ when $y\in D_1$ to get
		\begin{equation*}
			\begin{split}
				I_{D_1}\lesssim&|E_{ww'xz}|^{-q}|x-z|^{-k_2}\int_{D\cap\{|y-z|<1\}}\frac{|\tilde{h}_z|^{n-1-l_1}|\tilde{h}_z|^{-(n-1-q)}}{|\tilde{h}_y|^q}\d y\\
				\lesssim&|E_{ww'xz}|^{-q}|x-z|^{-k_2}\int_{|\tilde{s}_y-\tilde{s}_z|<1}|\tilde{h}_z|^{-(n-1-q)}\left(\int_{|\tilde{h}_y|<\frac12|\tilde{h}_z|}\frac{\d\tilde{h}_y}{|\tilde{h}_z|^q}\right)\d\tilde{s}_y\\
				\lesssim&|E_{ww'xz}|^{-q}|x-z|^{-k_2},
			\end{split}
		\end{equation*}
		while in the case of $n-1<l_1<n$, we use $|y-z|^l\gtrsim|\tilde{s}_y-\tilde{s}_z|^{l_1-(n-1)}|\tilde{h}_z|^{n-1}$ to get
		\begin{equation*}
			\begin{split}
				I_{D_1}\lesssim&|E_{ww'xz}|^{-q}|x-z|^{-k_2}\int_{|\tilde{s}_y-\tilde{s}_z|<1}|\tilde{s}_y-\tilde{s}_z|^{n-1-l_1}|\tilde{h}_z|^{-(n-1-q)}\left(\int_{|\tilde{h}_y|<\frac12|\tilde{h}_z|}\frac{\d\tilde{h}_y}{|\tilde{h}_z|^q}\right)\d\tilde{s}_y\\
				\lesssim&|E_{ww'xz}|^{-q}|x-z|^{-k_2}.
			\end{split}
		\end{equation*}
		For $I_{D_2}$, recall $|\tilde{h}_z|\leq|x-z|\sim\tilde{s}_y$ holds when $y\in D$, we have
		\begin{equation}\label{ID2}
			\begin{split}
				I_{D_2}\lesssim&|E_{ww'xz}|^{-q}|x-z|^{-(k_2+l_2+\beta-n)}\int_{D\cap\{|y-z|\geq1\}}\frac{\langle\tilde{s}_y-\tilde{s}_{y_0}\rangle^{-\beta-}\langle\tilde{s}_y-\tilde{s}_z\rangle^{-l_2}|\tilde{h}_z|^{n-1}}{\tilde{s}_y^{n-l_2-\beta}|\tilde{h}_z|^{n-1-q}|\tilde{h}_y|^q}\d y\\
				\lesssim&|E_{ww'xz}|^{-q}|x-z|^{-(k_2+l_2+\beta-n)}\int_{\tilde{s}_y\sim|x-z|}\frac{\langle\tilde{s}_y-\tilde{s}_{y_0}\rangle^{-\beta-}\langle\tilde{s}_y-\tilde{s}_z\rangle^{-l_2}}{\tilde{s}_y^{1-l_2-\beta}|\tilde{h}_z|^{n-1-q}}\left(\int_{|\tilde{h}_y|<\frac12|\tilde{h}_z|}\frac{\d\tilde{h}_y}{|\tilde{h}_y|^q}\right)\d\tilde{s}_y\\
				\lesssim&|E_{ww'xz}|^{-q}|x-z|^{-(k_2+l_2+\beta-n)}\int_{\mathbb{R}}\langle\tilde{s}_y-\tilde{s}_{y_0}\rangle^{-\beta-}\langle\tilde{s}_y-\tilde{s}_z\rangle^{-l_2}\langle\tilde{s}_y\rangle^{-(1-l_2-\beta)}\d\tilde{s}_y\\
				\lesssim&|E_{ww'xz}|^{-q}|x-z|^{-(k_2+l_2+\beta-n)}.
			\end{split}
		\end{equation}
		Now it is clear that
		\begin{equation*}\label{l29}
			I_D\lesssim|E_{ww'xz}|^{-q}\langle|x-z|^{-(k_1+l_1-n)}\rangle\langle x-z\rangle^{-\min\{k_2,k_2+l_2+\beta-n\}}.
		\end{equation*}
		and the bound for \eqref{l17} has been completely shown.
		
		\noindent\emph{Part 3. Integral over $\tilde{\Gamma}\cap\Gamma$}

		A routine calculation combining \eqref{l1} and $z=x+\tilde{s}_z\frac{(w'-w)}{|w'-w|}+\tilde{h}_z$ shows that
		\begin{equation*}
			\tilde{h}_y-h_y=\left(\left(\mbox{$\frac{|x-z|}{2}-s_y$}\right)\mbox{$\frac{\tilde{s}_z}{|x-z|}$}-\tilde{s}_y\right)\mbox{$\frac{(w'-w)}{|w'-w|}$}+\left(\mbox{$\frac{|x-z|}{2}-s_y$}\right)\mbox{$\frac{\tilde{h}_z}{|x-z|}$}.
		\end{equation*}
		If $y\in\tilde{\Gamma}\cap\Gamma$, we have $|\tilde{h}_y-h_y|\geq(\frac{|x-z|}{2}-s_y)\frac{|\tilde{h}_z|}{|x-z|}\sim|E_{ww'xz}||x-y|$
		and therefore
		\begin{equation}\label{l32}
			\begin{split}
				\frac{1}{|h_y|^p|\tilde{h}_y|^q}&\lesssim|\tilde{h}_y-h_y|^{-\min\{p,q\}}\left(|h_y|^{-\max\{p,q\}}+|\tilde{h}_y|^{-\max\{p,q\}}\right)\\
				&\lesssim|E_{ww'xz}|^{-\min\{p,q\}}|x-y|^{-\min\{p,q\}}\left(|h_y|^{-\max\{p,q\}}+|\tilde{h}_y|^{-\max\{p,q\}}\right).
			\end{split}
		\end{equation}
		
		We first consider the integral over $\tilde{\Gamma}\cap\Gamma_+$, where $\min\{|x-y|,|y-z|\}=|x-y|$ and $|y-z|\sim|x-z|$. Then
		\begin{equation*}
			\begin{split}
				&\int_{\tilde{\Gamma}\cap\Gamma_+}\frac{\langle|x-y|^{-k_1}\rangle\langle|y-z|^{-l_1}\rangle\langle y-y_0\rangle^{-\beta-}}{\langle x-y\rangle^{k_2}\langle y-z\rangle^{l_2}|E_{xyyz}|^p|E_{ww'xy}|^q}\d y\\
				\lesssim&\int_{\tilde{\Gamma}\cap\Gamma_+}\frac{\langle y-y_0\rangle^{-\beta-}\langle|x-y|^{-k_1}\rangle\langle|y-z|^{-l_1}\rangle}{|x-y|^{-(p+q)}\langle x-y\rangle^{k_2}\langle y-z\rangle^{l_2}|h_y|^p|\tilde{h}_y|^q}\d y\\
				\lesssim&|E_{ww'xz}|^{-\min\{p,q\}}\int_{\tilde{\Gamma}\cap\Gamma_+}\frac{\langle y-y_0\rangle^{-\beta-}\langle|x-y|^{-k_1}\rangle\langle|y-z|^{-l_1}\rangle\left(|h_y|^{-\max\{p,q\}}+|\tilde{h}_y|^{-\max\{p,q\}}\right)}{|x-y|^{-\max\{p,q\}}\langle x-y\rangle^{k_2}\langle y-z\rangle^{l_2}}\d y.
			\end{split}
		\end{equation*}
		Since when $y\in\tilde{\Gamma}\cap\Gamma_+$, we have $|h_y|<\frac{|x-z|}{2}-s_y<\frac{|x-z|}{2}$, and $|\tilde{h}_y|<\tilde{s}_y<|x-z|$, as how we treat \eqref{l7} and \eqref{l8}, it follows that
		\begin{equation*}\label{l35}
			\begin{split}
				&\int_{\tilde{\Gamma}\cap\Gamma_+}\frac{\langle|x-y|^{-k_1}\rangle\langle|y-z|^{-l_1}\rangle\langle y-y_0\rangle^{-\beta-}}{\langle x-y\rangle^{k_2}\langle y-z\rangle^{l_2}|E_{xyyz}|^p|E_{ww'xy}|^q}\d y\\
				\lesssim&|E_{ww'xz}|^{-\min\{p,q\}}\langle|x-z|^{-(k_1+l_1-n)}\rangle\langle x-z\rangle^{-\min\{k_2,l_2,k_2+l_2+\beta-n,k_2+l_2-\max\{p,q\}\}}.
			\end{split}
		\end{equation*}
		
		Finally consider the integral over $\tilde{\Gamma}\cap\Gamma_-$, where $\min\{|x-y|,|y-z|\}=|y-z|\sim\frac{|x-z|}{2}+s_y$, $|x-y|\sim\tilde{s}_y\sim|x-z|$. We split $\tilde{\Gamma}\cap\Gamma_-$ into
		\begin{equation*}
			\begin{split}
				E=\left\{y\in\tilde{\Gamma}\cap\Gamma_-;~|\tilde{h}_y|\geq\mbox{$\frac12$}|\tilde{h}_z|\right\},\quad F=\left\{y\in\tilde{\Gamma}\cap\Gamma_-;~|\tilde{h}_y|<\mbox{$\frac12$}|\tilde{h}_z|\right\},
			\end{split}
		\end{equation*}
		Similar to \eqref{l23}, we immediately get
		\begin{equation*}\label{l36}
			\begin{split}
				&\int_E\frac{\langle|x-y|^{-k_1}\rangle\langle|y-z|^{-l_1}\rangle\langle y-y_0\rangle^{-\beta-}}{\langle x-y\rangle^{k_2}\langle y-z\rangle^{l_2}|E_{xyyz}|^p|E_{ww'xy}|^q}\d y\\
				\lesssim&|E_{ww'xz}|^{-q}\int_E\frac{\langle|x-y|^{-k_1}\rangle\langle|y-z|^{-l_1}\rangle\langle y-y_0\rangle^{-\beta-}}{\langle x-y\rangle^{k_2}\langle y-z\rangle^{l_2}|E_{xyyz}|^p}\d y\\
				\lesssim&\begin{cases}
					|E_{ww'xz}|^{-q}\langle|x-z|^{-\max\{0,k_1+l_1-n\}}\rangle\langle x-z\rangle^{-\min\{k_2,l_2,k_2+l_2+\beta-n,k_2+l_2-p\}},&k_1+l_1\neq n,\\
					|E_{ww'xz}|^{-q}\langle|x-z|^{0-}\rangle\langle x-z\rangle^{-\min\{k_2,l_2,k_2+l_2+\beta-n,k_2+l_2-p\}},&k_1+l_1=n.
				\end{cases}
			\end{split}
		\end{equation*}
		On the other hand, when $y\in F$, note that $|E_{ww'xz}|\sim\frac{|\tilde{h}_z|}{|x-z|}\sim\frac{|\tilde{h}_z|}{|x-y|}$, \eqref{l32} also says
		\begin{equation*}
			\frac{1}{|h_y|^p|\tilde{h}_y|^q}\lesssim|E_{ww'xz}|^{-q}\frac{|\tilde{h}_z|^{q-\min\{p,q\}}}{|x-y|^q}\left(|h_y|^{-\max\{p,q\}}+|\tilde{h}_y|^{-\max\{p,q\}}\right),
		\end{equation*}
		and thus
		\begin{equation}\label{l38}
			\begin{split}
				&\int_F\frac{\langle|x-y|^{-k_1}\rangle\langle|y-z|^{-l_1}\rangle\langle y-y_0\rangle^{-\beta-}}{\langle x-y\rangle^{k_2}\langle y-z\rangle^{l_2}|E_{xyyz}|^p|E_{ww'xy}|^q}\d y\\
				\lesssim&\int_F\frac{\langle|x-y|^{-k_1}\rangle\langle|y-z|^{-l_1}\rangle\langle y-y_0\rangle^{-\beta-}}{|x-y|^{-q}|y-z|^{-p}\langle x-y\rangle^{k_2}\langle y-z\rangle^{l_2}|h_y|^p|\tilde{h}_y|^q}\d y\\
				\lesssim&|E_{ww'xz}|^{-q}\int_F\frac{|\tilde{h}_z|^{q-\min\{p,q\}}\left(|h_y|^{-\max\{p,q\}}+|\tilde{h}_y|^{-\max\{p,q\}}\right)\langle|x-y|^{-k_1}\rangle\langle|y-z|^{-l_1}\rangle\langle y-y_0\rangle^{-\beta-}}{\langle x-y\rangle^{k_2}\langle y-z\rangle^{l_2}|y-z|^{-p}}\d y\\
				\lesssim&|E_{ww'xz}|^{-q}\int_F\frac{\langle|x-y|^{-k_1}\rangle\langle|y-z|^{-l_1}\rangle\langle y-y_0\rangle^{-\beta-}}{\langle x-y\rangle^{k_2}\langle y-z\rangle^{l_2}|y-z|^{-\max\{p,q\}}|h_y|^{\max\{p,q\}}}\d y\\
				&+|E_{ww'xz}|^{-q}\int_F\frac{|\tilde{h}_z|^{q-\min\{p,q\}}\langle|x-y|^{-k_1}\rangle\langle|y-z|^{-l_1}\rangle\langle y-y_0\rangle^{-\beta-}}{\langle x-y\rangle^{k_2}\langle y-z\rangle^{l_2}|y-z|^{-p}|\tilde{h}_y|^{\max\{p,q\}}}\d y,
			\end{split}
		\end{equation}
		where we have used $|\tilde{h}_z|\sim|\tilde{h}_y-\tilde{h}_z|\leq|y-z|$ for the first term on the RHS of \eqref{l38}. Since $|h_y|<\frac{|x-z|}{2}+s_y<\frac{|x-z|}{2}$ for $y\in F$, as how we treat \eqref{l7} and \eqref{l8}, it again follows that
		\begin{equation*}\label{l39}
			\begin{split}
				&|E_{ww'xz}|^{-q}\int_F\frac{\langle|x-y|^{-k_1}\rangle\langle|y-z|^{-l_1}\rangle\langle y-y_0\rangle^{-\beta-}}{\langle x-y\rangle^{k_2}\langle y-z\rangle^{l_2}|E_{xyyz}|^p|E_{ww'xy}|^q}\d y\\
				\lesssim&|E_{ww'xz}|^{-q}\langle|x-z|^{-(k_1+l_1-n)}\rangle\langle x-z\rangle^{-\min\{k_2,l_2,k_2+l_2+\beta-n,k_2+l_2-\max\{p,q\}\}}.
			\end{split}
		\end{equation*}
		
		Estimating the last term on the RHS of \eqref{l38} is not parallel, but more like mimicking the treatment for \eqref{ll24}. When $|x-z|\lesssim1$, which implies $|x-y|\sim\tilde{s}_y\lesssim1$, $|y-z|\leq\frac{|x-z|}{2}\lesssim1$, $\langle|x-y|^{-1}\rangle\sim|x-z|^{-1}$ and $\langle x-y\rangle\sim\langle y-z\rangle\sim1$ when $y\in F$, we first have
		\begin{equation}\label{l40}
			\begin{split}
				&|E_{ww'xz}|^{-q}\int_F\frac{|\tilde{h}_z|^{q-\min\{p,q\}}\langle|x-y|^{-k_1}\rangle\langle|y-z|^{-l_1}\rangle\langle y-y_0\rangle^{-\beta-}}{\langle x-y\rangle^{k_2}\langle y-z\rangle^{l_2}|y-z|^{-p}|\tilde{h}_y|^{\max\{p,q\}}}\d y\\
				\lesssim&|E_{ww'xz}|^{-q}|x-z|^{-k_1}\int_F\frac{\langle y-y_0\rangle^{-\beta-}|\tilde{h}_z|^{q-\min\{p,q\}}}{|y-z|^{l_1-p}|\tilde{h}_y|^{\max\{p,q\}}}\d y,
			\end{split}
		\end{equation}
		and it is easy to check when $y\in F$ that
		\begin{equation}\label{l41}
			\frac{1}{|y-z|^{l_1-p}}\lesssim\begin{cases}
				|x-z|^{p-l_1},\quad&0\leq l_1\leq p,\\
				|\tilde{h}_z|^{p-l_1},&p<l_1\leq n-1,\\
				\frac{1}{|\tilde{s}_y-\tilde{s}_z|^{l_1-(n-1)}|\tilde{h}_z|^{n-1-p}},&p<n-1<l_1<n,
			\end{cases}
		\end{equation}
		so putting \eqref{l41} into \eqref{l40} and using $|\tilde{h}_z|\leq|x-z|$, one easily obtains
		\begin{equation*}\label{l42}
			\begin{split}
				&|E_{ww'xz}|^{-q}\int_F\frac{|\tilde{h}_z|^{q-\min\{p,q\}}\langle|x-y|^{-k_1}\rangle\langle|y-z|^{-l_1}\rangle\langle y-y_0\rangle^{-\beta-}}{\langle x-y\rangle^{k_2}\langle y-z\rangle^{l_2}|y-z|^{-p}|\tilde{h}_y|^{\max\{p,q\}}}\d y\\
				\lesssim&|E_{ww'xz}|^{-q}|x-z|^{-k_1}\int_{\{\tilde{s}_y\sim|x-z|\}\cap\{|\tilde{s}_y-\tilde{s}_z|\lesssim1\}}\left(\int_{|\tilde{h}_y|<\frac12|\tilde{h}_z|}\frac{|\tilde{h}_z|^{q-\min\{p,q\}}\d\tilde{h}_y}{|y-z|^{l_1-p}|\tilde{h}_y|^{\max\{p,q\}}}\right)\d\tilde{s}_y\\
				\lesssim&\begin{cases}
					|E_{ww'xz}|^{-q}|x-z|^{-(k_1-\max\{0,p-l_1\}-1)}|\tilde{h}_z|^{n-1-\max\{p,l_1\}},\quad&0\leq l_1\leq n-1,\\
					|E_{ww'xz}|^{-q}|x-z|^{-(k_1+l_1-n)},&n-1<l_1<n,
				\end{cases}\\
				\lesssim&|E_{ww'xz}|^{-q}|x-z|^{-(k_1+l_1-n)}.
			\end{split}
		\end{equation*}
		
		When $|x-z|\gtrsim1$, similar to the treatment for $I_D$ above by considering $I_{D_1}$ and $I_{D_2}$, we first observe that if $y\in F\cap\{y\in\mathbb{R}^n;\,|y-z|<1\}$, it follows that $\langle|x-y|^{-1}\rangle\sim\langle y-z\rangle\sim1$, $\langle x-y\rangle\sim|x-z|$, $|\tilde{h}_z|\lesssim|\tilde{h}_y-\tilde{h}_z|<1$, and
		\begin{equation*}
			\frac{1}{|y-z|^{l_1-p}}\lesssim\begin{cases}
				|\tilde{h}_z|^{\min\{0,p-l_1\}},&0\leq l_1\leq n-1,\\
				\frac{1}{|\tilde{s}_y-\tilde{s}_z|^{l_1-(n-1)}|\tilde{h}_z|^{n-1-p}},&n-1<l_1<n,
			\end{cases}
		\end{equation*}
		so we may first deduce
		\begin{equation*}
			\begin{split}
				&|E_{ww'xz}|^{-q}\int_{F\cap\{|y-z|<1\}}\frac{|\tilde{h}_z|^{q-\min\{p,q\}}\langle|x-y|^{-k_1}\rangle\langle|y-z|^{-l_1}\rangle\langle y-y_0\rangle^{-\beta-}}{\langle x-y\rangle^{k_2}\langle y-z\rangle^{l_2}|y-z|^{-p}|\tilde{h}_y|^{\max\{p,q\}}}\d y\\
				\lesssim&|E_{ww'xz}|^{-q}|x-z|^{-k_2}\int_{|\tilde{s}_y-\tilde{s}_z|<1}\mathbbm{1}_{\{|\tilde{h}_z|\lesssim1\}}\left(\int_{|\tilde{h}_y|<\frac12|\tilde{h}_z|}\frac{|\tilde{h}_z|^{q-\min\{p,q\}}\d\tilde{h}_y}{|y-z|^{l_1-p}|\tilde{h}_y|^{\max\{p,q\}}}\right)\d\tilde{s}_y\\
				\lesssim&|E_{ww'xz}|^{-q}|x-z|^{-k_2}.
			\end{split}
		\end{equation*}
		Now we are only left to consider the integration over $F\cap\{y\in\mathbb{R}^n;\,|y-z|<1\}$ when $|x-z|\gtrsim1$, where $\langle|x-y|^{-1}\rangle\sim\langle|y-z|^{-1}\rangle\sim1$. Almost parallel to the discussion for $I_{D_2}$ in \eqref{ID2}, we have
		\begin{equation*}
			\begin{split}
				&|E_{ww'xz}|^{-q}\int_{F\cap\{|y-z|\geq1\}}\frac{|\tilde{h}_z|^{q-\min\{p,q\}}\langle|x-y|^{-k_1}\rangle\langle|y-z|^{-l_1}\rangle\langle y-y_0\rangle^{-\beta-}}{\langle x-y\rangle^{k_2}\langle y-z\rangle^{l_2}|y-z|^{-p}|\tilde{h}_y|^{\max\{p,q\}}}\d y\\
				\lesssim&|E_{ww'xz}|^{-q}|x-z|^{-(k_2+l_2+\beta-n)}\int_{F\cap\{|y-z|\geq1\}}\frac{\langle\tilde{s}_y-\tilde{s}_{y_0}\rangle^{-\beta-}|x-z|^{l_2+\beta-n}|\tilde{h}_z|^{n-1-p}}{|y-z|^{l_2-p}|\tilde{h}_z|^{n-1-\max\{p,q\}}|\tilde{h}_y|^{\max\{p,q\}}}\d y\\
				\lesssim&|E_{ww'xz}|^{-q}|x-z|^{-(k_2+l_2+\beta-n)},
			\end{split}
		\end{equation*}
		where the minor difference here is that we have used $|\tilde{h}_z|\lesssim|\tilde{h}_y-\tilde{h}_z|\lesssim|y-z|\lesssim|x-z|\sim\tilde{s}_y$ and $p<n-1$ to deduce for $y\in F\cap\{|y-z|\geq1\}$ that
		\begin{equation*}
			\frac{|x-z|^{l_2+\beta-n}|\tilde{h}_z|^{n-1-p}}{|y-z|^{l_2-p}}\lesssim\frac{|x-z|^{l_2+\beta-n}|y-z|^{n-1}}{|y-z|^{l_2}}\lesssim\frac{|x-z|^{l_2+\beta-1}}{|y-z|^{l_2}}\lesssim\langle\tilde{s}_y-\tilde{s}_z\rangle^{-l_2}\langle\tilde{s}_y\rangle^{-(1-l_2-\beta)}.
		\end{equation*}
		
		Now the estimate for the integral over $\tilde{\Gamma}\cap\Gamma$ has been shown, and the proof is complete.
	\end{proof}
	
	\section{The proofs of Lemma \ref{lemma4.6} and Lemma \ref{lemma4.9} }\label{app-002}
	
	\subsection{The proof of Lemma \ref{lemma4.6}}\
	
	By \eqref{equ4.17}, we have
	$$|\partial_{\lambda}^{s_{i}}R_{0}^{\pm}(\lambda^{2m})(x)|\lesssim |x|^{2m-n}+|x|^{-\frac{n-1}{2}+s_i},\qquad 0\le s_i\le \mbox{$\frac{n+1}{2}$}.$$
	On the other hand, a direct computation yields 
	\begin{equation*}\label{equ4.1.2}
		\begin{array}{ll}
			\left|\partial_{\lambda}^{l}\left( e^{ \i s \lambda_{k}|x-y|\mp  \i s \lambda|x|}\right)\right|
			=\left|\partial_{\lambda}^{l}\left( e^{ \i s \lambda_{k}(|x-y|-|x|)} e^{ \i s(\lambda_k\mp\lambda)|x|}\right)\right| \lesssim_{l}\lambda^{-l}\langle y\rangle^{l}.
		\end{array}
	\end{equation*}
	If $l>[\frac{n}{2m}]+2$, then it follows that
	\begin{equation}\label{eq4.5.1.1}
		\begin{split}
			&\left|v(y) \left(\prod_{j=0}^{l-1}(R_{0}^{\pm,(s_{j})}(\lambda^{2m})V)\partial_\lambda^{s_l}\left(|x-\cdot|^{-\tau}e^{\mp \i\lambda s|x|}e^{\i\lambda_{k}s|x-y|}  \right)\right)(y)\right| \\
			\lesssim&
			\lambda^{-s_l}\left| \int_{\R^{nl}}\langle y\rangle^{-\frac{\beta}{2}}\prod_{i=0}^{l-1}\left| \partial_{\lambda}^{s_{i}}R_{0}^{\pm}(\lambda^{2m})( z_{i}-z_{i+1})V(z_{i+1})\right|
			\langle z_l\rangle^{s_l}|z_l-x|^{-\tau}dz_{1}\cdots dz_{l}\right| \\
			\lesssim&
			\lambda^{-s_l}\int_{\R^{nl}}\langle y\rangle^{-\frac{\beta}{2}-s_0}\prod_{i=0}^{l-1}\left( (|z_{i}-z_{i+1}|^{-\frac{n-1}{2}}+|z_{i}-z_{i+1}|^{2m-n})\langle z_{i+1} \rangle^{-\beta+s_{i}+s_{i+1}}\right)
			|z_l-x|^{-\tau}\d z_{1}\cdots \d z_{l}\\
			\lesssim&
			\lambda^{-s_l}\langle y\rangle^{-\frac{\beta}{2}+s_0}\int_{\R^{n}}|y-z|^{-\frac{n-1}{2}}
			\langle z\rangle^{-\beta+s_{l-1}+s_l}|z-x|^{-\tau}\d z,
		\end{split}
	\end{equation}
	where $z_{0}=y$, $s_0+\cdots+s_l\le \frac{n+1}{2}$ and  the last inequality follows by repeatedly using the following estimate
	\begin{equation*}
		\begin{aligned}
			\int_{\R^{n}}\big(|x-z|^{-\max\{n-2mj,\, \frac{n-1}{2}\}}&+|x-z|^{-\frac{n-1}{2}}\big)\langle z\rangle^{-\frac{n+3}{2}}\big(|x-z|^{-\frac{n-1}{2}}+|z-y|^{2m-n}\big)dz\\
			\lesssim &|x-z|^{-\max\{n-2m(j+1),\, \frac{n-1}{2}\}}+|x-z|^{-\frac{n-1}{2}}, \quad \mbox{if}\,\,\, 2mj<n,
		\end{aligned}
	\end{equation*}
	which, in turn, follows from Lemma \ref{lmEd}, the fact that  $\beta- s_i-s_{i+1}> \frac{n+3}{2}$, and the inequality
	$$ \langle z_i\rangle^{-s_i} \langle z_{i+1}\rangle^{-s_{i}}\left(|z_{i}-z_{i+1}|^{-\frac{n-1}{2}+s_i}+|z_{i}-z_{i+1}|^{2m-n}\right)\leq |z_i-z_{i+1}|^{-\frac{n-1}{2}}+|z_i-z_{i+1}|^{2m-n}.$$
	Therefore, the Minkowski's inequality implies that the $L^2$ norm of the right hand side of \eqref{eq4.5.1.1} is controlled by 
	\begin{equation*}
		\begin{aligned}
			&\int_{\R^{n}}\left\|\langle y\rangle^{-\frac{\beta}{2}+s_0}|y-z|^{-\frac{n-1}{2}} \right\|_{L^2_y} \langle z\rangle^{-\beta+s_{l-1}+s_l}|z-x|^{-\tau} \d z\\
			\lesssim&\int_{\R^{n}}\langle z\rangle^{-\beta+s_{l-1}+s_l+\max\{-\frac{\beta}{2}+s_0, -\frac{n-1}{2}\}}|z-x|^{-\tau} \d z\lesssim \langle x\rangle^{-\tau},
		\end{aligned}
	\end{equation*}
	where we have used Lemma \ref{lem3.10}, as well as the facts $\frac{\beta}{2}-s_0>\frac12$ and $-\beta+s_{l-1}+s_l+\max\{\frac{1-\beta}{2}+s_0, -\frac{n-1}{2}\}<-n$ when $\beta>n+2$.
	Therefore \eqref{eq4.53.1} follows.
	
	In order to prove \eqref{eq4.53.2}, first note that \eqref{eq2.10} implies
	\begin{equation*}
		\begin{aligned}
			\left|\partial_{\lambda}^{s_{j}}\left(R_{0}^{+}(\lambda^{2m})(x-y)-R_{0}^{-}(\lambda^{2m})(x-y)\right)\right| \lesssim\sum_{s_{j,1}+s_{j,2}=s_j}\lambda^{n-2m-s_{j,1}}|x-y|^{s_{j,2}}
			\lesssim\lambda^{n-2m-s_{j}},
		\end{aligned}
	\end{equation*}
	provided  $\lambda|x-y|\le 1$.
	On the other hand, by  \eqref{eq2.8}, we derive  when $\lambda|x-y|\ge 1$ and $0<\lambda<1$ that
	\begin{equation*}
		\begin{aligned}
			\left| \langle y\rangle^{-1}\langle x\rangle^{-1}\partial_{\lambda}^{s_{j}}R_{0}^{\pm}(\lambda^{2m})(x-y)\right| &\lesssim\sum_{0\le j\le{\frac{n-3}{2}}}\sum_{s_{j,1}+s_{j,2}=s_j}{\lambda^{j+2-2m-s_{j,1}}|x-y|^{-(n-2-j)+s_{j,2}} \langle y\rangle^{-1}\langle x\rangle^{-1}}\\
			&\lesssim\lambda^{n-2m-s_{j}}.
		\end{aligned}
	\end{equation*}
	Combining the above two inequalities, we obtain
	\begin{equation*}
		\left| \langle y\rangle^{-1}\langle x\rangle^{-1}\partial_{\lambda}^{s_{j}}\left(R_{0}^{+}(\lambda^{2m})(x-y)-R_{0}^{-}(\lambda^{2m})(x-y)\right)\right|\lesssim\lambda^{n-2m-s_{j}},\qquad 0<\lambda<1.
	\end{equation*}
	Thus the LHS of \eqref{eq4.53.2} is bounded by
	\begin{equation*}
		\begin{aligned}
			&\lambda^{n-2m-s_{j}}\left\|\langle z_0\rangle^{-\frac{\beta}{2}-s_0}\int_{\mathbb{R}^{nl}} \langle z_{j}\rangle^{-s_{j}+1} \langle z_{j+1}\rangle^{-\beta+s_{j}+1}\right. \\
			&\times\left. \prod_{i\ne j, 0\le i\le l-1}\left((|z_{i}-z_{i+1}|^{-\frac{n-1}{2}}+|z_{i}-z_{i+1}|^{2m-n})\langle z_{i+1} \rangle^{-\beta+s_{i}+s_{i+1}} \right)|z_l-x|^{-\tau} \d z_1\cdots \d z_l \right\|_{L_{z_0}^2}\\
			\lesssim&\lambda^{n-2m-s_{j}},
		\end{aligned}
	\end{equation*}
	where we have used $\beta-s_{j-1}-1>\frac{n+1}{2}$ and $\beta-s_i-s_{i+1}> \frac{n+1}{2}$, and \eqref{eq4.53.2} is now proved.	
	
	\subsection{The proof of Lemma \ref{lemma4.9}}\
	
	Recall that by \eqref{eq2.10} and  \eqref{equ4.17}, one has
	\begin{align}\label{equ.4.2.r0}
		R_{0}^{+}(\lambda^{2m})(x-y)-R_{0}^{-}(\lambda^{2m})(x-y)=\int_0^1{\tilde{r}^{+}_0(\lambda, s, x-y)-\tilde{r}^{-}_0(\lambda, s, x-y)\d s},
	\end{align}
	and
	\begin{equation}\label{equ.4.2.r1}
		R_{0}^{\pm}(\lambda^{2m})(x-y)=\int_0^1{ \tilde{r}^{\pm}_{2m-n, 1}(\lambda, s, x-y)\d s}+ r^{\pm}_{2m-n, 0}(\lambda,  x-y),
	\end{equation}
	where  
	\begin{equation*}
		\tilde{r}^{\pm}_0(\lambda, s, x-y)=\sum_{k\in I^{\pm}} \sum_{j=0}^{\frac{n-3}{2}} C_{j,0} \lambda_{k}^{n-2 m} e^{\i s \lambda_{k}|x-y|}(1-s)^{n-j-3},
	\end{equation*}
	\begin{equation*}
		\tilde{r}^{\pm}_{2m-n, 1}(\lambda, s, x-y)=\sum_{k\in I^{\pm}}|x-y|^{2m-n} \left(\sum_{l=0}^{2m-3} C_{l,2m-n} e^{\i s \lambda_{k}|x-y|}(1-s)^{2m-3-l}\right),
	\end{equation*}
	and
	\begin{equation*}
		r^{\pm}_{2m-n, 0}(\lambda,  x-y)=\sum\limits_{k\in I^{\pm}}\sum\limits_{j=2m-2}^{\frac{n-3}{2}}D_{j}\lambda_k^{j+2-2m}|x-y|^{j+2-n}e^{\i\lambda_{k}|x-y|}.
	\end{equation*}
	
	We first prove \eqref{eq4.62} and \eqref{eq4.62.1}. Set
	\begin{equation*}\label{equ4.8.2}
		\begin{cases}
			\omega_{1,1}^{\pm}(\lambda,s,y,x)
			=e^{\mp \i\lambda s|x|}v(y)\left((R_{0}^{\pm}(\lambda^{2m})V)^{l}\tilde{r}^{\pm}_{2m-n, 1}(\lambda, s, x-\cdot)\right)(y),  \\[0.25cm]
			\omega_{1,0}^{\pm}(\lambda,s,y,x)=e^{\mp \i\lambda |x|}v(y)\left((R_{0}^{\pm}(\lambda^{2m})V)^{l}r^{\pm}_{2m-n, 0}(\lambda, x-\cdot)\right)(y).
		\end{cases}
	\end{equation*}
	\eqref{eq4.62} follows from \eqref{equ.4.2.r1} and the fact that $\omega_{1,0}^{\pm}(\lambda,s,y,x)$ is actually independent of $s$. To obtain \eqref{eq4.62.1}, a direct computation yields that for each $0\le \gamma\le\frac{n+1}{2}$,
		\begin{equation*}\small
			\begin{aligned}
				&\left\| \partial_{\lambda}^{\gamma}\omega_{1,1}^{\pm}(\lambda,s,\cdot,x) \right\|_{L^2}\\
				 \lesssim&\sum_{s_0+\cdots+s_l=\gamma}\sum_{k\in I^{\pm}}\sum_{j=0}^{2m-3}
				\left\|v(y) \left(\prod_{j=0}^{l-1}(R_{0}^{\pm,(s_{j})}(\lambda^{2m})V)\partial_\lambda^{s_l}\left(|x-\cdot|^{2m-n}e^{\mp \i\lambda s|x|}e^{\i\lambda_{k}s|x-\cdot|}  \right)\right)(y)\right\|_{L^2_y}\\
				\lesssim&\lambda^{-\gamma}\langle x\rangle^{-\frac{n-1}{2}},
			\end{aligned}
		\end{equation*}
	where we have used  \eqref{eq4.53.1} and the fact $\frac{n-1}{2}\le n-2m$ in  the last inequality.
	Similarly,
		\begin{equation*}\small
			\begin{aligned}
				&\left\| \partial_{\lambda}^{\gamma}\omega_{1,0}^{\pm}(\lambda,s,\cdot,x) \right\|_{L^2} \\
				\lesssim&\sum_{s_0+\cdots+s_l=\gamma}\sum_{k\in I^{\pm}}\sum_{j=2m-2}^{\frac{n-3}{2}}
				\left\|v(y) \left(\prod_{j=0}^{l-1}(R_{0}^{\pm,(s_{j})}(\lambda^{2m})V)\partial_\lambda^{s_l}\left(\lambda_k^{j+2-2m}|x-\cdot|^{j+2-n}e^{\mp \i\lambda s|x|}e^{\i\lambda_{k}s|x-\cdot|}  \right)\right)(y)\right\|_{L^2_y}\\
				\lesssim&\lambda^{-\gamma}\langle x\rangle^{-\frac{n-1}{2}}.
			\end{aligned}
		\end{equation*}
	These two estimates give \eqref{eq4.62.1} immediately.
	
	We next prove \eqref{eq4.63}-\eqref{eq4.63.1}. Note that
	\begin{equation}\label{equ4.65}
		\begin{aligned}
			&v(y)\left((R_{0}^{+}(\lambda^{2m})V)^{l}R_{0}^{+}(\lambda^{2m})(x-y)-(R_{0}^{-}(\lambda^{2m})V)^{l}R_{0}^{-}(\lambda^{2m})(x-y)\right) \\
			=&v(y)\sum_{j=0}^{l-1}(R_{0}^{-}(\lambda^{2m})V)^{j}(R_{0}^{+}(\lambda^{2m})-R_{0}^{-}(\lambda^{2m}))V(R_{0}^{+}(\lambda^{2m})V)^{l-j-1}R_{0}^{+}(\lambda^{2m})(x-y)\\
			&+v(y)R_{0}^{-}(\lambda^{2m})V)^{l}\left(R_{0}^{+}(\lambda^{2m})(x-y)-R_{0}^{-}(\lambda^{2m})(x-y)\right).
		\end{aligned}
	\end{equation}
	We take $ \phi\in C_c^\infty(\mathbb{R})$ with $ \phi(t)=1$ ($|t|\leq1$) and $ \phi(t)=0$ ($|t|\geq2$), to define
		\begin{align*}\small
				\omega_{2,1}^{-}(\lambda,s,y,x)=&\phi(\lambda\langle x\rangle)e^{\i\lambda s|x|}v(y)\left((R_{0}^{-}(\lambda^{2m})V)^{l}
			\tilde{r}^{-}_0(\lambda, s, x-\cdot)\right)(y)\\
			&+(1-\phi(\lambda\langle x\rangle))e^{\i\lambda s|x|}v(y)\left((R_{0}^{-}(\lambda^{2m})V)^{l}\tilde{r}^{-}_{2m-n, 1}(\lambda, s, x-\cdot)\right)(y),
		\end{align*}
		\begin{align*}\scriptsize
			&\omega_{2,1}^{+}(\lambda,s,y,x)\\
			=&\phi(\lambda\langle x\rangle)e^{-\i\lambda s|x|}v(y)\left((R_{0}^{-}(\lambda^{2m})V)^{l} \tilde{r}^{+}_0(\lambda, s, x-\cdot)\right)(y)\\
			&+(1-\phi(\lambda\langle x\rangle))e^{-\i\lambda s|x|}v(y)\left((R_{0}^{-}(\lambda^{2m})V)^{l}\tilde{r}^{+}_{2m-n, 1}(\lambda, s, x-\cdot)\right)(y)+e^{-\i\lambda s|x|}v(y)\\
			&\times\sum_{r=0}^{l-1}\left((R_{0}^{-}(\lambda^{2m})V)^{r}(R_{0}^{+}(\lambda^{2m})-R_{0}^{-}(\lambda^{2m}))V
			(R_{0}^{+}(\lambda^{2m})V)^{l-r-1}\tilde{r}^{+}_{2m-n, 1}(\lambda, s, x-\cdot)\right)(y),
		\end{align*}
		\begin{align*}\small
			\omega_{2,0}^{-}(\lambda,s,y,x)=(1-\phi(\lambda\langle x\rangle))e^{\i\lambda |x|}v(y)\left((R_{0}^{-}(\lambda^{2m})V)^{l}
			r^{-}_{2m-n, 0}(\lambda, x-\cdot)\right)(y),
		\end{align*}
		and
		\begin{align*}\scriptsize
			&\omega_{2,0}^{+}(\lambda,s,y,x)\\
			=&(1-\phi(\lambda\langle x\rangle))e^{-\i\lambda |x|}v(y)\left((R_{0}^{-}(\lambda^{2m})V)^{l}
			r^{+}_{2m-n, 0}(\lambda, x-\cdot)\right)(y)+e^{-\i\lambda |x|}v(y)\\
			&\times\sum_{j=0}^{l-1}\left((R_{0}^{-}(\lambda^{2m})V)^{j}(R_{0}^{+}(\lambda^{2m})-R_{0}^{-}(\lambda^{2m}))V
			(R_{0}^{+}(\lambda^{2m})V)^{l-j-1}r^{+}_{2m-n, 0}(\lambda, x-\cdot)\right)(y).
		\end{align*}
		Now \eqref{eq4.63} follows if we combine \eqref{equ.4.2.r0},  \eqref{equ.4.2.r1} and \eqref{equ4.65}, while similar to the proof of \eqref{eq4.62.1}, one checks \eqref{eq4.63.2} and \eqref{eq4.63.1} by using \eqref{eq4.53.1} and \eqref{eq4.53.2}. Therefore the proof is complete.

\end{appendix}

\section*{Acknowledgements}
T. Huang was supported by National Key R\&D Program of China under the grant 2023YFA1010300 and the National Natural Science Foundation of China under the grants 12101621 and 12371244.
S. Huang was supported by the National Natural Science Foundation of China under the grants 12171178 and 12171442.
Q. Zheng was supported by the National Natural Science Foundation of China under the grant 12171178.

%
%



\begin{thebibliography}{99}




\bibitem{AH} S. Agmon and L. H\"{o}rmander, Asymptotic properties of solutions of differential equations with simple characteristics. \textit{J. Analyse Math.} \textbf{30} (1976), 1-38.




\bibitem{CHHZ} H. Cheng, S. Huang, T. Huang and Q. Zheng, Pointwise estimates for the fundamental solutions of higher order Schr\"{o}dinger equations in odd dimensions \uppercase\expandafter{\romannumeral1}: low dimensional case. \textit{arXiv}: 2401.04969v10.









 \bibitem{EGG23} M. B. Erdo\v{g}an, M. Goldberg and W. R. Green, Counterexamples to $L^p$-boundedness of wave operators for classical and higher order Schr\"{o}dinger operators. \textit{J. Funct. Anal.} \textbf{285} (2023), no. 5, 110008.

 \bibitem{EGG231}  M. B. Erdo\v{g}an, M. Goldberg and W. R. Green, Dispersive estimates for higher order Schr\"{o}dinger operators with scaling-critical potentials.  \textit{arXiv}: 2308.11745.

\bibitem{EG10}  M. B. Erdo\v{g}an and W. R. Green, Dispersive estimates for the Schr\"{o}dinger equation for $C^{\frac{n-3}{2}}$ potentials in odd dimensions. \textit{Int. Math. Res. Not. IMRN} \textbf{13} (2010), 2532-2565.

\bibitem{EG22} M. B. Erdo\v{g}an and W. R. Green, The $L^p$-continuity of wave operators for higher order Schr\"{o}dinger operators.  \textit{Adv. Math.} \textbf{404} (2022), Paper No. 108450, 41pp.

\bibitem{EG23} M. B. Erdo\v{g}an and W. R. Green, A note on endpoint $L^p$-continuity of wave operators for classical and higher order Schr\"{o}dinger operators. \textit{J. Differential Equations} \textbf{355} (2023), 144-161.

\bibitem{EGL}  M. B. Erdo\v{g}an, W. R. Green and  K. LaMaster, $L^p$-continuity of wave operators for higher order Schr\"{o}dinger operators with threshold eigenvalues in high dimensions. 	\emph{arXiv:} 2407.07069.

\bibitem{EGT}  M. B. Erdo\v{g}an, W. R. Green and  E. Toprak, On the fourth order Schr\"{o}dinger equation in three dimensions: dispersive estimates and zero energy resonances. \textit{J. Differential Equations} \textbf{271} (2021),  152-185.














\bibitem{FSY} H. Feng, A. Soffer and X. Yao, Decay estimates and Strichartz estimates of fourth-order Schr\"{o}dinger operator. \textit{J. Funct. Anal.} \textbf{274} (2018), no. 2,  605-658.

\bibitem{FSWY} H. Feng, A. Soffer, Z. Wu and X. Yao, Decay estimates for higher-order elliptic operators. \textit{Trans. Amer. Math. Soc.} \textbf{373} (2020), no. 4, 2805-2859.


%


\bibitem{GY} A.  Galtbayar and K. Yajima,  The $L^p$-boundedness of wave operators for fourth order Schr\"{o}dinger operators on $\mathbb{R}^4$. \emph{J. Spectr. Theory} \textbf{14} (2024), no. 1, 271-354.

\bibitem{GG15} M. Goldberg  and  W. R. Green, Dispersive estimates for higher dimensional Schr\"{o}dinger operators with threshold eigenvalues I: The odd dimensional case. \textit{J. Funct. Anal.} \textbf{269} (2015), 633-682.

\bibitem{GG21} M. Goldberg  and  W. R. Green, On the $L^p$ boundedness of the wave operators for fourth order Schr\"{o}dinger operators. \emph{Trans. Amer. Math. Soc.} \textbf{374} (2021), 4075-4092.



\bibitem{GV}  M. Goldberg and M. Visan, A counterexample to dispersive estimates for Schr\"{o}dinger operators in higher dimensions. \textit{Comm. Math. Phys.} \textbf{266} (2006), 211-238.

\bibitem{GT19} W. R. Green and  E. Toprak, On the Fourth order Schr\"{o}dinger equation in four dimensions: dispersive estimates and zero energy resonances. \emph{J. Differential Equations} \textbf{267} (2019), no. 3, 1899-1954.





%
%









%


\bibitem{HHZ} T. Huang, S. Huang and Q. Zheng, Inhomogeneous oscillatory integrals and global smoothing effects for dispersive equations. \textit{J. Differential Equations} \textbf{263} (2017), no. 12, 8606-8629.








\bibitem{JSS} J. L. Journ\'{e}, A. Soffer and C. D. Sogge, Decay estimates for Schr\"{o}dinger operators. \textit{Comm. Pure Appl. Math.} \textbf{44} (1991), no. 5, 573-604.








\bibitem{Mi} A. Miyachi, On some estimates for the wave equation in $L^p$ and $H^p$.  \textit{J. Fac. Sci. Univ. Tokyo} \textbf{27} (1980), 331-354.

\bibitem{MWY} H. Mizutani, Z. Wan and X. Yao, $L^p$-boundedness of wave operators for bi-Schr\"{o}dinger operators on the line.  \textit{Adv. Math.}	\textbf{451} (2024), 1089806.


\bibitem{RS} I. Rodnianski and W. Schlag, Time decay for solutions of Schr\"{o}dinger equations with rough and time-dependent potentials. \textit{Invent. Math.} \textbf{155} (2004), no. 3, 451-513.


\bibitem{Sch07}  W. Schlag, Dispersive estimates for Schr\"{o}dinger operators: a survey. Mathematical aspects of nonlinear dispersive equations. \textit{Ann. of Math. Stud.} \textbf{163} (2007), 255-285.  Princeton Univ. Press, Princeton, NJ.

\bibitem{Sch21} W. Schlag, On pointwise decay of waves. \textit{J. Math. Phys.} \textbf{62} (2021), Paper No. 061509, 27pp.


\bibitem{SWY} A. Soffer, Z. Wu and X. Yao, Decay estimates for bi-Schr\"{o}dinger operators in dimension one.  \textit{Ann. Henri Poincar\'{e}.} \textbf{23} (2022),  no. 8, 2683-2744.

\bibitem{Ya} K. Yajima, The $W^{k,p}$ continuity of wave operators for Schr\"{o}dinger operators. \textit{J. Math. Soc. Japan} \textbf{47} (1995), no. 3, 551-581.



\end{thebibliography}
\end{document}